\documentclass[11pt,twoside,a4paper]{report}
\usepackage{amsmath, amssymb}
\usepackage{shuffle}
\usepackage[latin1]{inputenc}
\usepackage{color}
\usepackage[T1]{fontenc}
\usepackage[english]{babel}
\newcommand{\tun}{\begin{picture}(5,0)(-2,-1)
\put(0,0){\circle*{2}}
\end{picture}}

\newcommand{\tdeux}{\begin{picture}(7,7)(0,-1)
\put(3,0){\circle*{2}}
\put(3,5){\circle*{2}}
\put(3,0){\line(0,1){5}}
\end{picture}}

\newcommand{\ttroisun}{\begin{picture}(15,12)(-5,-1)
\put(3,0){\circle*{2}}
\put(6,7){\circle*{2}}
\put(0,7){\circle*{2}}
\put(-0.65,0){$\vee$}
\end{picture}}

\newcommand{\ttroisdeux}{\begin{picture}(5,15)(-2,-1)
\put(0,0){\circle*{2}}
\put(0,5){\circle*{2}}
\put(0,10){\circle*{2}}
\put(0,0){\line(0,1){5}}
\put(0,5){\line(0,1){5}}
\end{picture}}

\newcommand{\tdun}[1]{\begin{picture}(10,5)(-2,-1)
\put(0,0){\circle*{2}}
\put(3,-2){\tiny #1}
\end{picture}}

\newcommand{\tddeux}[2]{\begin{picture}(12,5)(0,-1)
\put(3,0){\circle*{2}}
\put(3,5){\circle*{2}}
\put(3,0){\line(0,1){5}}
\put(6,-2){\tiny #1}
\put(6,3){\tiny #2}
\end{picture}}

\newcommand{\tdtroisun}[3]{\begin{picture}(20,12)(-5,-1)
\put(3,0){\circle*{2}}
\put(6,7){\circle*{2}}
\put(0,7){\circle*{2}}
\put(-0.65,0){$\vee$}
\put(5,-2){\tiny #1}
\put(9,5){\tiny #2}
\put(-5,5){\tiny #3}
\end{picture}}

\newcommand{\tdtroisdeux}[3]{\begin{picture}(12,15)(-2,-1)
\put(0,0){\circle*{2}}
\put(0,5){\circle*{2}}
\put(0,10){\circle*{2}}
\put(0,0){\line(0,1){5}}
\put(0,5){\line(0,1){5}}
\put(3,-2){\tiny #1}
\put(3,3){\tiny #2}
\put(3,9){\tiny #3}
\end{picture}}

\newcommand{\pquatrecinq}{\begin{picture}(15,9)(-5,-1)
\put(0,0){\circle*{2}}
\put(7,0){\circle*{2}}
\put(0,7){\circle*{2}}
\put(7,7){\circle*{2}}
\put(0,0){\line(0,1){7}}
\put(7,0){\line(0,1){7}}
\put(.5,1.5){$\scriptstyle \diagup$}
\end{picture}}
\newcommand{\pquatresix}{\begin{picture}(15,9)(-5,-1)
\put(0,0){\circle*{2}}
\put(7,0){\circle*{2}}
\put(0,7){\circle*{2}}
\put(7,7){\circle*{2}}
\put(0,0){\line(0,1){7}}
\put(7,0){\line(0,1){7}}
\put(0,1.5){$\scriptstyle \diagdown$}
\end{picture}}
\newcommand{\pquatresept}{\begin{picture}(15,9)(-5,-1)
\put(0,0){\circle*{2}}
\put(7,0){\circle*{2}}
\put(0,7){\circle*{2}}
\put(7,7){\circle*{2}}
\put(0,0){\line(0,1){7}}
\put(7,0){\line(0,1){7}}
\put(.5,1.5){$\scriptstyle \diagup$}
\put(0,1.5){$\scriptstyle \diagdown$}
\end{picture}}

\newcommand{\pdtroisun}[3]{\begin{picture}(23,12)(-7,-1)
\put(3,7){\circle*{2}}
\put(-0.65,0){$\wedge$}
\put(6,0){\circle*{2}}
\put(0,0){\circle*{2}}
\put(5,5){\tiny #1}
\put(-7,-2){\tiny #2}
\put(9,-2){\tiny #3}
\end{picture}}

\input{xy}
\xyoption{all}

\newtheorem{defi}{\indent Definition}
\newtheorem{lemma}[defi]{\indent Lemma}
\newtheorem{cor}[defi]{\indent Corollary}
\newtheorem{theo}[defi]{\indent Theorem}
\newtheorem{prop}[defi]{\indent Proposition}

\newenvironment{proof}{\textbf{Proof.}}{\hfill $\Box$}

\newcommand{\totimes}{\overline{\otimes}}
\newcommand{\tdelta}{\tilde{\Delta}}
\newcommand{\bfP}{\mathbf{P}}
\newcommand{\bfC}{\mathbf{C}}
\newcommand{\bfB}{\mathbf{B}}
\newcommand{\bfD}{\mathbf{D}}
\newcommand{\bfL}{\mathbf{L}}
\newcommand{\calG}{\mathcal{G}}
\newcommand{\calM}{\mathcal{M}}
\newcommand{\calB}{\mathcal{B}}
\newcommand{\calO}{\mathcal{O}}

\newcommand{\N}{\mathbb{N}}
\newcommand{\K}{\mathbb{K}}

\newcommand{\deuxas}{\mathbf{2As}}
\newcommand{\primdeuxas}{\mathbf{Prim2As}}
\newcommand{\binfini}{\mathbf{B_\infty}}
\renewcommand{\brace}{\mathbf{Brace}}
\newcommand{\ass}{\mathbf{As}}
\newcommand{\petitbinfini}{\mathbf{b_\infty}}
\newcommand{\ascom}{\mathbf{AsCom}}
\newcommand{\primascom}{\mathbf{PrimAsCom}}
\newcommand{\prelie}{\mathbf{PreLie}}
\newcommand{\com}{\mathbf{Com}}
\newcommand{\fg}{\mathcal{F}}
\newcommand{\bffg}{\mathbf{F}}
\newcommand{\wcfg}{\mathcal{N}c\mathcal{F}}
\newcommand{\bfwcfg}{\mathbf{NcF}}
\newcommand{\sfg}{\mathcal{SF}}
\newcommand{\bfsfg}{\mathbf{SF}}
\newcommand{\wcsfg}{\mathcal{N}c\mathcal{SF}}
\newcommand{\bfwcsfg}{\mathbf{NcSF}}
\newcommand{\gr}{\mathcal{G}}
\newcommand{\bfgr}{\mathbf{G}}
\newcommand{\wcgr}{\mathcal{N}c\mathcal{G}}
\newcommand{\bfwcgr}{\mathbf{NcG}}
\newcommand{\sgr}{\mathcal{SG}}
\newcommand{\bfsgr}{\mathbf{SG}}
\newcommand{\wcsgr}{\mathcal{N}c\mathcal{SG}}
\newcommand{\bfwcsgr}{\mathbf{NcSG}}
\newcommand{\qo}{q\mathcal{O}}
\newcommand{\bfqo}{\mathbf{qO}}
\newcommand{\od}{\mathcal{O}}
\newcommand{\bfod}{\mathbf{O}}
\newcommand{\bfcfg}{\mathbf{CF}}
\newcommand{\bfcwcfg}{\mathbf{CNcF}}
\newcommand{\bfcsfg}{\mathbf{CSF}}
\newcommand{\bfcwcsfg}{\mathbf{CNcSF}}
\newcommand{\bfcgr}{\mathbf{CG}}
\newcommand{\bfcwcgr}{\mathbf{CNcG}}
\newcommand{\bfcsgr}{\mathbf{CSG}}
\newcommand{\bfcwcsgr}{\mathbf{CNcSG}}
\newcommand{\bfcqo}{\mathbf{CqO}}
\newcommand{\bfcod}{\mathbf{CO}}

\newcommand{\rond}[1]{*++[o][F-]{#1}}

\begin{document}

\title{Algebraic structures associated to operads}
\date{}
\author{Lo\"\i c Foissy\\ \\
{\small \it Fédération de Recherche Mathématique du Nord Pas de Calais FR 2956}\\
{\small \it Laboratoire de Mathématiques Pures et Appliquées Joseph Liouville}\\
{\small \it Université du Littoral Côte d'Opale-Centre Universitaire de la Mi-Voix}\\ 
{\small \it 50, rue Ferdinand Buisson, CS 80699,  62228 Calais Cedex, France}\\ \\
{\small \it Email: foissy@lmpa.univ-littoral.fr}}

\maketitle

\begin{abstract}
We study different algebraic structures associated to an operad and their relations: to any operad $\mathbf{P}$ is attached a bialgebra,
the monoid of characters of this bialgebra, the underlying pre-Lie algebra and its enveloping algebra; all of them can be explicitely described
with the help of the operadic composition. non-commutative versions are also given. 

We denote by $\mathbf{b_\infty}$ the operad of $\mathbf{b_\infty}$ algebras, describing all Hopf algebra structures on a symmetric coalgebra.
If there exists an operad morphism from $\mathbf{b_\infty}$ to $\mathbf{P}$, a pair $(A,B)$ of cointeracting bialgebras is also constructed, that it to say:
$B$ is a bialgebra, and $A$ is a graded Hopf algebra in the category of $B$-comodules. Most examples of such pairs 
(on oriented graphs, posets$\ldots$) known in the literature are shown to be obtained from an operad; colored versions of these examples and
other ones, based on Feynman graphs, are introduced and compared. \\

\end{abstract}

\tableofcontents

\chapter*{Introduction}

Operads --a terminology due to May-- appear in the 70's \cite{MacLane,Stasheff,May,BV,Cohen},
to study loop spaces in algebraic topology; see \cite{LodayRenaissance} for a historical review. 
They are now widely used in various contexts; our point here is to study certain constructions associated to an operad
from a combinatorial Hopf-algebraic point of view. To a given operad $\bfP$, several objects are attached, such as monoids and groups, 
pre-Lie and brace algebras, bialgebras and Hopf algebras, or pairs of interacting or cointeracting Hopf algebras.\\

Let us precise the structures we obtain here. If $\bfP$ is an operad, then:
\begin{enumerate}
\item the space $\displaystyle \bfP=\bigoplus_{n\geq 0} \bfP(n)$ is a brace algebra \cite{Gerstenhaber,Ronco1} and a pre-Lie algebra 
\cite{Chapoton1,Chapoton2,Oudom}; in particular, its pre-Lie product is given by:
$$\forall p\in \bfP(n),\: q\in \bfP,\: p\bullet q=\sum_{i=1}^n p\circ_i q,$$
where $\circ_i$ is the $i$-th partial composition of the operad $\bfP$. Moreover, the quotient of coinvariant $coinv\bfP$ is a pre-Lie (but generally not a brace)
quotient of $\bfP$.
\item As observed in \cite{ChapotonLivernet}, this pre-Lie structure induces two monoid compositions $\lozenge$ and $\lozenge'$ 
on the space $\displaystyle \overline{\bfP}=\prod_{n\geq 0}\bfP(n)$: 
\begin{align*}
\forall p\in\bfP(n), q\in \overline{\bfP},\: p\lozenge' q&=q+p\circ(q,\ldots,q),\\
p\lozenge q&=q+\sum_{1\leq i_1<\ldots<i_k\leq n} x\circ_{i_1,\ldots,i_k}\circ (y,\ldots,y),
\end{align*}
where:
\begin{align*}
 x\circ_{i_1,\ldots,i_k}\circ (y,\ldots,y)&=x\circ \underbrace{(I,\ldots,y,\ldots,y,\ldots,I)}_{\mbox{\scriptsize{the $y$'s in position $i_1,\ldots,i_k$}}}.
\end{align*}
These two monoids are isomorphic; in both cases, $\overline{coinv\bfP}$ is a monoid quotient.
\item Using the Guin-Oudom construction \cite{Oudom}, the symmetric algebra $S(\bfP)$ inherits a product $*$, making it a Hopf algebra
with its usual product $\Delta$; similarly, using the brace structure, we obtain a dendriform Hopf algebra structure on $T(\bfP)$
\cite{Ronco1,Loday3,FoissyDend}, with the deconcatenation coproduct $\Delta_{dec}$, denoted by $\bfD_\bfP$. 
Moreover, considering the quotient $S(coinv\bfP)$ of $S(\bfP)$, we obtain a Hopf algebra $D_\bfP$, which is an explicit description 
of the enveloping algebra of the pre-Lie algebra $coinv\bfP$, and morphisms:
$$\xymatrix{D_\bfP&(S(\bfP),*,\Delta) \ar@{->>}[l] \ar@{^(->}[r]&\bfD_\bfP}$$ 
Here are examples of products $*$.
if $p_1\in \bfP(n_1)$, $p_2 \in \bfP(n_2)$, $q_1,q_2\in \bfP$:
\begin{itemize}
\item In $\bfD_\bfP$, 
\begin{align*}
p_1*q_1&=p_1q_1+q_1p_1+\sum_{1\leq i\leq n_1} p_1\circ_i q_1,\\
p_1*q_1q_2&=p_1q_1q_2+q_1p_1q_2+q_1q_2p_1+\sum_{1\leq i\leq n_1} (p_1\circ_i q_1)q_2\\
&+\sum_{1\leq i\leq n_1} q_1(p_1\circ_i q_2)+\sum_{1\leq i<j\leq n_1}p\circ_{i,j}(q_1,q_2),\\
p_1p_2*q_1&=p_1p_2q_1+p_1q_1p_2+q_1p_1p_2+\sum_{1\leq i\leq n_1}(p_1\circ_i q_1)p_2+\sum_ {1\leq i\leq n_2}p_1(p_2\circ_i q_1).
\end{align*}
\item In $S(\bfP)$,
\begin{align*}
p_1*q_1&=p_1q_1+\sum_{1\leq i\leq n_1} p_1\circ_i q_1,\\
p_1*q_1q_2&=p_1q_1q_2+\sum_{1\leq i\leq n_1} (p_1\circ_i q_1)q_2
+\sum_{1\leq i\leq n_1} q_1(p_1\circ_i q_2)+\sum_{1\leq i\neq j\leq n_1}p\circ_{i,j}(q_1,q_2),\\
p_1p_2*q_1&=p_1p_2q_1+\sum_{1\leq i\leq n_1}(p_1\circ_i q_1)p_2+\sum_ {1\leq i\leq n_2}p_1(p_2\circ_i q_1).
\end{align*}
\end{itemize}
These Hopf algebras are graded, but not connected, the elements of $\bfP(1)$, including the unit $I$ of $\bfP$,
being homogeneous of degree $0$ in them.
\item With the help of a technical condition called $0$-boundedness (definition \ref{defi16}), we can define a coproduct $\Delta_*$
on $T(\bfP^*)$, in duality with $*$, making $\bfD_\bfP^*=(T(\bfP^*),m_{conc},\Delta_*)$, where $m_{conc}$ is the concatenation product,
a bialgebra, generally not a Hopf algebra. By abelianization, one obtains bialgebras $(S(\bfP^*),m,\Delta_*)$ and $D_\bfP^*=(S((coinv\bfP)^*,m,\Delta_*)$,
respectively in duality with $S(\bfP)$ and $D_\bfP$, with morphisms:
$$\xymatrix{D_\bfP^*\ar@{^(->}[r]&S(\bfP^*)&\bfD_\bfP^*\ar@{->>}[l]}$$
Moreover, $(coinv\bfP,\lozenge)$ and $(\bfP,\lozenge)$ are the monoids of characters of respectively $D_\bfP^*$ and $S(\bfP^*)$.
\item Replacing $\bfP$ by its augmentation ideal $\bfP_+$, we obtain similar objects:
\begin{itemize}
\item Monoids:
$$\xymatrix{&(\overline{\bfP},\lozenge)\ar@{->>}[rd]&\\
(\overline{\bfP_+},\lozenge)\ar@{->>}[rd] \ar@{^(->}[ru]&&(\overline{coinv\bfP},\lozenge)=M_\bfP^D\\
&(\overline{coinv\bfP_+},\lozenge)=G_\bfP^D \ar@{^(->}[ru]&}$$
Moreover, $G_\bfP^D$ and $(\overline{\bfP_+},\lozenge)$ are groups.
\item Graded and connected Hopf algebras $B_\bfP$, $S(\bfP_+)$ and $\bfB_\bfP$, and their graded dual:
\begin{align*}
&\xymatrix{D_\bfP&S(\bfP) \ar@{^(->}[r] \ar@{->>}[l]&\bfD_\bfP\\
B_\bfP\ar@{^(->}[u] &S(\bfP_+)\ar@{^(->}[u] \ar@{^(->}[r] \ar@{->>}[l]&\bfB_\bfP\ar@{^(->}[u]}&
&\xymatrix{D_P^*\ar@{->>}[d]\ar@{^(->}[r]&S(\bfP^*)\ar@{->>}[d]&\bfD_\bfP^*\ar@{->>}[d]\ar@{->>}[l]\\
B^*_\bfP\ar@{^(->}[r]&S(\bfP^*_+)&\bfB^*_\bfP\ar@{->>}[l]}
\end{align*}
\end{itemize}
\end{enumerate}

Let us now assume that there exists an operadic morphism $\theta_\bfP$ from the operad $\petitbinfini$, 
ruling Hopf-algebraic structures on symmetric coalgebras, to $\bfP$. If this holds, for any vector space $V$, the free $\bfP$-algebra $F_\bfP(V)$
is a $\petitbinfini$ algebra, so the symmetric coalgebra $A_\bfP(V)$ is given a product $\star$, making it a graded, connected Hopf algebra. 
If $V$ is one-dimensional, we shall simply write $A_\bfP$. 
We prove that $A_\bfP$ is a Hopf algebra in the category of $D_\bfP$-modules, that it to say its unit, product, counit, coproduct, 
are all morphisms of $D_\bfP$-modules, for a certain action induced by the operadic composition.
Dually, the graded dual $A_\bfP^*$ of $A_\bfP$ is a Hopf algebra in the category of $D_\bfP^*$-comodules,
that it to say its unit, product, counit, coproduct, are all morphisms of $D_\bfP^*$-comodules. In terms of characters, this means that
the monoid ($M_\bfP^D,\lozenge)$ acts by group endomorphisms on the group of characters of $A_\bfP$.
We shall say that $(A_\bfP,D_\bfP)$ is a pair of interacting bialgebras, and $(A^*_\bfP,D^*_\bfP)$ is a pair of cointeracting bialgebras.
Note that all these results admit colored versions, tensoring $\bfP$ with the operad of morphisms from $V$ to $V^{\otimes n}$,
where $V$ is a vector space. \\
 
Several pairs of  cointeracting bialgebras are known in the literature, usually based on combinatorial objects,
such as trees, posets, graphs, etc;
 the proofs of the cointeraction (and of the fact that they are indeed bialgebras) 
usually need combinatorial operations such as cuts, extractions, contractions$\ldots$ We here obtains algebraic proofs:
\begin{itemize}
\item For the pair of cointeracting bialgebras of rooted trees of \cite{ManchonCalaque}, the operad to consider is $\prelie$,
seen as a quotient of $\petitbinfini$, through the canonical projection $\theta_\prelie$.
\item For the posets or quasi-posets of \cite{ManchonFoissyFauvet2},
use the operad on quasi-posets of \cite{ManchonFoissyFauvet}.
\item For the oriented graphs of \cite{Manchon2}, use an operad on graphs here introduced.
\item The example for Feynman graphs is here treated.
\end{itemize}
As all these constructions are functorial, one also obtains morphisms of pairs of (co)interacting bialgebras, relating quasi-posets, graphs,
Feynman graphs, and sub-families such as posets, graphs without cycles, simple graphs, etc.\\

This paper is organized as follows. In the first, short, chapter, we give reminders on operads. The examples of associative algebras $\ass$
and commutative, associative algebras $\com$ are introduced. The second chapter is devoted to the study of the two operads
$\binfini$ et $\petitbinfini$. For the reader's comfort, we give a complete proof of the bijection between, for any vector space $V$,
the set of $\binfini$ structures on $V$, and the set of products $*$ on $T(V)$ making $(T(V),*,\Delta_{dec})$ a Hopf algebra (theorem \ref{theo2}).
Two quotients of $\binfini$ are described, namely $\brace$ (in which case the product $*$ comes from a dendrifrom Hopf algebra structure),
and $\ass$ (obtaining quasi-shuffle Hopf algebras). The quite complicated operad $\binfini$ is shown to be isomorphic to a suboperad
of the simpler $\deuxas$, proving again Loday and Ronco's rigidity result \cite{Loday2,Loday}. A dual construction is also given, needing
the technical condition of $0$-boundedness (definition \ref{defi16}). Cocommutative versions of these results are given,
replacing $T(V)$ by $S(V)$, $\binfini$ by $\petitbinfini$, and $\deuxas$ by $\ascom$; the operad $\brace$ is replaced by $\prelie$,
obtaining a diagram of  operads: 
 $$
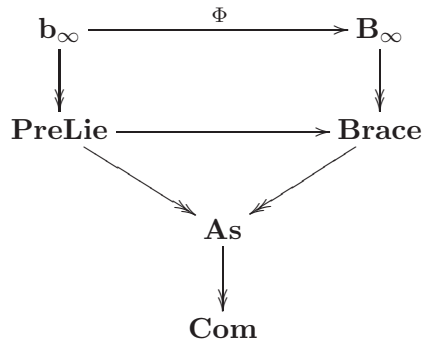
\xymatrix{\petitbinfini\ar[rr]^\Phi\ar@{->>}[d]&&\binfini \ar@{->>}[d]\\
\prelie \ar[rr] \ar@{->>}[rd]&&\brace\ar@{->>}[ld]\\
&\ass\ar@{->>}[d]&\\
&\com&}$$
It is also proved that, if $(V,\lfloor,\rfloor)$ is a $0$-bounded $\petitbinfini$, then its completion $\overline{V}$ is given a non-bilinear product $\lozenge$
defined by $x\lozenge y=\lfloor e^x,e^y\rfloor$, making it a monoid, isomorphic to the monoid of characters of the dual of $S(V)$ (theorem \ref{theo31}).

The next chapter applies these construction to operads. The brace algebra structure (hence, $\binfini$ structure) on $\bfP$
is defined in proposition \ref{prop37}; the implied pre-Lie product (hence, $\petitbinfini$ structure) in corollary \ref{cor38};
the associated dendriform products on $T(\bfP)$ and associative product on $S(\bfP)$ are given in proposition \ref{prop40},
giving two bialgebras and two Hopf algebras, namely $D_\bfP$, $\bfD_\bfP$, $B_\bfP$ and $\bfB_\bfP$.
If $\theta_\bfP:\petitbinfini \longrightarrow \bfP$ is an operadic morphism, then any $\bfP$-algebra $A$ is also $\petitbinfini$, which implies
that $S(A)$ becomes a Hopf algebra with a product $\star$ induced by the $\bfP$-algebra structure; in particular, if $\bfP=\prelie$
and $\theta_\prelie$ is the canonical surjection, we obtain in this way the Oudom-Guin construction.
Applied to the free $\bfP$-algebra in one variable $F_\bfP$, we prove that the graded, connected Hopf algebra $A_\bfP=S(F_\bfP)$, 
with this product $\star$, is in interaction with $D_\bfP$. There are two ways to obtain a duality; firstly, dualizing the composition of $\bfP$
gives bialgebras $\bfD_\bfP^*$ (non-commutative) and $D_\bfP^*$ (commutative), such that $A_\bfP^*$ and $D_\bfP^*$ are in cointeraction
(corollary \ref{cor51}). Secondly, using the $0$-boundedness of $\bfP$, one obtains  bialgebras $\bfD'_\bfP$ (non-commutative) and $D'_\bfP$
(commutative), different but isomorphic to the preceding ones (proposition \ref{prop52}, corollary \ref{cor53}).
Consequently, we obtain a monoid $M_\bfP^D$ of characters of $D_\bfP^*$, which can be described with the product $\lozenge$ of chapter 2,
acting on the completion of $F_\bfP$; it is shown that this is the monoid of formal endomorphisms of $F_\bfP$ (proposition \ref{prop58}),
that is to say a Faà di Bruno-like monoid.

Examples are treated in the last chapter. We start with $\com$ and $\ass$, getting back Faà di Bruno bialgebras of commutative 
or non-commutative formal diffeomorphisms coacting on $\K[X]$. For the operad $\prelie$, we get back the two cointeracting bialgebras of rooted
trees of \cite{ManchonCalaque}, the first one being the Connes-Kreimer Hopf algebra, the second one being given by an extraction-contraction process.
We then treat the case of Feynman graphs (definition \ref{defi60}), and give them two operadic compositions, $\circ$ and $\nabla$ (theorem \ref{theo67}).
It contains a suboperad of Feynman graphs without oriented cycles and a suboperad of simple Feynman graphs; forgetting the external structure,
we obtain operads on various families of graphs; considering Hasse graphs, we obtain operads on quasi-posets (or equivalently, finite topologies)
and posets, which were described in \cite{ManchonFoissyFauvet}. Consequently, we obtain pairs of interacting bialgebras
on these objects; for certain families of graphs, this was done in \cite{Manchon2}; for quasi-posets, in \cite{ManchonFoissyFauvet2,FoissyEhrhart}.
The functoriality also gives morphisms between these objects. The last chapter is a summary of the different objects attached to an operad used
in the paper.

\chapter*{Notations}
\begin{itemize}
\item For all $n\in \mathbb{N}$, we put $[n]=\{1,\ldots, n\}$. In particular, $[0]=\emptyset$.
\item $\K$ is a commutative field of characteristic zero. All objects (vector spaces, algebras, coalgebras, operad$\ldots$) in this text are taken over $\K$.
\item Let $V$ be a vector space. 
\begin{itemize}
\item We denote by $T(V)$ the tensor algebra of $V$, that is to say:
$$T(V)=\bigoplus_{n=0}^\infty V^{\otimes n}.$$
It is given an algebra structure with the concatenation product $m_{conc}$, and a coalgebra structure with the deconcatenation coproduct $\Delta_{dec}$:
\begin{align*}
\forall x_1,\ldots,x_n\in V,\: \Delta_{dec}(x_1\ldots x_n)&=\sum_{i=0}^n x_1\ldots x_i\otimes x_{i+1}\ldots x_n.
\end{align*}
The shuffle product $\shuffle$ makes $(T(V),\Delta_{dec})$ a commutative Hopf algebra. For example, if $x_1,x_2,x_3,x_4\in V$:
\begin{align*}
x_1\shuffle x_2&=x_1x_2+x_2x_1,\\
x_1\shuffle x_2x_3&=x_1x_2x_3+x_2x_1x_3+x_2x_3x_1,\\
x_1x_2\shuffle x_3x_4&=x_1x_2x_3x_4+x_1x_3x_2x_4+x_1x_3x_4x_2+x_3x_1x_2x_4+x_3x_1x_4x_2+x_3x_4x_1x_2.
\end{align*}
The augmentation ideal of $T(V)$ is denoted by $T_+(V)$.
\item We denote by $S(V)$ the symmetric algebra of $V$, with its usual product.
It is a Hopf algebra,  with the coproduct $\Delta$ defined by: 
\begin{align*}
\forall x_1,\ldots,x_k\in V,\: \Delta(x_1\ldots x_k)&=\sum_{I\subseteq [k]} \prod_{i\in I} x_i\otimes \prod_{i\notin I} x_i.
\end{align*}
The augmentation ideal of $S(V)$ is denoted by $S_+(V)$.
\end{itemize}
\item Let $(\bfP(n))_{n\geq 0}$ be a family of vector spaces. We denote:
\begin{align*}
\bfP&=\bigoplus_{n\geq 0}\bfP(n),&\overline{\bfP}&=\prod_{n\geq 0}\bfP(n).
\end{align*}
The family $(\bfP_+(n))_{n\geq 0}$ is defined by $\bfP_+(n)=\begin{cases}
\bfP(n)\mbox{ if }n\geq 2,\\
(0)\mbox{ otherwise}.
\end{cases}$
\item Recall that a $\mathbb{S}$-module is a family $(\bfP(n))_{n\geq 0}$ such that for all $n\geq 0$, $\bfP(n)$ 
is a right $\mathfrak{S}_n$-module, where $\mathfrak{S}_n$ is the $n$-th symmetric group. We denote:
\begin{align*}
coinv\bfP(n)&=\frac{\bfP(n)}{Vect(p-p^\sigma\mid p\in \bfP(n),\:\sigma\in \mathfrak{S}_n)},\\
inv\bfP(n)&=\{p\in \bfP(n)\mid \forall \sigma \in \mathfrak{S}_n,\: p^\sigma=p\}.
\end{align*}
both $(coinv\bfP(n))_{n\geq 0}$ and $(inv\bfP(n))_{n\geq 0}$ are $\mathfrak{S}$-modules, with a trivial action of the symmetric groups.
\end{itemize}

\chapter{Reminders on operads}

We here briefly recall the notions we shall use later on operads. We refer to \cite{Vallette,Markl,Mendez,Yau}  for more details.

\section{Non-$\Sigma$ operads and operads}

A non-$\Sigma$ operad $\bfP$ is a family $(\bfP(n))_{n\geq 0}$ of vector spaces with, for all $n,k_1,\dots,k_n \geq 0$, a composition $\circ$:
\begin{align*}
\circ:\left\{\begin{array}{rcl}
\bfP(n)\otimes \bfP(k_1)\otimes \ldots \otimes \bfP(k_n)&\longrightarrow&\bfP(k_1+\ldots+k_n)\\
p\otimes p_1\otimes \ldots \otimes p_n&\longrightarrow&p\circ(p_1,\ldots,p_n),
\end{array}\right.
\end{align*}
such that for all $n,k_1,\ldots,k_n,l_{1,1},\ldots,l_{n,k_n}\geq 0$, for all $p\in \bfP(n)$, $p_i\in \bfP(k_i)$, $p_{i,j}\in \bfP(l_{i,j})$:
\begin{align*}
(p\circ (p_1,\ldots,p_n))\circ (p_{1,1},\ldots,p_{n,k_n})
&=p\circ (p_1\circ (p_{1,1},\ldots,p_{1,k_1}),\ldots,p_n\circ (p_{n,1},\ldots,p_{n,k_n})).
\end{align*}
There exists an element $I\in \bfP(1)$, called the unit of $\bfP$, such that for all $n\geq 1$, for all $p\in \bfP(n)$:
\begin{align*}
I\circ p&=p,&p\circ (I,\ldots,I)&=p.
\end{align*}

An operad is a non-$\Sigma$ operad $\bfP$ and a $\mathbb{S}$-module;
there is a compatibility between the action of the symmetric groups and the composition we won't detail here.
We just mention that for all $n,k_1,\ldots,k_n \geq 0$, for all $p\in \bfP(n)$, $p_i \in \bfP(k_i)$, for all $\sigma \in \mathfrak{S}_n$,
$\sigma_i \in \mathfrak{S}_{k_i}$, there exists a particular $\sigma'\in \mathfrak{S}_{k_1+\ldots+k_n}$ such that:
$$p^\sigma \circ (p_1^{\sigma_1},\ldots,p_n^{\sigma_n})=(p\circ (p_{\sigma^{-1}(1)},\ldots,p_{\sigma^{-1}(n)}))^{\sigma'}.$$

\textbf{Notations.} Let $\bfP$ be a non-$\Sigma$ operad, $n\geq 1$, $i_1,\ldots,i_k \in [n]$, all distinct, $p_1,\ldots,p_k \in \bfP$.
We put:
$$p\circ_{i_1,\ldots,i_k}(p_1,\ldots,p_k)=p\circ (p'_1,\ldots,p'_n),\: \mbox{ where }p'_j=\begin{cases}
p_l \mbox{ if }j=i_l,\\
I\mbox{ otherwise}.
\end{cases}$$
For example, if $p\in \bfP(3)$, $p_1,p_2,p_3\in \bfP$:
\begin{align*}
p\circ_1p_1&=p\circ (p_1,I,I),&p\circ_3 p_1&=p\circ (I,I,p_1),\\
p\circ_{1,2}(p_1,p_2)&=p\circ (p_1,p_2,I),&p\circ_{2,1}(p_1,p_2)&=p\circ(p_2,p_1,I),\\
p\circ_{1,3}(p_1,p_2)&=p\circ(p_1,I,p_2),&p\circ_{3,1}(p_1,p_2)&=p\circ(p_2,I,p_1),\\
p\circ_{1,2,3}(p_1,p_2,p_3)&=p\circ(p_1,p_2,p_3),&p\circ_{2,3,1}(p_1,p_2,p_3)&=p\circ(p_3,p_1,p_2).
\end{align*}

\textbf{Remark.} We shall here only consider operads $\bfP$ such that $\bfP(0)=(0)$.

\section{Algebras over an operad}

Let $V$ be a vector space. Let us recall the construction of the operad $\bfL_V$ of multilinear endomorphisms of $V$:
\begin{itemize}
\item  For all $n\geq 0$, $\bfL_V(n)$ is the space of linear maps from $V^{\otimes n}$ to $V$.
\item The operadic composition is defined by:
$$\forall f\in \bfL_V(n),\: f_i\in \bfL_V(k_i),\:f\circ (f_1,\ldots,f_n)=f\circ (f_1\otimes \ldots \otimes f_n)\in \bfL_V(k_1+\ldots+k_n).$$ 
The unit is $Id_V$. 
\item The action of $\mathfrak{S}_n$ on $\bfL_V(n)$ is defined by:
$$f^\sigma(x_1 \ldots x_n)=f(x_{\sigma^{-1}(1)}\ldots x_{\sigma^{-1}(n)}).$$
\end{itemize}

Let $\bfP$ be an operad. An algebra over $\bfP$ is a pair $(V,\rho)$, where $V$ is vector space and $\rho:\bfP\longrightarrow\bfL_V$
is an operad morphism. In other words, for all $n\geq 1$, there exists a map:
$$\left\{\begin{array}{rcl}
\bfP(n)\otimes V^{\otimes n}&\longrightarrow&V\\
p\otimes v_1\ldots v_n&\longrightarrow&p.(v_1\ldots v_n)=\rho(p)(v_1\ldots v_n),
\end{array}\right.$$
such that:
\begin{itemize}
\item For all $v\in V$, $I.v=v$.
\item For all $p\in \bfP(n)$, $p_i\in \bfP(k_i)$, $u_i\in V^{\otimes k_i}$, $p\circ(p_1,\ldots,p_n).u_1\ldots u_n=p.((p_1.u_1)\ldots(p_n.u_n))$.
\item For all $p\in \bfP_n$, $\sigma\in \mathfrak{S}_n$, $x_1,\ldots,x_n \in V$, 
$p^\sigma.(x_1\ldots x_n)=p.(x_{\sigma^{-1}(1)}\ldots x_{\sigma^{-1}(n)})$.
\end{itemize}

Let $\bfP$ be an operad and $V$ be a vector space. The free $\bfP$-algebra generated by $V$ is:
$$F_\bfP(V)=\bigoplus_{n\geq 0} \bfP(n)\otimes_{\mathfrak{S}_n} V^{\otimes n},$$
where for all $n\geq 0$:
$$\bfP(n)\otimes_{\mathfrak{S}_n} V^{\otimes n}
=\frac{\bfP(n)\otimes V^{\otimes n}}{Vect(p^\sigma \otimes x_1\ldots x_n-p\otimes x_{\sigma^{-1}(1)} \ldots x_{\sigma^{-1}(n)})
\mid p\in \bfP(n),x_1,\ldots,x_n\in V)}.$$
The $\bfP$-algebra structure is given by:
$$p.(\overline{p_1\otimes u_1}\otimes \ldots \otimes \overline{p_n\otimes u_n})
=\overline{(p\circ (p_1,\ldots,p_n))\otimes (u_1\ldots u_n)}.$$

\textbf{Examples.} \begin{enumerate}
\item The operad $\com$ of commutative, associative algebras is generated by $m\in \com(2)$ and the relations:
\begin{align*}
m^{(12)}&=m,&m\circ_1m&=m\circ_2m.
\end{align*}
\begin{itemize}
\item Consequently, the $\com$-algebras are the commutative, associative (non necessarily unitary) algebras;
for any vector space $V$, $F_\com(V)=S_+(V)$. 
\item For all $n\geq 1$, $\com(n)=Vect(e_n)$, and for all $n,k_1,\ldots,k_n \geq 1$:
$$e_n \circ (e_{k_1},\ldots,e_{k_n})=e_{k_1+\ldots+k_n}.$$
\item For all $n\geq 1$, $\sigma \in \mathfrak{S}_n$, $e_n^\sigma=e_n$.
\item For any associative, commutative algebra $(V,\cdot)$, for any $x_1,\ldots,x_n \in V$:
$$e_n.(x_1,\ldots,x_n)=x_1\cdot\ldots\cdot x_n.$$
\end{itemize}
\item The operad $\ass$ of   associative algebras is generated by $m\in \ass(2)$ and the relation:
\begin{align*}
m\circ_1m&=m\circ_2m.
\end{align*}
\begin{itemize}
\item Consequently, the $\ass$-algebras are the associative (non necessarily unitary) algebras; 
for any vector space $V$, $F_\ass(V)=T_+(V)$. 
\item For all $n\geq 1$, $\ass(n)=Vect(\mathfrak{S}_n)$. Here are examples of compositions:
\begin{align*}
(12)\circ_1(12)&=(123),&(12)\circ_1(21)&=(213),&(12)\circ_2(12)&=(123),&(12)\circ_2(21)&=(132),\\
(21)\circ_1(12)&=(312),&(21)\circ_1(21)&=(321),&(21)\circ_2(12)&=(231),&(21)\circ_2(21)&=(321).
\end{align*}
\item For any $\sigma,\tau\in \mathfrak{S}_n$, $\sigma^\tau=\sigma\circ \tau$.
\item For any associative algebra $(V,\cdot)$, for any $\sigma\in \mathfrak{S}_n$, for any $x_1,\ldots,x_n \in V$,
$$\sigma.(x_1,\ldots,x_n)=x_{\sigma^{-1}(1)}\cdot\ldots\cdot x_{\sigma^{-1}(n)}.$$
\end{itemize}\end{enumerate}

\section{Operadic species}

We shall often work with operadic species \cite{Mendez}. A linear species $\bfP$ is a functor from the category of finite sets, with bijections as arrows,
to the category of vector spaces. An operadic species is a species $\bfP$ with compositions, defined for all non-empty finite sets $A$ and $B$,
and $a\in A$:
$$\circ_a:\left\{\begin{array}{rcl}
\bfP(A)\otimes \bfP(B)&\longrightarrow&\bfP(A\sqcup B\setminus \{a\})\\
p\otimes q&\longrightarrow&p\circ_a q,
\end{array}\right.$$
such that:
\begin{itemize}
\item For all finite sets $A,B,C$, for all $a\neq b \in A$, $p\in \bfP(A)$, $q\in \bfP(B)$, $r\in \bfP(C)$:
$$(p\circ_a q)\circ_b r=(p\circ_b r)\circ_a q.$$
\item For all finite sets $A,B,C$, for all $a \in A$, $b\in B$, $p\in \bfP(A)$, $q\in \bfP(B)$, $r\in \bfP(C)$:
$$(p\circ_a q)\circ_b r=p\circ_a( q\circ_b r).$$
\item For all singleton $\{a\}$, there exists $I_a\in \bfP(\{a\})$ such that:
\begin{itemize}
\item For all finite set $A$, for all $a\in A$, for all $p\in \bfP(A)$, $p\circ_a I_a=p$.
\item For all singleton $\{a\}$, for all finite set $A$, for all $p\in \bfP(A)$, $I_a\circ_a p=p$.
\end{itemize}\end{itemize}

\textbf{Examples.} \begin{enumerate}
\item We take $\com(A)=Vect(\{A\})$ for non-empty all finite set $A$. For all finite sets $A$, $B$, and $a\in A$:
$$\{A\}\circ_a\{B\}=\{A\sqcup B\setminus \{a\}\}.$$
For all singleton $\{a\}$, $I_a=\{a\}$.
\item For all non-empty finite set $A$, $\ass(A)$ is the set of all words $w$ with letters in $A$, such that any element $a\in A$ appears exactly one time
in $w$. If $a_1\ldots a_n \in \ass(A)$ and $a\in A$, there exists a unique $i\in [n]$ such that $a_i=a$; then:
$$a_1\ldots a_n \circ_a w=a_1\ldots a_{i-1} w a_{i+1}\ldots a_n.$$
For any singleton $\{a\}$, $I_a=a$.
\end{enumerate}

If $\bfP$ is an operadic species, one deduces an operad, also denoted by $\bfP$, in the following way:
\begin{itemize}
\item For all $n\geq 1$, $\bfP(n)=\bfP([n])$.
\item For all $p\in \bfP(m)$, $q\in \bfP(n)$, $i\in [n]$:
$$p\circ_i q=\bfP(\sigma_{m, n}^{(i)})(p\circ_i q),$$
with $\sigma_{m, n}^{(i)}:([m]\setminus \{i\})\sqcup [n]\longrightarrow [m+n-1]$ defined by: 
\begin{align*}
\sigma_{m, n}^{(i)}(1)&=1,&\sigma_{m, n}^{(i)}(\dot{1})&=i,&\sigma_{m,n}^{(i)}(i+1)&=n+i,\\
&\vdots&&\vdots&&\vdots\\
\sigma_{m, n}^{(i)}(i-1)&=i-1,&\sigma_{m, n}^{(i)}(\dot{n})&=n+i-1,&\sigma_{m, n}^{(i)}(m)&=m+n-1,
\end{align*}
where the elements of $[m]$ are denoted by $1,\ldots,m$, the elements of $[n]$ by $\dot{1},\ldots,\dot{n}$.
\item The unit is $I_1\in \bfP(\{1\})$.
\item For all $p\in \bfP(n)$ and $\sigma\in \mathfrak{S}_n$, $p^\sigma=\bfP(\sigma^{-1})(p)$.
\end{itemize}
For example, the operad associated to the operadic species $\com$ is the operad $\com$;
the operad associated to the operadic species $\ass$ is the operad $\ass$.

\chapter{Infinitesimal structures on primitive elements}

\section*{Introduction}

Let $V$ be a vector space, $m$ a product on $T(V)$ making $H=(T(V),m,\Delta_{dec})$ a bialgebra. Is it possible to obtain $m$ from
a structure on the space $V$ of primitive elements of $H$? The answer is positive, this is the $\binfini$ structure 
(or multibrace, or LR, or Hirsch algebra structure \cite{Getzler,Gerstenhaber,Kadei,Loday,Loday2})
on $V$. More precisely, there is a bijection between the sets of such products $m$ and the set of $\binfini$ structures on $V$
(theorem \ref{theo2}). As proved in \cite{Loday}, the operad $\binfini$ can be seen as a suboperad of the operad of
$2$-associative algebras (definition \ref{defi5} and theorem \ref{theo11}; we here give a complete proof of these well-known results.
We will also be interested in several particular cases of such products $m$: right-sided in the sense of \cite{Loday},
or two-sided, leading to quotients of the operad of $\binfini$ algebras, such as the $\brace$ operad 
(definition \ref{defi12}, \cite{Gerstenhaber,Ronco1}) or the associative operad.
In particular, if the $\binfini$ structure on $V$ is brace, then $m$ can be split into a dendriform structure \cite{Ronco1}, which we here
explicitly describe in theorem \ref{theo13}. 
We give a condition on a $\binfini$ algebra to obtain a dual bialgebra, namely the $0$-boundedness condition
(definition \ref{defi16} and \ref{prop17}).\\

There are cocommutative equivalents of these results, working on a symmetric coalgebras instead of tensor coalgebras;
then $\binfini$ structures are replaced by $\petitbinfini$ structures, and brace algebras by pre-Lie algebras in the second part of this chapter.

\section{$\binfini$ algebras}

\subsection{Definition and main property}


\begin{defi} 
Let $V$ be a vector space equipped with a map:
$$\langle-,-\rangle:\left\{\begin{array}{rcl}
T(V) \otimes T(V)&\longrightarrow&V\\
x_1\ldots x_k \otimes y_1\ldots y_l&\longrightarrow& \langle x_1\ldots x_k,y_1\ldots y_l\rangle.
\end{array}\right.$$
For all $k,l \in \N$, we put $\langle-,-\rangle_{k,l}=\langle-,-\rangle_{\mid V^{\otimes k} \otimes V^{\otimes l}}$. 
We shall say that $A$ is a $\binfini$ algebra if the following axioms are satisfied:
\begin{itemize}
\item For any $k\geq 0$, $\langle-,-\rangle_{0,k}=\langle-,-\rangle_{k,0}=\begin{cases}
Id_V\mbox{ if }k=1,\\
0\mbox{ otherwise}.
\end{cases}$
\item for any tensors $u,v,w \in T(V)$:
\begin{align}
\label{EQ1}\sum_{i=1}^{lg(u)+lg(v)} \sum_{\substack{u=u_1\ldots u_i,\\ v=v_1\ldots v_i}}\langle \langle u_1,v_1\rangle \ldots \langle u_i,v_i\rangle,w\rangle
&=\sum_{i=1}^{lg(v)+lg(w)} \sum_{\substack{v=v_1\ldots v_i,\\ w=w_1\ldots w_i}}\langle u, \langle v_1,w_1\rangle \ldots \langle v_i,w_i\rangle\rangle.
\end{align}\end{itemize}
The operad of $\binfini$ algebras is denoted by $\binfini$.
\end{defi}

\textbf{Remark.} If $i> lg(u)+lg(v)$, $u=u_1\ldots u_i$ and $v=v_1\ldots v_i$, there exists an index $i$ such that $u_i=v_i=1$,
so $\langle u_i,v_i\rangle=0$. Consequently, (\ref{EQ1}) can be rewritten as:
\begin{align}
\sum_{i\geq 1} \sum_{\substack{u=u_1\ldots u_i,\\ v=v_1\ldots v_i}} \langle\langle u_1,v_1\rangle \ldots \langle u_i,v_i\rangle,w\rangle
=\sum_{i\geq 1} \sum_{\substack{v=v_1\ldots v_i,\\ w=w_1\ldots w_i}} \langle u,\langle v_1,w_1\rangle \ldots \langle v_i,w_i\rangle\rangle.
\end{align}

A combinatorial description of free $\binfini$-algebras can be found in \cite{FMPBinfini}.

\begin{theo} \label{theo2}
Let $V$ be a vector space.  Let $Bialg(V)$ be the set of products $*$ on $T(V)$, making $(T(V),*,\Delta_{dec})$ a bialgebra.
Let $\binfini(V)$ be the set of $\binfini$ structures on $V$. These two sets are in bijections, via the maps:
\begin{align*}
\Phi_V&:\left\{\begin{array}{rcl}
B_\infty(V)&\longrightarrow&Bialg(V)\\
\langle-,-\rangle&\longrightarrow&\displaystyle *\mbox{ defined by }
u*v= \sum_{i\geq 1} \sum_{\substack{u=u_1\ldots u_i,\\ v=v_1\ldots v_i}} \langle u_1,v_1\rangle \ldots \langle u_i,v_i\rangle
\end{array}\right.\\
\Psi_V&:\left\{\begin{array}{rcl}
Bialg(V)&\longrightarrow&B_\infty(V)\\
*&\longrightarrow&\langle-,-\rangle\mbox{ defined by } \langle u,v\rangle=\pi(u*v),
\end{array}\right.\end{align*}
where $\pi$ is the canonical projection on $V$.
\end{theo}

\textbf{Remark.} If $*=\Phi_V(\langle-,-\rangle)$, (\ref{EQ1}) can be rewritten as:
\begin{align}
\langle u,v*w\rangle=\langle u*v,w\rangle.
\end{align}

The proof of this theorem will need the following two lemmas:

\begin{lemma} \label{lem3}
Let $C$ be a connected coalgebra and let $\phi,\psi:C\longrightarrow (T(V),\Delta_{dec})$ be two coalgebra morphisms. Then 
$\phi=\psi$ if, and only if, $\pi \circ \phi=\pi\circ \psi$.
\end{lemma}

\begin{proof} Let $(C_n)_{n\geq 0}$ be the coradical filtration of $C$, and $deg$ the associated degree; that is to say:
\begin{itemize}
\item $C_0$ is the coradical of $C$, that is to say the sum of simple subcoalgebras of $C$. 
As $C$ is connected, it has a unique group-like element denoted by $1_C$, and $C_0=Vect(1_C)$.
\item If $n\geq 1$, $C_n$ is uniquely defined by:
$$C_n=\{x\in C\mid \Delta(x)\in C\otimes C_{n-1}+C_{n-1}\otimes C).$$
\end{itemize}

As the unique group-like element of $T(V)$ is $1$, $\phi(1_C)=\psi(1_C)=1$.
For all $x \in Ker(\epsilon_C)$, we put $\tdelta(x)=\Delta(x)-x\otimes 1_C-1_C\otimes x$.
Then $\tdelta$ is a coassociative coproduct on $Ker(\epsilon_C)$ and for all $x\in Ker(\epsilon_C)\cap C_n$:
$$\tdelta(x)\in C_{n-1}\otimes C_{n-1}.$$

We assume that $\pi \circ \phi=\pi\circ \psi$. Let us prove that for all $x\in C$, $\phi(x)=\psi(x)$, by induction on $n=deg(x)$.
If $n=0$, then $x=\lambda 1_C$ for a certain $\lambda \in \K$, so $\phi(x)=\psi(x)=\lambda 1$.
Let us assume the result at all ranks $<n$. As $\phi(1_C)=\psi(1_C)=1$, using the induction hypothesis:
$$\tdelta_{dec}\circ \phi(x)=(\phi\otimes \phi)\circ \tdelta(x)=(\psi\otimes \psi)\circ \tdelta(x)=\tdelta_{dec} \circ \psi(x),$$
so $\phi(x)-\psi(x)\in Ker(\tdelta_{dec})=V$. Hence, $\phi(x)-\psi(x)=\pi(\phi(x)-\psi(x))=0$. \end{proof}

\begin{lemma} \label{lem4}
Let $\langle-,-\rangle: T(V)\otimes T(V)\longrightarrow V$ be any linear map such that $\langle 1,1\rangle=0$. The following map is a coalgebra morphism:
$$*:\left\{ \begin{array}{rcl}
T(V)\otimes T(V)&\longrightarrow&T(V)\\
u\otimes v&\longrightarrow&\displaystyle\sum_{i\geq 1} \sum_{\substack{u=u_1\ldots u_i,\\ v=v_1\ldots v_i}} \langle u_1,v_1\rangle \ldots \langle u_i,v_i\rangle.
\end{array}\right.$$
Moreover, for all $u,v\in T(V)$, $\pi(u*v)=\langle u,v\rangle$.
\end{lemma}

\begin{proof} Note that, as $\langle 1,1\rangle=0$, for all $u,v\in T(V)$:
$$u*v=\sum_{i=1}^{lg(u)+lg(v)} \sum_{\substack{u=u_1\ldots u_i,\\ v=v_1\ldots v_i}} \langle u_1,v_1\rangle \ldots \langle u_i,v_i\rangle,$$
so $u*v\in T(V)$. Let $u,v$ be two tensors of $T(V)$.
\begin{align*}
\Delta_{dec}(u*v)&=\sum_{i\geq 1} \sum_{\substack{u=u_1\ldots u_i,\\ v=v_1\ldots v_i}}\sum_{p=0}^i
\langle u_1,v_1\rangle \ldots \langle u_p,v_p\rangle\otimes \langle u_{p+1},v_{p+1}\rangle\ldots \langle u_i,v_i\rangle\\
&=\sum_{\substack{u=u^{(1)}u^{(2)},\\ v=v^{(1)}v^{(2)}}}\sum_{i,j \geq 1}
\sum_{\substack{u^{(1)}=u^{(1)}_1\ldots u^{(1)}_i,\\v^{(1)}=v^{(1)}_1\ldots v^{(1)}_i,\\
u^{(2)}=u^{(2)}_1\ldots u^{(2)}_j,\\v^{(2)}=v^{(2)}_i\ldots v^{(2)}_j}}
\langle u^{(1)}_1,v^{(1)}_1\rangle \ldots \langle u^{(1)}_i,v^{(1)}_i\rangle\otimes 
\langle u^{(2)}_1,v^{(2)}_1\rangle \ldots \langle u^{(2)}_i,v^{(2)}_j\rangle\\
&=\sum_{\substack{u=u^{(1)}u^{(2)},\\ v=v^{(1)}v^{(2)}}} u^{(1)}*v^{(1)}\otimes u^{(2)}*v^{(2)}\\
&=\Delta_{dec}(u)*\Delta_{dec}(v).
\end{align*}
So $*$ is a coalgebra morphism. \end{proof}\\

\begin{proof} (Theorem \ref{theo2}).
Let us first prove that $\Phi_V$ is well-defined. By lemma \ref{lem4}, the product $*=\Phi_V(\langle-,-\rangle)$ is indeed a coalgebra morphism.
Moreover, by convention $1*1=1$, and for all $v=x_1\ldots x_k \in T(V)$, with $k\geq 1$:
\begin{align*}
1*v&=\langle 1,x_1\rangle\ldots \langle 1,x_k\rangle+0=x_1\ldots x_k=v,\\
v*1&=\langle x_1,1\rangle\ldots \langle x_k,1\rangle+0=x_1\ldots x_k=v.
\end{align*}
So $1$ is a unit for the product $*$. For all tensors $u,v \in T(V)$, if $u$ or $v$ is not a word of length $0$,
then $u*v$ is a sum of words of length $\geq 1$: this implies that $\epsilon(u*v)=\epsilon(u)\epsilon(v)$.

We consider the two following morphisms:
\begin{align*}
F&:\left\{\begin{array}{rcl}
T(V)\otimes T(V)\otimes T(V)&\longrightarrow&T(V)\\
u\otimes v\otimes w&\longrightarrow&(u*v)*w,
\end{array}\right.\\
G&:\left\{\begin{array}{rcl}
T(V)\otimes T(V)\otimes T(V)&\longrightarrow&T(V)\\
u\otimes v\otimes w&\longrightarrow&u*(v*w).
\end{array}\right.
\end{align*}
As $*$ is a coalgebra morphism, $F$ and $G$ are coalgebra morphisms. Moreover, for all tensors $u,v,w\in T(V)$:
\begin{align*}
\pi\circ F(u\otimes v\otimes w)&=\langle u*v,w\rangle,&
\pi\circ G(u\otimes v\otimes w)&=\langle u,v*,w\rangle.
\end{align*}
By (\ref{EQ1}), $\pi\circ F=\pi\circ G$. By lemma \ref{lem3}, $F=G$, so $*$ is associative. Finally, $*\in Bialg(V)$.\\

Let us prove that $\Psi_V$ is well-defined. Let $*\in Bialg(V)$ and $\langle-,-\rangle=\pi\circ *$.
The unit of $*$ is a group-like of $T(V)$, so is equal to $1$. Hence, if $x_1,\ldots,x_k\in V$:
$$\langle 1,x_1\ldots x_k\rangle=\pi(1*x_1\ldots x_k)=\pi(x_1\ldots x_k)=\delta_{1,k} x_1\ldots x_k,$$
so $\langle-,-\rangle_{0,k}=Id_V$ if $k=1$ and $0$ otherwise. The condition on $\langle-,-\rangle_{k,0}$ is proved similarly.

Let $u,v,w\in T(V)$. Then:
$$\langle u*v,w\rangle=\pi((u*v)*w)=\pi(u*(v*w))=\langle u,v*w\rangle,$$
so $\langle-,-\rangle\in B_\infty(V)$.\\

Let $\langle-,-\rangle\in B_\infty(V)$ and let $*=\Phi_V(\langle-,-\rangle)$. By lemma \ref{lem3}, 
for all $u,v \in T(V)$, $\pi(u*v)=\langle u,v\rangle$, so $\Psi_V\circ \Phi_V=Id_{B_\infty(V)}$.
Let $*\in Bialg(V)$, $\langle-,-\rangle=\Psi_V(*)$ and $\star=\Phi_V(\langle-,-\rangle)$. Then $*$ and $\star$ are both coalgebra morphisms
from $T(V)\otimes T(V)$ to $T(V)$, and for all $u,v \in T(V)$:
$$\pi(u*v)=\langle u,v\rangle=\pi(u\star v).$$
By lemma \ref{lem4}, $*=\star$, so $\Phi_V\circ \Psi_V=Id_{Bialg(V)}$. \end{proof}\\

\textbf{Example.} Let $x_1,x_2,y_1,y_2\in V$.
\begin{align*}
x_1*y_1&=x_1y_1+y_1x_1+\langle x_1,y_1\rangle,\\
x_1x_2*y_1&=x_1x_2y_1+x_1y_1x_2+y_1x_1x_2+\langle x_1,y_1\rangle x_2+x_1\langle x_2,y_1\rangle+\langle x_1x_2,y_1\rangle,\\
x_1*y_1y_2&=x_1y_1y_2+y_1x_1y_2+y_1y_2x_1+\langle x_1,y_1\rangle y_2+y_1\langle x_1,y_2\rangle+\langle x_1,y_1y_2\rangle.
\end{align*}

\textbf{Remarks.} \begin{enumerate}
\item For any vector space, there exists a trivial $\binfini$ structure on it, defined by:
$$\langle-,-\rangle_{k,l}=\begin{cases}
Id \mbox{ if $(k,l)=(0,1)$ or $(1,0)$},\\
0\mbox{ otherwise}.
\end{cases}$$
The associated product $\Phi_V(0)$ is the shuffle product $\shuffle$.
\item The coalgebra $(T(V),\Delta_{dec})$ is connected, so for any product $*\in Bialg(V)$, $(T(V),*,\Delta_{dec})$ is a Hopf algebra.
\end{enumerate}

\subsection{$2$-associative algebras}

\begin{defi}\label{defi5}
A $2$-associative algebra \cite{Loday} is a family $(V,*,m)$, where $V$ is a vector space, and  $*$, $m$ are both associative products on $V$. 
The operad of $2$-associative algebras is denoted by $\deuxas$. It is generated by $*$ and $m$, both in $\deuxas(2)$, with the relations:
\begin{align*}
*\circ_1*&=*\circ_2*,&
m\circ_1m&=m\circ_2m.
\end{align*}\end{defi}

\begin{lemma}\begin{enumerate}
\item Let $A$ be a $2$-associative algebra. The products $m$ and $*$ on $A$ are extended to $\mathbf{U}A=\K\oplus A$, 
making it a $2$-associative algebra: for all $a,a'\in A$,
\begin{align*}
a*1&=a,&1*a'&=a',\\
a1&=a,&1a'&=a'.
\end{align*}

\item Let $A$ and $B$ be two $2$-associative algebras. Then $A\totimes B=(A\otimes \K)\oplus (\K\otimes B)\oplus(A\otimes B)$
is a $2$-associative algebra: if $a,a'\in A$, $b,b'\in B$,
\begin{align*}
&\begin{array}{c|c|c|c}
*&a'\otimes 1&1\otimes b'&a'\otimes b'\\
\hline a\otimes 1&a*a'\otimes 1&a\otimes b'&a*a'\otimes b'\\
\hline 1\otimes b&a'\otimes b&1\otimes b*b'&a'\otimes b*b'\\
\hline a\otimes b&a*a'\otimes b&a\otimes b*b'&a*a'\otimes b*b'
\end{array}&
&\begin{array}{c|c|c|c}
m&a'\otimes 1&1\otimes b'&a'\otimes b'\\
\hline a\otimes 1&aa'\otimes 1&a\otimes b'&aa'\otimes b'\\
\hline 1\otimes b&0&1\otimes bb'&0\\
\hline a\otimes b&0&a\otimes bb'&0
\end{array}\end{align*}
Moreover, if $A$, $B$ and $C$ are three $2$-associative algebras, then $(A\totimes B)\totimes C=A\totimes (B\totimes C)$.
\end{enumerate}\end{lemma}

\begin{proof} Direct verifications. \end{proof}

\begin{defi}
A $2$-associative bialgebra is a family $(A,*,m,\Delta)$, where:
\begin{enumerate}
\item $(A,*,m)$ is a $2$-associative algebra.
\item $\Delta:\mathbf{U}A\longrightarrow \mathbf{U}A\otimes \mathbf{U}A=(\K\otimes \K)\oplus (A\totimes A)\equiv \mathbf{U}(A\totimes A)$
is a coassociative, counitary coproduct, whose counit is the canonical projection on $\K$, and sending $1$ to $1\otimes 1$.
\item $\Delta$ is a morphism of $2$-associative algebras.
\end{enumerate}\end{defi}

\textbf{Remark.} Let $A$ be a $2$-associative bialgebra.
\begin{enumerate}
\item As the counit of  $\mathbf{U}A$ is the canonical projection on $\K$, for all $x\in A$:
$$\Delta(a)=a^{(1)}\otimes a^{(2)}=a\otimes 1+1\otimes a+a'\otimes a'',\mbox{ with }a'\otimes a''\in A\otimes A.$$
\item For all $a,b\in A$:
\begin{align*}
\Delta(a*b)&=\Delta(a)*\Delta(b),\\
\Delta(ab)&=(a\otimes 1+1\otimes a+a'\otimes a'')(b\otimes 1+1\otimes b+b'\otimes b'')\\
&=ab\otimes 1+a\otimes b+ab'\otimes b''+1\otimes ab+a'\otimes a''b\\
&=a^{(1)}\otimes a^{(2)}b+ab^{(1)}\otimes b^{(2)}-a\otimes b.
\end{align*}
In other words, $(\mathbf{U}A,*,\Delta)$ is a bialgebra and $(\mathbf{U}A,m,\Delta)$ is an infinitesimal bialgebra \cite{Berstein,Loday2,Loday}.
\end{enumerate}

\begin{defi}\begin{enumerate}
\item If $A$ is a $2$-associative algebra, we denote by $Prim(A)$ the space of elements $a\in A$ such that $\Delta(a)=a\otimes 1+1\otimes a$.
\item For all $n\geq 1$, we denote by $\primdeuxas(n)$ the space of elements $p\in \deuxas(n)$ such that for any $2$-associative bialgebra
$A$, for any $a_1,\ldots,a_n \in Prim(A)$, $p.(a_1,\ldots,a_n)\in Prim(A)$. 
\end{enumerate}\end{defi}

Note that $\primdeuxas$ is a suboperad of $\deuxas$.

\begin{prop}
Let $X_1,\ldots,X_n$ be indeterminates, and $V_n$ be the vector space generated by $X_1,\ldots,X_n$.
Recall that $F_\deuxas(V_n)$ is the free $2$-associative algebra generated by $V_n$. Let $\Delta$ be the unique $2$-associative
algebra morphism such that:
$$\Delta:\left\{\begin{array}{rcl}
\mathbf{U}F_\deuxas(V_n)&\longrightarrow&\mathbf{U}F_\deuxas(V_n)\otimes \mathbf{U}F_\deuxas(V_n)\\
X_i,\: i\in [n]&\longrightarrow&X_i\otimes 1+1\otimes X_i,\\
1&\longrightarrow&1\otimes 1.
\end{array}\right.$$
Then $(F_\deuxas(V_n),*,m,\Delta)$ is a $2$-associative bialgebra. Moreover, for all $n\geq 1$:
$$\primdeuxas(n)=\{p\in \deuxas(n)\mid p.(X_1,\ldots,X_n) \in Prim(F_\deuxas(V_n))\}.$$
\end{prop}

\begin{proof} 1. We consider $(\Delta\otimes Id)\circ \Delta$ and $(Id \otimes \Delta)\circ \Delta$. They are both morphisms of $2$-associative algebras
from $F_\deuxas(V_n)$ to $F_\deuxas(V_n)^{\otimes 3}$, sending $X_i$ to $X_i\otimes 1\otimes 1+1\otimes X_i\otimes 1+1\otimes 1\otimes X_i$
for all $i$, so they are equal: $\Delta$ is coassociative. Consequently, $F_\deuxas(V_n)$ is indeed a $2$-associative bialgebra.\\

2. $\subseteq$: immediate, as $F_\deuxas(V_n)$ is a $2$-associative bialgebra and $X_1,\ldots,X_n$ are primitive elements of this bialgebra.\\
$\supseteq$: let $p\in \deuxas(n)$, such that $p.(X_1,\ldots,X_n)\in Prim(F_\deuxas(V_n))$.
Let $A$ be a $2$-associative bialgebra and let $a_1,\ldots,a_n$ be primitive elements of $A$.
By universal property, there exists a $2$-associative algebra morphism $\phi:F_\deuxas(V_n)\longrightarrow A$,
sending $X_i$ to $a_i$ for all $i$. As $a_i$ is primitive for all $i$, $\Delta\circ \phi(X_i)=(\phi\otimes \phi)\circ \Delta(X_i)$ for all $i$;
as moreover both $(\phi\otimes \phi)\circ \Delta$ and $\Delta\circ \phi$ are $2$-associative algebra morphisms from
$F_\deuxas(V_n)$ to $A\otimes A$, they are equal, so $\phi$ is a $2$-associative bialgebra morphism.
As $p.(X_1,\ldots,X_n)$ is primitive, its image by $\phi$ also is:
$$\phi(p.(X_1,\ldots,X_n))=p.(\phi(X_1),\ldots,\phi(X_n))=p.(a_1,\ldots,a_n)\in Prim(A).$$
So $p\in \primdeuxas(n)$. \end{proof}\\

\textbf{Remark.} The following map is injective:
$$\left\{\begin{array}{rcl}
\deuxas(n)&\longrightarrow&F_\deuxas(V_n)\\
p&\longrightarrow&p.(X_1,\ldots,X_n).
\end{array}\right.$$
By restriction, the following map is injective:
$$\theta_{2As}:\left\{\begin{array}{rcl}
\primdeuxas(n)&\longrightarrow&Prim(F_\deuxas(V_n))\\
p&\longrightarrow&p.(X_1,\ldots,X_n).
\end{array}\right.$$

\begin{lemma}
Let $A=(A,*,m,\Delta)$ be a connected (as a coalgebra) $2$-associative bialgebra. 
\begin{enumerate}
\item The following map is a coalgebra morphism:
$$\xi:\left\{\begin{array}{rcl}
T(Prim(A))&\longrightarrow&\mathbf{U}A\\
a_1\ldots a_k&\longrightarrow&a_1\ldots  a_k.
\end{array}\right.$$
From now on, we assume that $A=(T(V),\Delta_{dec})$ as a coalgebra, and that $m$ is the concatenation product $m_{conc}$.
\item For all $k,l\in \mathbb{N}$, there exists a unique $p_{k,l}\in \primdeuxas(k+l)$, independent of $A$,
such that for all $a_1,\ldots,a_k,b_1,\ldots,b_l\in Prim(A)$:
$$\pi(a_1\ldots a_k*b_1\ldots b_l)=p_{k,l}.(a_1,\ldots,a_k,b_1,\ldots,b_l).$$
\end{enumerate}\end{lemma}

\begin{proof}
1. By the compatibility between $\Delta$ and $m$, an easy induction proves that for all $a_1,\ldots,a_k\in Prim(A)$:
$$\Delta(a_1\ldots a_k)=\sum_{i=0}^k a_1\ldots a_i \otimes a_{i+1}\ldots  a_k.$$
So $\xi$ is a coalgebra morphism. As $\xi_{\mid Prim(T(Prim(A)))}=\xi_{\mid Prim(A))}=Id_{Prim(A)}$ is injective,
$\xi$ is injective. Let us prove that $\xi$ is surjective. Let $(A_n)$ be the coradical filtration of $\mathbf{U}A$. We use the notations
of the proof of lemma \ref{lem4}. For all $a\in A_n$, let us prove that $a\in Im(\xi)$ by induction on $n$.
If $n=0$, we can assume that $a$ is the unique group-like of $A$. Then $a=\xi(1)$. Let us assume the result at all rank $<n$.
We can suppose that $\epsilon(a)=0$, that is to say $a\in A$. Then:
$$\tdelta_{dec}^{(n-1)}(a)=a_1\otimes \ldots \otimes a_n\in Prim(A)^{\otimes n},$$
so $\tdelta_{dec}^{(n-1)}(a)=\tdelta^{(n-1)}(a_1 \ldots  a_n)$. As a consequence, $a-a_1 \ldots  a_n\in A_{n-1}$.
By the induction hypothesis, $a-a_1 \ldots  a_n\in Im(\xi)$, so $a\in Im(\xi)$. \\

2. Let us prove the existence of $p_{k,l}$ by induction on $n=k+l$. Firstly, if $k=0$,
$$\pi(1*b_1\ldots b_l)=\pi(b_1\ldots b_l)=\begin{cases}
b_1\mbox{ if }l=1,\\
0\mbox{ otherwise.}
\end{cases}$$
So we take $p_{0,l}=\delta_{l,1}I$. Similarly, we take $p_{k,0}=\delta_{k,1}I$. We now assume that $k,l\geq 1$.
There is nothing more to prove if $n\leq 1$. Let us assume that $n\geq 2$. By lemma \ref{lem4}, if $a_1,\ldots,a_k,b_1,\ldots,b_l \in V$:
\begin{align*}
a_1\ldots a_k*b_1\ldots b_l&=\sum_{i=2}^{k+l}\sum_{\substack{a_1\ldots a_k=u_1\ldots u_i,\\ b_1\ldots b_l=v_1\ldots v_i}}
p_{lg(u_1),lg(v_1)}.(u_1,v_1)\ldots  p_{lg(u_i),lg(v_i)}.(u_i,v_i)\\
&+\pi(a_1\ldots a_k*b_1\ldots b_l).
\end{align*}
So there exists $p_{k,l}\in \deuxas(k+l)$, such that $\pi(a_1\ldots a_k*b_1\ldots b_l)=p_{k,l}.(a_1,\ldots,a_k,b_1,\ldots,b_l)$.
By definition, $p_{k,l}\in \primdeuxas$. By the injectivity of $\theta_{2As}$, $p_{k,l}$ is unique. \end{proof}

\begin{theo} \label{theo11}
The operad $\binfini$ is isomorphic to the operad $\primdeuxas$, through the morphism:
$$\left\{\begin{array}{rcl}
\binfini&\longrightarrow&\primdeuxas\\
\langle-,-\rangle_{k,l}&\longrightarrow&p_{k,l}.
\end{array}\right.$$
\end{theo}

\begin{proof} By theorem \ref{theo2}, for any connected $2$-associative bialgebra $A$, $p_{k,l}$ defines a $\binfini$ algebra structure on $Prim(A)$,
which gives the existence of this morphism.

Let $V$ be a space, $(W,\langle-,-\rangle)$ be a $\binfini$ algebra and $f:V\longrightarrow W$ be a linear map.
By restriction, $F_\deuxas(V)$ contains $F_\primdeuxas(V)$; by definition of $\primdeuxas$, 
$F_{\primdeuxas}(V)\subseteq Prim(F_\deuxas(V))$.
If $*=\Phi_W(\langle-,-\rangle)$, $(T(W),*,m,\Delta)$ is a $2$-associative algebra.
There exists a unique $2$-associative algebra morphism $F:F_\deuxas(V)\longrightarrow T(W)$, such that $F(x)=f(x)$ for all $x \in V$.
By construction, for all $x_1,\ldots,x_k,y_1,\ldots,y_l\in V$:
\begin{align*}
F(p_{k,l}.(x_1\ldots x_k,y_1\ldots y_l))&=F\circ \pi(x_1\ldots x_k*y_1\ldots y_l)\\
&=\pi(f(x_1)\ldots f(x_k)*f(y_1)\ldots f(y_l))\\
&=\langle f(x_1)\ldots f(x_k),f(y_1)\ldots f(y_l)\rangle_{k,l}.
\end{align*}
So the restriction of $F$ to $F_\primdeuxas(V)$ is a morphism of $\binfini$ algebras taking its values in $W$.
Finally, $F_\primdeuxas(V)$ satisfies a universal property in the category of $\binfini$ algebras: consequently, the operad morphism
from $\binfini$ to $\primdeuxas$ is an isomorphism. \end{proof}

\subsection{Quotients of $\binfini$}

\begin{defi} \label{defi12}
\cite{Gerstenhaber,Ronco1} 
The operad $\brace$ is the quotient of $\binfini$ by the operadic ideal generated by the elements $\langle-,-\rangle_{k,l}$, $k\geq 2$, $l\geq 0$. 
Consequently, a brace algebra is a vector space $V$ equipped with a map $\langle-,-\rangle:V\otimes T(V)\longrightarrow V$,
such that:
\begin{itemize}
\item For all $x\in V$, $\langle x,1\rangle=x$.
\item For all $x,y_1,\ldots,y_k \in V$, for all tensor $w\in T(V)$:
\begin{align}
\label{EQ4} \langle \langle x,y_1\ldots y_k\rangle,w\rangle
=\sum_{w=w_1\ldots w_{2k+1}} \langle x,w_1\langle y_1,w_2\rangle w_3\ldots w_{2k-1}\langle y_k,w_{2k}\rangle w_{2k+1}\rangle.
\end{align}\end{itemize}\end{defi}

Recall that a dendriform algebra is a family $(V,\prec,\succ)$, where $V$ is a vector space, and $\prec$, $\succ$ are bilinear products on $V$
such that for any $x,y,z\in V$:
\begin{align*}
(x\prec y)\prec z&=x\prec (y\prec z+y\succ z),\\
(x\succ y)\prec z&=x\succ (y\prec z),\\
x\succ (y\succ z)&=(x\prec y+x\succ y)\succ z.
\end{align*}
This implies that $*=\prec+\succ$ is associative.

\begin{theo}\label{theo13}
Let $V$ be a brace algebra. Then $T_+(V)$ is a dendriform Hopf algebra \cite{Loday3,Ronco1,FoissyDend}, 
with the deconcatenation coproduct and the products given in 
the following way: for all $x_1,\ldots,x_k \in V$, for all $v=y_1\ldots y_l \in T(V)$, with $(k,l)\neq (0,0)$,
\begin{align*}
x_1\ldots x_k*v&=\sum_{v=v_1\ldots v_{2k+1}} v_1\langle x_1,v_2\rangle v_3\ldots v_{2k-1}\langle x_k,v_{2k}\rangle v_{2k+1},\\
x_1\ldots x_k\prec v&=\sum_{v=v_1\ldots v_{2k}}
\langle x_1,v_1\rangle v_2\ldots v_{2k-2} \langle x_k,v_{2k-1}\rangle v_{2k},\\
x_1\ldots x_k\succ v&=\sum_{\substack{v=v_1\ldots v_{2k+1},\\ v_1\neq 1}} v_1\langle x_1,v_2\rangle v_3\ldots v_{2k-1}\langle x_k,v_{2k}\rangle v_{2k+1}
\end{align*}\end{theo}

\begin{proof} As brace algebras are also $\binfini$ algebras, theorem \ref{theo2} can be applied. The product $*=\Phi_V(\langle-,-\rangle)$
is given by the announced formula: we immediately obtain that $(T(V),*,\Delta)$ is a Hopf algebra. 

Let $u,v,w\in T(V)$, $u$ being a non-empty tensor. The formulas defining $*$ and $\prec$ give, with Sweedler's notations:
\begin{align*}
u*v&=v^{(1)} (u\prec v^{(2)}),\\
(uv)*w&=(u*w^{(1)})(v\prec w^{(2)}),\\
(uv)\prec w&=(u\prec w^{(1)})(v\prec w^{(2)}).
\end{align*}
By subtraction, we also obtain that $(uv)\succ w=(u\succ w^{(1)})(v\prec w^{(2)})$. We consider:
$$A=\{u\in T(V)\mid \forall v,w \in T_+(V), (u\prec v)\prec w=u\prec(v*w)\}.$$
Firstly, $1\in A$: indeed, for all $v,w\in T_+(V)$, $(1\prec v)\prec w=0=1\prec(v*w)$.
Let us assume that $u\in A$ and let $x\in V$. For all $v,w\in T_+(V)$:
\begin{align*}
(ux\prec v)\prec w&=((u\prec v^{(1)})(x\prec v^{(2)}))\prec w\\
&=((u \prec v^{(1)})\prec w^{(1)})((x\prec v^{(2)})\prec w^{(2)})\\
&=(u\prec (v^{(1)}*w^{(1)}))((\langle x,v^{(2)}\rangle v^{(3)})\prec w^{(2)})\\
&=(u\prec (v^{(1)}*w^{(1)}))(\langle x,v^{(2)}\rangle \prec w^{(2)})(v^{(3)}\prec w^{(3)})\\
&=(u\prec (v^{(1)}*w^{(1)}))\langle\langle x,v^{(2)}\rangle ,w^{(2)}\rangle w^{(3)}(v^{(3)}\prec w^{(4)})\\
&=(u\prec (v^{(1)}*w^{(1)}))\langle x,v^{(2)}*w^{(2)}\rangle (v^{(3)}*w^{(3)})\\
&=(u\prec (v*w)^{(1)})\langle x,(v*w)^{(2)}\rangle (v*w)^{(3)}\\
&=(u\prec (v*w)^{(1)})(x\prec(v*w)^{(2)})\\
&=(ux)\prec (v*w).
\end{align*}
So $ux\in A$. As a consequence, $A=T(V)$: for all $u,v,w\in T_+(V)$, $(u\prec v)\prec w=u\prec (v*w)$.\\

Let $u=x_1\ldots x_k \in T(V)$, with $k\geq 1$, $v\in T(V)$ and $y\in V$. Then:
$$\Delta_{dec}(yv)=yv^{(1)}\otimes v^{(2)}+1\otimes yv,$$
so:
\begin{align*}
u\succ yv&=u*yv-u\prec v\\
&=yv^{(1)}\langle x_1,v^{(2)}v^{(3)}\ldots v^{(2k-1)}\langle x_k,v^{(2k)}\rangle v^{(2k+1)}\\
&+\langle x_1,(yv)^{(1)}\rangle (yv)^{(2)}\ldots (yv)^{(2k-2)}\langle x_k,(yv)^{(2k-1)}\rangle (yv)^{(2k)}\\
&-\langle x_1,(yv)^{(1)}\rangle (yv)^{(2)}\ldots (yv)^{(2k-2)}\langle x_k,(yv)^{(2k-1)}\rangle (yv)^{(2k)}\\
&=y\left(v^{(1)}\langle x_1,v^{(2)}\rangle v^{(3)}\ldots v^{(2k-1)}\langle x_k,v^{(2k)}\rangle v^{(2k+1)}\right)\\
&=y(u*v).
\end{align*}
Let $B=\{w\in T(V)\mid\forall u,v\in T_+(V), (u*v)\succ w=u\succ(v\succ w)\}$.
Firstly, $1\in B$: indeed, if $u,v\in T_+(V)$, $(u*v)\succ 1=0=u\succ(v\succ 1)$. Let $w\in B$ and $z\in V$. For all $u,v\in T_+(V)$:
\begin{align*}
(u*v)\succ zw&=z((u*v)*w)=z(u*(v*w))=u\succ(z(v*w))=u\succ(v\succ (zw)).
\end{align*} 
So $zw\in B$: consequently, $B=T(V)$. For all $u,v,w\in T_+(V)$, $(u*v)\succ w=u\succ (v\succ w)$. So $(T_+(V),\prec,\succ)$ is a dendriform algebra.\\

Let us prove the axioms of a dendriform Hopf algebra, that is to say for all $u,v\in T_+(V)$:
\begin{align*}
\tdelta_{dec}(u\prec v)&=u\otimes v+u'\otimes u''*v+u\prec v'\otimes v''+u'\prec v\otimes u''+u'\prec v'\otimes u''*v'',\\
\tdelta_{dec}(u\succ v)&=v\otimes y+v'\otimes u*v+u\succ v'\otimes v''+u'\succ v\otimes u''+u'\succ v'\otimes u''*v''.
\end{align*}
As we already know that $(T(V),*,\Delta_{dec})$ is a bialgebra, it is enough to prove one of these two relations, say the second one. 
Let $y\in V$ and $u\in T_+(V)$.
\begin{align*}
\tdelta_{dec}(u\succ y)&=\tdelta_{dec}(yu)=y\otimes u+yu'\otimes u''=y\otimes u+u'\succ y\otimes u'',
\end{align*}
which proves the relation for $v\in V$, as then $\tdelta_{dec}(v)=0$. If $v$ is a tensor of length $\geq 2$, we write it as $v=yw$, with $y\in V$
and $w\in T_+(V)$. Then:
\begin{align*}
\tdelta_{dec}(u\succ v)&=\tdelta_{dec}(y(u*w))\\
&=y\otimes u*w+yu\otimes w+yw\otimes u+y(u'*w)\otimes u''\\
&+yu'\otimes u''*w+y(u*w')\otimes w''+yw'\otimes u*w''+y(u'*w')\otimes u''*w''\\
&=yw\otimes u+(y\otimes u*w+yw'\otimes u*w'')+(yu\otimes w+y(u*w')\otimes w'')\\
&+y(u'*w)\otimes u''+(yu'\otimes u''*w+y(u'*w')\otimes u''*w'')\\
&=yw\otimes u+(y\otimes u*w+yw'\otimes u*w'')+(u\succ y\otimes w+u\succ yw'\otimes w'')\\
&+u'\succ yw\otimes u''+(u'\succ y\otimes u''*w+u'\succ yw'\otimes u''*w'')\\
&=v\otimes u+v'\otimes u*v''+u\succ v'\otimes v''+u'\succ v\otimes u''+u'\succ v'\otimes u''*v''.
\end{align*}
So the second compatibility is verified. \end{proof}\\

\textbf{Remarks.} \begin{enumerate}
\item The products $\prec$ and $\succ$ can also be inductively defined: let $x_1,\ldots,x_k,y_1,\ldots,y_l\in V$, with $k,l\geq 1$.
\begin{align*}
x_1\ldots x_k \prec 1&=x_1\ldots x_k,\\
1\prec y_1\ldots y_l&=0,\\
x_1\ldots x_k \prec y_1\ldots y_l&=\sum_{p=0}^l \langle x_1,y_1\ldots y_p\rangle (x_2\ldots x_k*y_{p+1}\ldots y_l),\\
x_1\ldots x_k \succ 1&=0,\\
1\succ y_1\ldots y_l&=y_1\ldots y_l,\\
x_1\ldots x_k \succ  y_1\ldots y_l&=y_1(x_1\ldots x_k*y_2\ldots y_l).
\end{align*}
\item If $(V,\langle-,-\rangle)$ is brace, putting $*=\Phi_V(\langle-,-\rangle)$, (\ref{EQ4}) can be written in the following way:
\begin{align}
\forall x\in V,\: u,v\in T(V),\: \langle \langle x,u\rangle, v\rangle=\langle x,u*v\rangle.
\end{align}\end{enumerate}

\begin{prop}
The quotient of $\binfini$ by the operadic ideal generated by the elements $\langle-,-\rangle_{k,l}$, $k\geq 2$ or $l\geq 2$,
is isomorphic to the operad $\ass$ of associative algebras.
\end{prop}

\begin{proof} This quotient is generated by $\langle -,-\rangle_{1,1}$.
The unique relation defining $\binfini$ algebras which does not become trivial in this quotient is:
$$\forall x,y,z\in V,\: \langle \langle x,y\rangle,z\rangle=\langle x,\langle y,z\rangle\rangle.$$
So this quotient is indeed $\ass$. \end{proof}\\

Consequently, if $(V,\cdot)$ is an associative algebra, it is also a $\binfini$ algebra; 
$T(V)$ becomes a bialgebra with the product $*=\Phi_V(\cdot)$. For all $x_1,\ldots,k_k \in V$, $v$ a word in $T(V)$:
$$x_1\ldots x_k*v=\sum_{\substack{v=v_1\ldots v_{2k+1},\\ lg(v_2),\ldots,lg(v_{2k})\leq 1}}
v_1 (x_1\cdot v_2)v_3\ldots v_{2k-1} (x_k\cdot v_{2k}) v_{2k+1},$$
with the convention $x\cdot 1=x$ for all $x\in V$. This is the quasi-shuffle product associated to $\cdot$ \cite{Hoffman2,FoissyPatras}.
It is a dendriform algebra, with the following products:
\begin{align*}
x_1\ldots x_k\prec v&=\sum_{\substack{v=v_1\ldots v_{2k},\\ lg(v_1),\ldots,lg(v_{2k-1})\leq 1}}
(x_1\cdot v_1)v_2\ldots v_{2k-2} (x_k\cdot v_{2k-1}) v_{2k},\\
x_1\ldots x_k\succ v&=\sum_{\substack{v=v_1\ldots v_{2k+1},\\ v_1\neq 1,\:lg(v_2),\ldots,lg(v_{2k})\leq 1}}
v_1 (x_1\cdot v_2)v_3\ldots v_{2k-1} (x_k\cdot v_{2k}) v_{2k+1},
\end{align*} 

\subsection{Brace modules}

\begin{defi}
Let $V$ be a brace algebra. A brace module over $V$ is a vector space $M$ with a map:
$$\leftarrow:\left\{\begin{array}{rcl}
M\otimes T(V)&\longrightarrow&M\\
m\otimes x_1\ldots x_k&\longrightarrow&m\leftarrow x_1\ldots x_k,
\end{array}\right.$$
such that:
\begin{itemize}
\item For all $m\in M$, $m\leftarrow 1=m$.
\item For all $m\in M$, $y_1,\ldots,y_k \in V$, for all tensor $w\in T(V)$:
\begin{align}\label{EQ6}
(m\leftarrow y_1\ldots y_k)\leftarrow w=
m\leftarrow \left(\sum_{w=w_1\ldots w_{2k+1}} w_1\langle y_1,w_2\rangle w_3\ldots w_{2k-1}\langle y_k,w_{2k}\rangle w_{2k+1}\right).
\end{align}\end{itemize}\end{defi}

\textbf{Remark.} Let $*=\Phi_V(\langle-,-\rangle)$. Then (\ref{EQ6}) can be rewritten as:
\begin{align}
\forall m\in M,\: u,v \in T(V),\:(m\leftarrow u)\leftarrow v=m\leftarrow(u*v).
\end{align}
So a brace module over $V$ is a (right) module over the algebra $(T(V),*)$.\\

\textbf{Example.} If $V$ is a brace algebra, $(V,\langle-,-\rangle)$ is a brace module over itself.

\subsection{Dual construction}

\textbf{Notation}. Let $V$ be a graded space, such that the homogeneous components of $V$ are finite-dimensional.
We denote by $V^\circledast$ the dual of $V$. The graded dual of $V$ is:
$$V^*=\bigoplus_{n\geq 0} V_n^*\subseteq V^\circledast.$$
If $V_0=(0)$, then $S(V)$ and $T(V)$ are also graded spaces, and $S(V)^*$ is isomorphic to $S(V^*)$;
$T(V)^*$ is isomorphic to $T(V^*)$.

\begin{defi} \label{defi16}
Let $V$ be a graded $\binfini$ algebra. We shall say that $V$ is $0$-bounded if:
\begin{itemize}
\item For all $n\geq 0$, $V_n$ is finite-dimensional.
\item For all $m, n\geq 0$, there exists $B(m, n)\geq 0$ such that for all $p_1,\ldots,p_k,q_1,\ldots,q_l\geq 0$ with $p_1+\ldots+p_k=m$
and $q_1+\ldots+q_l=n$:
$$\sharp\{i\mid p_i=0\}+\sharp\{j\mid q_j=0\}>B(m, n)\Longrightarrow \langle V_{p_1}\ldots V_{p_k},V_{q_1}\ldots V_{q_l}\rangle=(0).$$
\end{itemize}\end{defi}

\textbf{Examples.}\begin{enumerate}
\item If $V$ is connected, that is to say if $V_0=(0)$, then $V$ is $0$-bounded, with $B(m, n)=0$ for all $m, n$.
\item If $V$ is associative, then $\langle-,-\rangle_{k,l}=0$ if $k\geq 2$ or $l\geq 2$. Consequently, $V$ is $0$-bounded, with 
$B(m, n)=2$ for all $m, n$.
\end{enumerate}

Let us assume that $V$ is $0$-bounded. We identify $T(V^*)$ with a subspace of $T(V)^*$ by the pairing $\ll-,-\gg'$ such that for 
all $x_1,\ldots,x_l\in V$, $f_1,\ldots,f_k \in V^*$:
$$\ll f_1\ldots f_k,x_1\ldots x_l\gg'=\begin{cases}
0\mbox{ if }k\neq l,\\
f_1(x_1)\ldots f_k(x_k) \mbox{ if }k=l.
\end{cases}$$
Note that for all $F,G \in T(V^*),$ $X,Y\in T(V)$:
\begin{align*}
\ll \Delta_{dec}(F), X\otimes Y\gg'&=\ll F,XY\gg,&
\ll F\otimes G,\Delta_{dec}(X)\gg'&=\ll FG,X\gg.
\end{align*}

\begin{prop} \label{prop17}
Let $V$ be a $0$-bounded $\binfini$ algebra. 
We define a coproduct $\Delta'_*$ on $T(V^*)$ as the unique algebra morphism such that for all $f\in V^*$, for all $X,Y \in T(V)$,
$$\ll \Delta'_*(f),X\otimes Y\gg'=\ll f,\langle X,Y\rangle \gg'.$$
Then $(T(V^*),m_{conc},\Delta'_*)$ is a graded bialgebra. Moreover, for all $F,G\in T(V^*),$ $X,Y\in T(V)$:
\begin{align*}
\ll \Delta'_*(F),X\otimes Y\gg'&=\ll F,X*Y\gg',&\ll F\otimes G,\Delta_{dec}(Y)\gg'&=\ll FG,X\gg'.
\end{align*}
In other words, $\ll-,-\gg'$ is a Hopf pairing. \end{prop}

\begin{proof} The $\binfini$ bracket can be dualized into a map $\delta$ from $V^*$ to $(T(V)\otimes T(V))^*$.
Unfortunately, if $V_0\neq (0)$, the homogeneous components of $T(V)$ are not finite-dimensional, so $T(V^*)\otimes T(V^*)$
is (identified to) a strict subspace of $(T(V)\otimes T(V))^*$. But, by the $0$-bounded condition, for all $N \geq 0$:
\begin{align*}
\delta(V_N^*)&\subseteq \sum_{m+n=N} \sum_{\substack{p_1+\ldots+p_k=m,\\ q_1+\ldots+q_l=n,\\ 
\sharp\{i\mid p_i=0\}+\sharp\{j\mid q_j=0\}\leq B(m, n)}} V^*_{p_1}\ldots V^*_{p_k}\otimes V^*_{q_1}\ldots V^*_{q_l}.
\end{align*}
As this is a finite sum, $\delta(V_N^*)\subseteq T(V^*) \otimes T(V^*)$. We can define an algebra morphism $\Delta'_*$ 
from $T(V^*)$ to $T(V^*)\otimes T(V^*)$ by $\Delta'_*(f)=\delta(f)$ for all $f\in V^*$. \\

Let us consider:
$$A=\{F\in T(V^*)\mid \forall X,Y\in T(V), \ll\Delta'_*(F),X\otimes Y\gg'=\ll F,X*Y\gg'.\}.$$
Firstly, $1\in A$: for all $X,Y\in T(V)$,
$$\ll \Delta'_*(1),X\otimes Y\gg'=\ll 1\otimes 1,X\otimes Y\gg'=\varepsilon(X)\varepsilon(Y)=\varepsilon(X*Y)=\ll 1,X*Y\gg'.$$
Let $F,G\in A$. For all $X,Y \in T(V)$:
\begin{align*}
\ll\Delta'_*(FG),X\otimes Y\gg'&=\ll F^{(1)}_*G^{(1)}_*\otimes F^{(2)}_*G^{(2)}_*,X\otimes Y\gg'\\
&=\ll F^{(1)}_*\otimes G^{(1)}_*\otimes F^{(2)}_*\otimes G^{(2)}_*,X^{(1)} \otimes X^{(2)}\otimes Y^{(1)}\otimes Y^{(2)}\gg'\\
&=\ll F\otimes G,X^{(1)}* Y^{(1)} \otimes X^{(2)}*Y^{(2)}\gg'\\
&=\ll F\otimes G,(X* Y)^{(1)} \otimes (X^*Y)^{(2)}\gg'\\
&=\ll FG, X*Y\gg'.
\end{align*}
So $A$ is a subalgebra of $(T(V^*),m_{conc})$. In order to prove that $\ll-,-\gg'$ is a Hopf pairing, it is enough to prove that $V^*\subseteq A$.
Let $f\in V^*$. For all $X,Y\in T(V)$:
\begin{align*}
\ll \Delta'_*(f),X\otimes Y\gg'&=\ll f,\langle X,Y\rangle \gg'\\
&=\ll f,\pi(X*Y)\gg'\\
&=\ll f,X*Y \gg'+0\\
&=\ll f,X*Y\gg'.
\end{align*}
So $\ll-,-\gg'$ is a Hopf pairing.\\

Let $F\in T(V^*)$, $X,Y,Z\in T(V)$.
\begin{align*}
\ll (\Delta'_*\otimes Id)\circ \Delta'_*(F),X\otimes Y\otimes Z\gg'
&=\ll F,(X*Y)*Z\gg'\\
&=\ll F,X*(Y*Z)\gg'\\
&=\ll(Id \otimes \Delta'_*)\circ \Delta'_*(F),X\otimes Y\otimes Z\gg'.
\end{align*}
As the pairing is non degenerate, $\Delta'_*$ is coassociative: $(T(V^*),m_{conc},\Delta'_*)$ is a bialgebra. \end{proof}\\

\textbf{Remark.} If $V$ is connected, then $(T(V^*),m_{conc},\Delta'_*)$ is a graded, connected bialgebra, 
so is a Hopf algebra, isomorphic to the graded dual of $(T(V),*,\Delta_{dec})$.

\begin{cor} \label{cor18}
Let $V$ be a $0$-bounded $\binfini$ algebra.
\begin{itemize}
\item $T(V_0^*)$ is a subbialgebra of $T(V^*)$. 
\item Let $I_0$ be the ideal of $T(V^*)$ generated by $V_0^*$.
Then $I_0$ is a biideal of $T(V^*)$, and $T(V^*)/I_0$ is a graded, connected Hopf algebra, isomorphic to the graded dual of $T(V_+)$.
\end{itemize}\end{cor}

\begin{proof} We already noticed that:
$$\delta(V_0^*)\subseteq \sum_{k+l\leq B(0,0)} V_0^{\otimes k}\otimes V_0^{\otimes l},$$
so $\Delta'_*(V_0^*)\subseteq T(V_0^*)\otimes T(V_0^*)$: $T(V_0^*)$ is a subbialgebra.\\

As $V_+$ is a $\binfini$ subalgebra of $V$, $T(V_+)$ is a Hopf subalgebra of $T(V)$: its orthogonal $J$ is a biideal of $T(V^*)$.
Moreover, if $x_1,\ldots,x_k \in V_+$, for all $f\in V_0^*$:
\begin{align*}
\ll f, x_1\ldots x_k\gg'&=\begin{cases}
0 \mbox{ if }k\neq 1,\\
f(x_1)\mbox{ if }k=1
\end{cases}\\
&=0.
\end{align*}
 So $I_0 \subseteq J$. Moreover, the following map is an algebra isomorphism:
 $$\left\{\begin{array}{rcl}
 T(V_+^*)&\longrightarrow&T(V^*)/I_0\\
 f_1\ldots f_k&\longrightarrow& \overline{f_1\ldots f_k}.
 \end{array}\right.$$
For all $f_1,\ldots,f_k\in V_+^*$, $x_1,\ldots,x_l\in V_+$:
$$\ll \overline{f_1\ldots f_k},y_1\ldots y_l\gg'=\ll f_1\ldots f_k,y_1\ldots y_l\gg'=\begin{cases}
0\mbox{ if }k\neq l,\\
f_1(x_1)\ldots f_k(x_k) \mbox{ if }k=l.
\end{cases}$$
So the pairing induced by $\ll-,-\gg'$ between $T(V^*)/I_0$ and $T(V_+)$ is non degenerate: therefore, $I_0=J$.
Finally, $T(V^*)/I_0$ is the graded dual of $(T(V_+),*,\Delta)$, so is a Hopf algebra. \end{proof}

\section{$\petitbinfini$ algebras}

\subsection{Definition and main property}

\begin{defi}
Let $V$ be a vector space equipped with a map:
$$\lfloor-,-\rfloor:\left\{\begin{array}{rcl}
S(V)\otimes S(V)&\longrightarrow&V\\
x_1\ldots x_k \otimes y_1\ldots y_l&\longrightarrow& \lfloor x_1\ldots x_k,y_1\ldots y_l\rfloor.
\end{array}\right.$$
For all $k,l$, let us put $\lfloor-,-\rfloor_{k,l}=\lfloor-,-\rfloor_{\mid S^k(V) \otimes S^l(V)}$. 
We shall say that $V$ is a $\petitbinfini$ algebra if the following properties are satisfied:
\begin{itemize}
\item $\lfloor-,-\rfloor_{0,1}=\lfloor-,-\rfloor_{1,0}=\begin{cases}
Id_V \mbox{ if }k=1,\\
0\mbox{ otherwise}.\end{cases}$
\item We shall need the following notations: let $u=x_1\ldots x_k\in S^k(V)$ and $v=y_1\ldots y_l \in S^l(V)$. Let $I\subseteq \{1,\ldots,k+l\}$. 
We put $I=\{i_1,\ldots,i_p,j_1,\ldots,j_q\}$, with 
$$i_1<\ldots<i_p\leq k<j_1<\ldots <j_q.$$ 
We then put:
$$\lfloor u,v\rfloor_I=\lfloor x_{i_1}\ldots x_{i_p},y_{j_1}\ldots y_{j_q}\rfloor.$$
For all $u=x_1\ldots x_k \in S^k(V)$, $v=y_1\ldots y_l\in S^l(V)$, $w=z_1\ldots z_m \in S^m(V)$,
\begin{align}
\sum_{\substack{\{I_1,\ldots,I_p\}\\ \mbox{\scriptsize partition of } [l+m]}}\lfloor u,\lfloor v,w\rfloor_{I_1}\ldots \lfloor v,w\rfloor_{I_p}\rfloor
&=\sum_{\substack{\{I_1,\ldots,I_p\}\\ \mbox{\scriptsize partition of } [k+l]}}\lfloor\lfloor u,v\rfloor_{I_1}\ldots \lfloor u,v\rfloor_{I_p},w\rfloor.
\end{align}\end{itemize}
The operad of $\petitbinfini$ algebras is denoted by $\petitbinfini$.
\end{defi}

\begin{theo} \label{theo20}
Let $V$ be a vector space.  Let $bialg(V)$ be the set of products $\star$ on $T(V)$, making $(S(V),\star,\Delta)$ a bialgebra.
 Let $b_\infty(V)$ be the set of $\petitbinfini$ structures on $V$. These two sets are in bijections,
via the maps:
\begin{align*}
\phi_V&:\left\{\begin{array}{rcl}
b_\infty(V)&\longrightarrow&bialg(V)\\
\lfloor-,-\rfloor&\longrightarrow&\displaystyle \star\mbox{ defined by }
u\star v=\sum_{\substack{\{I_1,\ldots,I_p\}\\ \mbox{\scriptsize partition of } [k+l]}}\lfloor u,v\rfloor_{I_1}\ldots \lfloor u,v\rfloor_{I_p}
\end{array}\right.\\
\psi_V&:\left\{\begin{array}{rcl}
bialg(V)&\longrightarrow&b_\infty(V)\\
\star&\longrightarrow&\lfloor-,-\rfloor\mbox{ defined by } \lfloor u,v\rfloor=\pi(u\star v),
\end{array}\right.\end{align*}
where $\pi$ is the canonical projection on $V$.
\end{theo}

The proof of the following theorem is similar to the proof of theorem \ref{theo2}. In particular, its proof uses the following lemma:

\begin{lemma} \label{lem21}
Let $C$ be a connected coalgebra and let $\phi,\psi:C\longrightarrow S(V)$ be two coalgebra morphisms. Then 
$\phi=\psi$ if, and only if, $\pi \circ \phi=\pi\circ \psi$.
\end{lemma}

\textbf{Example.} Let $x_1,x_2,y_2,y_2\in V$.
\begin{align*}
x_1\star y_1&=x_1y_1+\lfloor x_1,y_1\rfloor,\\
x_1\star y_1y_2&=x_1y_1y_2+\lfloor x_1,y_1\rfloor y_2+\lfloor x_1,y_2\rfloor y_1+\lfloor x_1,y_1y_2\rfloor\\
x_1x_2\star y_1y_2&=x_1x_2y_1y_2+\lfloor x_1,y_1\rfloor x_2y_2+\lfloor x_1,y_2\rfloor x_2y_1+\lfloor x_1,y_1y_2\rfloor x_2\\
&+\lfloor x_2,y_1\rfloor x_1y_2+\lfloor x_2,y_2\rfloor x_1y_1+\lfloor x_2, y_1y_2\rfloor x_1\\
&+\lfloor x_1x_2, y_1\rfloor y_2+\lfloor x_1x_2,y_2\rfloor y_1+\lfloor x_1x_2,y_1y_2\rfloor\\
&+\lfloor x_1,y_1\rfloor\lfloor x_2,y_2\rfloor+\lfloor x_1,y_2\rfloor\lfloor x_2,y_1\rfloor.
\end{align*}

\textbf{Remarks.} \begin{enumerate}
\item The coalgebra $(S(V),\Delta)$ is connected so, for any $\star\in bialg(V)$, $(S(V),\star,\Delta)$ is a Hopf algebra.
\item Any vector space $V$ admits a trivial $\petitbinfini$ structure, defined by:
$$\lfloor-,-\rfloor_{k,l}=\begin{cases}
Id_V\mbox{ if $(k,l)=(1,0)$ or $(0,1)$},\\
0\mbox{ otherwise}.
\end{cases}$$
The associated product is the usual one of $S(V)$.
\end{enumerate}

\subsection{Associative-commutative algebras}

\begin{defi}\begin{enumerate}
\item An associative-commutative algebra is a $2$-associative algebra $(A,\star,m)$, such that $m$ is commutative.
 The operad of associative-commutative algebras is denoted by $\ascom$. It is generated by $m$ and $\star$, both in $\ascom(2)$, with the relations:
\begin{align*}
m^{(12)}&=m,&m\circ_1m&=m\circ_2m,&
\star\circ_1\star&=\star\circ_2\star.
\end{align*}
\item An associative-commutative algebra is a family $(A,\star,m,\Delta)$, such that $(A,\star,m)$ is an associative-commutative algebra,
and both $(\mathbf{U}A,\star,\Delta)$ and $(\mathbf{U}A,m,\Delta)$ are bialgebras.
\item For all $n\geq 1$, we denote by $\primascom(n)$ the subspace of elements $p\in \ascom(n)$ 
such that for any associative-commutative bialgebra $A$:
$$\forall a_1,\ldots,a_n \in Prim(A),\: p.(a_1,\ldots, a_n)\in Prim(A).$$
This is a suboperad of $\ascom$.
\end{enumerate}\end{defi}

As for $\binfini$ algebras:

\begin{theo}
The operads $\petitbinfini$ and $\primascom$ are isomorphic.
\end{theo}


\subsection{Quotients of $\petitbinfini$}

\label{sect223}

Recall that a (right) pre-Lie algebra --also called a right symmetric algebra or a Vinberg algebra-- is a vector space $V$
together with a bilinear product $\bullet$ such that:
$$\forall x,y,z\in V,\: x\bullet (y\bullet z)-(x\bullet y)\bullet z=x\bullet (z\bullet y)-(x\bullet z)\bullet y.$$

\begin{prop}\label{prop24}
The quotient of $\petitbinfini$ by the operadic ideal generated by $\lfloor-,-\rfloor_{k,l}$, $k\geq 2$, is isomorphic to the operad $\prelie$,
generated by $\bullet \in \prelie(2)$ and the relation:
$$\bullet\circ_1 \bullet-\bullet\circ_2\bullet=(\bullet\circ_1 \bullet-\bullet\circ_2\bullet)^{(23)}.$$
\end{prop}

\begin{proof}
We denote by $\petitbinfini'$ this quotient of $\petitbinfini$. If $V$ is a $\petitbinfini'$ algebra and $x,y,z\in V$,
in $S(V)$, $x\star y=xy+\lfloor x,y\rfloor$. So:
\begin{align*}
\lfloor \lfloor x,y \rfloor,z\rfloor-\lfloor x,\lfloor y,z\rfloor\rfloor&=\lfloor x\star y,z\rfloor+\lfloor xy,z\rfloor-\lfloor x,y\star z\rfloor-\lfloor x,yz\rfloor\\
&=\lfloor x\star y,z\rfloor-\lfloor x,y\star z\rfloor-\lfloor x,yz\rfloor\\
&=\lfloor x,yz\rfloor.
\end{align*}
As $yz=zy$ in $S^2(V)$, $\lfloor-,-\rfloor_{1,1}$ is pre-Lie. We obtain an operad morphism:
$$\Phi:\left\{\begin{array}{rcl}
\prelie&\longrightarrow&\mathbf{b}'_\infty\\
\bullet&\longrightarrow&\lfloor-,-\rfloor_{1,1}.
\end{array}\right.$$
Moreover, in a $\petitbinfini'$ algebra $V$, if $x,y_1,\ldots,y_{k+1}\in V$:
\begin{align*}
y_1\ldots y_k\star y_{k+1}&=y_1\ldots y_{k+1}+\sum_{i=1}^k y_1\ldots y_{i-1} \lfloor y_i,y_{k+1}\rfloor y_{i+1}\ldots y_k,
\end{align*}
so:
\begin{align*}
\lfloor x,y_1\ldots y_{k+1}\rfloor&=\lfloor x,y_1\ldots y_k\star y_{k+1}\rfloor
-\sum_{i=1}^k \lfloor x, y_1\ldots y_{i-1} \lfloor y_i,y_{k+1}\rfloor y_{i+1}\ldots y_k\rfloor\\
&=\lfloor \lfloor x,y_1\ldots y_k\rfloor,y_{k+1}\rfloor-\sum_{i=1}^k \lfloor x, y_1\ldots y_{i-1} \lfloor y_i,y_{k+1}\rfloor y_{i+1}\ldots y_k\rfloor.
\end{align*}
An easy induction proves that $\mathbf{b}'_\infty$ is generated by $\lfloor-,-\rfloor_{1,1}$, so $\Phi$ is surjective. \\

Let $(V,\bullet)$ be a pre-Lie algebra. The Oudom-Guin construction \cite{Oudom} allows to construct a product $\star$ on $S(V)$,
making it a Hopf algebra, isomorphic to the enveloping algebra of $V$. So $V$ is $\petitbinfini$. Moreover, for all $x_1,\ldots,x_k,y_1,\ldots,y_l\in V$,
with the notations of \cite{Oudom}:
$$\pi(x_1\ldots x_k\star y_1\ldots y_l)=\begin{cases}
0\mbox{ if }k\neq 1,\\
x_1\bullet y_1\ldots y_l \mbox{ if }k=1.
\end{cases}$$
So $V$ is a $b_\infty'$-algebra: we obtain an operad morphism:
$$\phi':\left\{\begin{array}{rcl}
\petitbinfini'&\longrightarrow&\prelie\\
\lfloor-,-\rfloor_{1,1}&\longrightarrow &\bullet.
\end{array}\right.$$
Then  $\Phi'\circ \Phi=Id_\prelie$, so $\Phi$ is injective. \end{proof}\\

\textbf{Remark.} As noticed in the preceding proof, if $(V,\bullet)$ is a pre-Lie algebra, $\star=\phi_V(\bullet)$ is the Oudom-Guin construction.
In particular, its $\petitbinfini$ brackets can be inductively computed:
\begin{itemize}
\item For all $x\in V$, $\lfloor x,1\rfloor=x\bullet 1=x$.
\item For all $x,y\in V$, $\lfloor x,y\rfloor=x\bullet y$.
\item For all $x,x_1,\ldots,x_k \in V$,
\begin{align*}
\lfloor x,x_1\ldots x_k\rfloor&=\lfloor\lfloor x, x_1\ldots x_{k-1}\rfloor, x_k\rfloor-\sum_{i=1}^{k-1}\lfloor x, x_1\ldots \lfloor x_i, x_k\rfloor \ldots x_{k-1}\rfloor.
\end{align*}
$\lfloor x,x_1\ldots x_k\rfloor$ is denoted by $x\bullet x_1\ldots x_k$ in \cite{Oudom}.
\end{itemize}

\begin{cor}
The quotient of $\petitbinfini$ by the operadic ideal generated by $\lfloor-,-\rfloor_{k,l}$, $k\geq 2$ or $l\geq 2$, is isomorphic to the operad $\ass$.
\end{cor}

\begin{proof} The antecedent by $\Phi$ of the element $\lfloor-,-\rfloor_{1,2}$ is $\bullet\circ_1 \bullet-\bullet \circ_2\bullet$.
So this quotient of $\petitbinfini$ is isomorphic to the quotient of $\prelie$ by the element $\bullet\circ_1 \bullet-\bullet \circ_2\bullet$.
This quotient is generated by the class $m$ of $\bullet$ and the relation $m\circ_1m=m\circ_2m$, so it is $\ass$. \end{proof}

\subsection{From $\binfini$ algebras to $\petitbinfini$ algebras}

\textbf{Notations.} We denote by $coS(V)$ the subalgebra of $(T(V),\shuffle)$ generated by $V$; it is a Hopf subalgebra of $(T(V),\shuffle,\Delta_{dec})$.
As the characteristic of the base field $\K$ is zero, this Hopf algebra is isomorphic to $S(V)$ via the morphism:
$$\iota_V:\left\{\begin{array}{rcl}
S(V)&\longrightarrow&coS(V)\\
x_1\ldots x_k&\longrightarrow&x_1\shuffle \ldots \shuffle x_k.
\end{array}\right.$$

\begin{lemma}
\begin{enumerate}
\item $coS(V)$ is the greatest cocommutative subcoalgebra of $T(V)$.
\item Let $*\in Bialg(V)$. The subalgebra of $(T(V),*)$ generated by $V$ is $coS(V)$.
\end{enumerate}\end{lemma}

\begin{proof} 1. $coS(V)$ is indeed a cocommutative coalgebra of $T(V)$. Let $C$ be a cocommutative coalgebra of $T(V)$.
The subalgebra $C'$ generated for the shuffle product by $C$ is then a cocommutative subbialgebra of $T(V)$. As $T(V)$ is connected as a coalgebra,
$C'$ also is: by the Cartier-Quillen-Milnor-Moore theorem, $C'$ is generated by $Prim(C')$. As $Prim(C')\subseteq Prim(T(V))=V$,
$C\subseteq C'\subseteq coS(V)$. 

2. As $*$ is a coalgebra morphism and $coS(V)^{\otimes 2}$ is a cocommutative subcoalgebra of $T(V)^{\otimes 2}$,
$coS(V)*coS(V)$ is a cocommutative subcoalgebra of $T(V)$. By the first point, it is included in $coS(V)$.
So $coS(V)$ is a subalgebra of $(T(V),*)$, containing $V$. Denoting by $A$ the subalgebra of $T(V),*)$ generated by $V$,
$A\subseteq coS(V)$. Moreover, for all $x_1,\ldots,x_k \in V$, by theorem \ref{theo2}:
\begin{align*}
x_1*\ldots*x_k&=x_1\shuffle \ldots \shuffle x_k+\mbox{a sum of words of length $<k$}.
\end{align*}
By a triangularity argument, $coS(V) \subseteq A$. \end{proof}\\

Let $V$ be a $\binfini$ algebra, and let $*=\phi_V(\langle-,-\rangle)$. By the preceding lemma,
$coS(V)$ is a subbialgebra of $(T(V),*,\Delta_{dec})$. Via $\iota_V$, $S(V)$ is also a bialgebra, so $V$ is a $\petitbinfini$ algebra.
This structure is given by:
$$\lfloor x_1\ldots x_k,y_1\ldots y_l\rfloor=\langle x_1\shuffle \ldots \shuffle x_k,y_1\shuffle \ldots \shuffle y_l\rangle.$$
At the operadic level, we obtain:

\begin{prop}
There is an operad morphism $\Phi$ from $\petitbinfini$ to $\binfini$, such that:
$$\forall k,l\geq 0,\:\Phi(\lfloor-,-\rfloor_{k,l})=\sum_{\sigma\in \mathfrak{S}_k,\tau\in \mathfrak{S}_l}\langle-,-\rangle_{k,l}^{\sigma\otimes \tau}.$$
\end{prop}

We finally obtain a commutative diagram of operads:
$$\xymatrix{\petitbinfini\ar[rr]^\Phi\ar@{->>}[d]&&\binfini \ar@{->>}[d]\\
\prelie \ar[rr] \ar@{->>}[rd]&&\brace\ar@{->>}[ld]\\
&\ass\ar@{->>}[d]&\\
&\com&}$$

\subsection{Dual construction}

\begin{defi}
Let $V$ be a graded $\petitbinfini$ algebra. We shall say that $V$ is $0$-bounded if:
\begin{itemize}
\item For all $n\geq 0$, $V_n$ is finite-dimensional.
\item For all $m, n\geq 0$, there exists $B(m, n)\geq 0$ such that:
$$k+l\geq B(m, n)\Longrightarrow \lfloor S^k(V_0)S(V)_m,S^l(V_0) S(V)_n\rfloor=(0).$$
\end{itemize}\end{defi}

\textbf{Examples.}\begin{enumerate}
\item If $V$ is connected, that is to say if $V_0=(0)$, then $V$ is $0$-bounded, with $B(m, n)=0$ for all $m, n$.
\item If $V$ is associative, then $\langle-,-\rangle_{k,l}=0$ if $k\geq 2$ or $l\geq 2$. Consequently, $V$ is $0$-bounded, with 
$B(m, n)=2$ for all $m, n$.
\item If $V$ is a $0$-bounded $\binfini$ algebra, it is also a $0$-bounded $\petitbinfini$ algebra.
\end{enumerate}

Let us assume that $V$ is $0$-bounded. We identify $S(V^*)$ with a subspace of $S(V)^*$ by the pairing $\ll-,-\gg'$ induced by the
pairing between $V^*$ and $V$. More precisely, let us choose a basis $(x_i)_{i\in I}$ of $V$, made of homogeneous elements of $V$.
The dual basis of $V^*$ is denoted by $(f_i)_{i\in I}$. We shall need the following notations:
\begin{itemize}
\item We denote by $\Lambda$ the set of sequences of positive integers $(\alpha_i)_{i\in I}$ whose support is finite.
\item For all $\alpha=(\alpha_i)_{i\in I}$, we put:
\begin{align*}
x^\alpha&=\prod_{i\in I}x_i^{\alpha_i},&
f^\alpha&=\prod_{i\in I}f_i^{\alpha_i},&
\alpha!&=\prod_{i\in I} \alpha_i!
\end{align*}
\end{itemize}
Then $(x^\alpha)_{\alpha \in \Lambda}$ is a basis of $S(V)$, $(f^\alpha)_{\alpha \in \Lambda}$ is a basis of $S(V^*)$, and the pairing is given by:
$$\ll f^\alpha,x^\beta\gg=\alpha!\delta_{\alpha,\beta}.$$
Note that for all $F,G \in S(V^*),$ $X,Y\in S(V)$:
\begin{align*}
\ll \Delta(F), X\otimes Y\gg&=\ll F,XY\gg,&
\ll F\otimes G,\Delta(Y)\gg&=\ll FG,X\gg.
\end{align*}

\begin{prop}
Let $V$ be a $0$-bounded $\petitbinfini$ algebra. 
We define a coproduct $\Delta_*$ on $S(V^*)$ as the unique algebra morphism such that for all $f\in V^*$, for all $X,Y \in S(V)$,
$$\ll \Delta_*(f),X\otimes Y\gg=\ll f,\lfloor X,Y\rfloor \gg.$$
Then $(S(V^*),\Delta_*)$ is a graded bialgebra. Moreover, for all $F,G\in S(V^*),$ $X,Y\in S(V)$:
\begin{align*}
\ll \Delta_*(F),X\otimes Y\gg&=\ll F,X*Y\gg,&\ll F\otimes G,\Delta(Y)\gg&=\ll FG,X\gg.
\end{align*}
In other words, $\ll-,-\gg$ is a Hopf pairing. \end{prop}

\begin{proof} Similar to the proof of proposition \ref{prop17}. \end{proof}\\

\textbf{Remarks.} \begin{enumerate}
\item If $V$ is a $0$-bounded $\binfini$ algebra, then $(S(V^*),m,\Delta_*)$ is the abelianization of the bialgebra $(T(V^*),m_{conc},\Delta_*)$.
\item If $V$ is connected, then $(S(V^*),m,\Delta_*)$ is a graded, connected Hopf algebra, isomorphic to the graded dual of $(S(V),*,\Delta)$.
\end{enumerate}

\begin{cor} Let $V$ be a $0$-bounded $\petitbinfini$ algebra.
\begin{itemize}
\item $S(V_0^*)$ is a subbialgebra of $S(V^*)$. 
\item Let $J_0$ be the ideal of $S(V^*)$ generated by the elements $f\in V_0^*$.
Then $J_0$ is a biideal of $S(V^*)$, and $S(V^*)/J_0$ is a graded, connected Hopf algebra, isomorphic to the graded dual of $S(V_+)$.
\end{itemize}\end{cor}

\begin{proof} Similar to the proof of corollary \ref{cor18}. \end{proof}\\

\textbf{Remark.} If $V$ is a $0$-bounded $\binfini$ algebra, then $(S(V^*),m,\Delta_*)$ is the abelianization of $(T(V^*),m_{conc},\Delta_*)$,
whereas $S(V^*)/J_0$ is the abelianization of $T(V^*)/I_0$.

\subsection{Associated groups and monoids}

\begin{theo} \label{theo31}
Let $(V,\lfloor-,-\rfloor)$ be a $0$-bounded $\petitbinfini$ algebra. Then $\overline{V}$ is given a monoid structure with the product defined by:
$$\forall x,y\in \overline{V},\: x\lozenge y=\lfloor e^x,e^y\rfloor.$$
It is isomorphic to the monoid of characters of  both $(S(V^*),m,\Delta_*)$ and $(T(V^*),m_{conc},\Delta_*)$.
\end{theo}

\begin{proof} As $S(V^*)$ is the abelianization of $T(V^*)$, these two bialgebras have the same monoid of characters.
Let us determine the monoid of characters of $S(V^*)$. For all $\alpha,\beta \in \Lambda$, we put:
$$\lfloor x^\alpha,y^\beta\rfloor=\sum_{i\in I} a^{(i)}_{\alpha,\beta} x_i.$$
If $j\in I$, for all $\alpha,\beta\in \lambda$:
\begin{align*}
\ll\Delta_*(f_j),x^\alpha\otimes x^\beta\gg&=\ll f_j,x^\alpha*x^\beta\gg\\
&=\ll f_j,\pi(x^\alpha*x^\beta)\gg\\
&=\ll f_j,\lfloor x^\alpha,y^\beta\rfloor\gg\\
&=a^{(j)}_{\alpha,\beta}.
\end{align*}
Hence:
$$\Delta_*(f_j)=\sum_{\alpha,\beta\in \Lambda} \frac{a^{(j)}_{\alpha,\beta}}{\alpha!\beta!} f^\alpha\otimes f^\beta.$$
The monoid of characters of $S(V^*)$ is identified, as a set, with the (complete) dual $(V^*)^\circledast$ of $V^*$,
that is to say:
$$\left(\bigoplus_{n\geq 1} V_n^*\right)^\circledast
=\prod_{n\geq 1} V_n^{**}=\prod_{n\geq 1}V_n=\overline{V}.$$
The identification is through the map:
$$\varphi:\left\{\begin{array}{rcl}
\overline{V}&\longrightarrow&Char(S(V^*))\\
x&\longrightarrow&\varphi_x:\left\{\begin{array}{rlc}
S(V^*)&\longrightarrow&\K\\
f^\alpha&\longrightarrow&\displaystyle \prod_{i\in I} \ll f_i,x\gg^{\alpha_i}.
\end{array}\right.\end{array}\right.$$
Let $x=\sum \lambda_i x_i$ and $y=\sum \mu_ix_i \in \overline{V}$. For all $j\in I$:
\begin{align*}
&\varphi_x*\varphi_y(f_j)&
&\varphi_{\lfloor e^x,e^y\rfloor}(f_j)\\
&=(\varphi_x\otimes \varphi_y)\circ \Delta_*(f_j)&
&=\ll f_j,\lfloor e^x,e^y\rfloor\gg\\
&=\sum_{\alpha,\beta\in \Lambda} \frac{a^{(j)}_{\alpha,\beta}}{\alpha!\beta!} \varphi_x(f^\alpha)\varphi_y(f^\beta)&
&=\sum_{\alpha,\beta \in \Lambda} \prod_{i\in I} \frac{\lambda_i^{\alpha^i}}{\alpha_i!}\frac{\mu_i^{\beta^i}}{\beta_i!}\ll f_j,\lfloor x^\alpha,y^\beta\rfloor\gg\\
&=\sum_{\alpha,\beta\in \Lambda} \frac{a^{(j)}_{\alpha,\beta}}{\alpha!\beta!} \prod_{i\in I}\lambda_i^{\alpha_i}\prod_{i\in I}\mu_i^{\beta_i};&
&=\sum_{\alpha,\beta\in \Lambda} \frac{a^{(j)}_{\alpha,\beta}}{\alpha!\beta!} \prod_{i\in I}\lambda_i^{\alpha_i}\prod_{i\in I}\mu_i^{\beta_i}.
\end{align*}
As $\varphi_x*\varphi_y$ and $\varphi_{\lfloor e^x,e^y\rfloor}$ are characters which coincide on $V^*$, they are equal.
Through the bijection $\varphi$, we obtain the  monoid structure on $\overline{V}$. \end{proof}\\

\textbf{Remark.} If $V$ is connected, then $S(V^*)$ is a Hopf algebra, and in this case $(\overline{V},\lozenge)$ is a group.

\begin{cor}\label{cor32}\begin{enumerate}
\item Let $(V,\bullet)$ be a $0$-bounded pre-Lie algebra. Then $\overline{V}$ is a monoid, with the product defined by:
$$\forall x,y\in \overline{V},\: x\lozenge y=y+x\bullet e^y.$$
\item Let $(V,\cdot)$ be a graded associative algebra. Then $\overline{V}$ is a monoid, with the product defined by:
$$\forall x,y \in \overline{V},\: x\lozenge y=x+y+x\cdot y.$$
\end{enumerate}\end{cor}

\begin{proof} 1. Here, $\lfloor-,-\rfloor_{k,l}=0$ if $k\geq 2$. So:
\begin{align*}
x\lozenge y&=\sum_{k\geq 0} \frac{1}{k!}\lfloor x^k,e^y\rfloor=\lfloor 1,e^y\rfloor+\lfloor x,e^y\rfloor
=\sum_{l\geq 0} \frac{1}{l!} \lfloor 1,y^l\rfloor+x\bullet e^y=y+x\bullet e^y.
\end{align*}

2. Here, $\lfloor-,-\rfloor_{k,l}=0$, except for $(k,l)=(0,1)$, $(1,0)$ and $(1,1)$. Hence:
\begin{align*}
x\lozenge y&=\sum_{k,l\geq 0}\frac{1}{k!l!}\lfloor x^k,y^l\rfloor_{k,l}=\lfloor x,1\rfloor+\lfloor 1,y\rfloor+\lfloor x,y\rfloor=x+y+x\cdot y.
\end{align*}
Note that, $\overline{V}$ can be identified with the monoid of elements of the associative, unitary algebra $\K\oplus \overline{V}$, 
whose constant terms are equal to $1$.\end{proof}\\

Using the graduation:
\begin{cor}
Let $V$ be a $0$-bounded $\petitbinfini$ algebra. 
\begin{enumerate}
\item Then $V_0$ and $\overline{V_+}$ are submonoids of $(\overline{V},\lozenge)$.
Moreover, $(\overline{V_+},\lozenge)$ is a group, isomorphic to the group of characters of both $(S(V_+^*),m,\Delta_*)$ and $(T(V^*_+),m,\Delta_*)$.
\item Let $x=x_0+x_+\in \overline{V}$, with $x_0\in V_0$ and $x_+\in \overline{V_+}$.
Then $x$ is a unit of $\overline{V}$ if, and only if, $x_0$ is a unit in $V_0$.
\end{enumerate}
\end{cor}

\begin{proof} 1. Immediate.

2. $\Longrightarrow$. The canonical projection on $V_0$ is a monoid morphism from $\overline{V}$ to $V_0$, which implies the first point.\\
$\Longleftarrow$. Let $y_0$ be the inverse of $x_0$ in $V_0$. We put $y=x\lozenge y_0$. Then:
$$y_0=(\lfloor e^x,e^{y_0}\rfloor)_0=\lfloor e^{x_0},e^{y_0}\rfloor=x_0\lozenge y_0=0.$$
So $y\in \overline{V}_+$, so is a unit of $\overline{V}$; by composition, $x$ is a unit of $\overline{V}$. \end{proof}

\subsection{pre-Lie modules}

\begin{defi}
Let $(V,\bullet)$ be a pre-Lie algebra. A pre-Lie module over $V$ is a vector space $M$, with a map:
$$\leftarrowtail:\left\{ \begin{array}{rcl}
M\otimes V&\longrightarrow&M\\
m\otimes x&\longrightarrow&m\leftarrowtail x,
\end{array}\right.$$
such that for all $m\in M$, $x_1,x_2 \in V$:
\begin{align*}
(m\leftarrowtail a_1)\leftarrowtail a_2-(m\leftarrowtail a_2)\leftarrowtail a_1
&=m\leftarrowtail(a_1\bullet a_2-a_2\bullet a_1).
\end{align*}\end{defi}

\textbf{Example.} If $V$ is a pre-Lie algebra, $(V,\bullet)$ is a pre-Lie module over itself.\\

\textbf{Remarks.} \begin{enumerate}
\item A pre-Lie module over $V$ is a module over the Lie algebra associated to $V$, so is a module over the Hopf algebra
$(S(V),*,\Delta)$, where $*=\phi_V(\bullet)$. The action is given in the following way:
\begin{itemize}
\item For all $m\in M$, $m\leftarrowtail 1=m$.
\item For all $m\in M$, $x_1,\ldots,x_k \in V$:
\begin{align*}
m\leftarrowtail x_1\ldots x_k&=(m\leftarrowtail x_1\ldots x_{k-1})\leftarrowtail x_k-\sum_{i=1}^{k-1}m\leftarrowtail (x_1\ldots (x_i\bullet x_k)\ldots x_{k-1}).
\end{align*}\end{itemize}
\item Let $V$ be a brace algebra and $M$ be a brace module over $V$. Then $V$ is also a pre-Lie algebra, with $\bullet=\langle-,-\rangle_{1,1}$.
Moreover, for all $m\in M$, $x_1,x_2\in V$:
\begin{align*}
m\leftarrow (x_1*x_2)&=m\leftarrow(x_1x_2+x_2x_2+x_1\bullet x_2)=(m\leftarrow x_1)\leftarrow x_2,\\
(m\leftarrow x_1)\leftarrow x_2-(m\leftarrow x_2)\leftarrow x_1&=m\leftarrow(x_1\bullet x_2-x_2\bullet x_1).
\end{align*}
So the restriction $\leftarrowtail$ of $\leftarrow$ to $M\otimes V$ makes $M$ a pre-Lie module over $V$. By the isomorphism
between $S(A)$ and $coS(A)$, for all $x_1,\ldots,x_k\in V$, $m\in M$:
\begin{align*}
m\leftarrowtail x_1\ldots x_k&=m\leftarrow (x_1\shuffle \ldots \shuffle x_k).
\end{align*} \end{enumerate}

\begin{defi}
Let $V$ be a $0$-bounded pre-Lie algebra and $M$ be a graded pre-Lie module over $V$. We shall say that $M$ is $0$-bounded if
for all $k,l \geq 0$, there exists $B(k,l)\geq 0$ such that:
$$p>B(k,l)\Longrightarrow M_k\leftarrowtail S^p(V_0)S(V)_l=(0).$$
\end{defi}

Note that if $V$ is connected, then any graded pre-Lie module over $V$ is $0$-bounded, with $B(k,l)=0$ for all $k,l$.

\begin{prop} \label{prop36}
 Let $V$ be a $0$-bounded pre-Lie algebra and $M$ be a $0$-bounded pre-Lie  module over $V$.
$\overline{M}$ is a module over the monoid $(\overline{V},\lozenge)$, with the action defined by:
\begin{align*}
\forall m\in \overline{M},\: \forall x\in \overline{V},\:
m\vartriangleleft x&=x\leftarrowtail e^y.
\end{align*} 
By restriction, it is also a module over the group $(\overline{V_+},\lozenge)$. 
\end{prop}

\begin{proof} By transposition of the action of $(S(V);*)$ on $M$, we obtain thanks to the $0$-bounded condition a
coaction of $(S(V^*),\Delta_*)$ on $M^*$; consequently, the dual of $M^*$, identified with $\overline{M}$, 
becomes a module over the monoid of characters of $(S(V^*),m,\Delta_*)$, identified with $(\overline{V},\lozenge)$.
The end of the proof is similar to the proof of theorem \ref{theo31}. \end{proof}

\chapter{Brace and pre-Lie structures on operads}

\section*{Introduction}

We now study the brace and pre-Lie structure on an operad $\bfP$ induced by the operadic composition
(proposition \ref{prop37} and corollary \ref{cor38}).
We have seen in the preceding chapter that these structures imply a product on $T(\bfP)$ making it a graded, non connected
dendriform Hopf algebra, named $\bfD_\bfP$ (proposition \ref{prop40}). 
By the $0$-boundedness condition, we construct a dual bialgebra $\bfD_\bfP^*$.
Considering a connected Hopf subalgebra of $\bfD_\bfP$, we obtain a graded, connected Hopf algebra on $T(\bfP_+)$,
named $\bfB_\bfP$, and its graded dual $\bfB^*_\bfP$, as a quotient of $\bfD_\bfP^*$.\\

Using the pre-Lie product induced by the brace structure, we obtain a graded, non connected Hopf algebra $D_\bfP$, which underlying coalgebra
is $S(coinv \bfP)$ with its usual coproduct; a graded, connected Hopf algebra $B_\bfP$, which underlying coalgebra is $S(coinv\bfP_+)$
with its usual coproduct; and bialgebras $D_\bfP^*$ and $B_\bfP^*$, in duality with the preceding ones. 
All these objects admit colored versions by any vector space $V$,
see propositions \ref{prop42}, \ref{prop43} and \ref{defi44}.
As $D_\bfP^*$ and $B_\bfP^*$ are commutative, they can be seen as coordinates bialgebras of a monoid:
these monoids are described in proposition \ref{prop54}, with the help of the pre-Lie product of $\bfP$, as in theorem \ref{theo31};
they appeared in \cite{ChapotonLivernet}.

A specially interesting case is obtained by operads $\bfP$ equipped with an operad morphism
$\theta_\bfP:\petitbinfini\longrightarrow \bfP$. In this case, for any vector space $V$, the coalgebra $S(F_\bfP(V))$,
where $F_\bfP(V)$ is the free $\bfP$-algebra generated by $V$, becomes a graded, connected Hopf algebra denoted by $A_\bfP(V)$.
We prove in theorem \ref{theo46} that $A_\bfP(V)$ and $D_\bfP(V)$ are bialgebras in interaction (definition \ref{defi41}),
that is to say that $A_\bfP(V)$ is a bialgebra in the categroy of $D_\bfP(V)$-modules;
dually, $A_\bfP^*(V)$ and $D_\bfP^*(V)$ are cointeracting bialgebras, in the sense of \cite{Manchon2},
that is to say that $A_\bfP^*(V)$ is a Hopf algebra in the category of $D_\bfP^*(V)$-comodules.

\section{Definition}

\begin{prop} \label{prop37}
Let $\bfP$ be a non-$\Sigma$ operad. We define a brace structure on $\displaystyle \bfP=\bigoplus_{n\geq 1} \bfP(n)$ by:
\begin{align*}
\forall p \in \bfP(n),p_1,\ldots,p_k\in \bfP,\: \langle p,p_1\ldots p_k\rangle=
\sum_{1\leq i_1<\ldots <i_k\leq n} p\circ_{i_1,\ldots,i_k}(p_1,\ldots,p_k).
\end{align*}
It is graded, putting the elements of $\bfP(n)$ homogeneous of degree $n-1$, and $0$-bounded. \end{prop}

\begin{proof} Let $p,p_1,\ldots,p_k,q_1,\ldots,q_l \in \bfP$. Then, using the associativity of the operadic composition:
\begin{align*}
\langle \langle p,p_1\ldots p_k\rangle,q_1\ldots q_l\rangle&=\sum \langle p\circ (I,\ldots,p_1,\ldots,p_k,\ldots,I),q_1\ldots q_l\rangle\\
&=\sum (p\circ (I,\ldots,p_1,\ldots,p_k,\ldots,I))\circ (I,\ldots,q_1,\ldots,q_l,\ldots,I)\\
&=\sum p\circ (I,\ldots, q_1,\ldots,q_{i_1},p_1\circ(I,\ldots,q_{i_1+1},\ldots,q_{i_1+i_2},\ldots,I),\ldots,\\
&\hspace{2cm} p_k\circ(I,\ldots,q_{i_1+\ldots+i_{2k-1}+1},\ldots,q_{i_1+\ldots+i_{2k}},\ldots,I),\ldots, \\
&\hspace{2cm}q_{i_1+\ldots+i_{2k+1}},\ldots,q_l,\ldots,I)\\
&=\sum_{q_1\ldots q_l=Q_1\ldots Q_{2k+1}}\langle p, Q_1\langle p_1,Q_2\rangle Q_3\ldots Q_{2k-1}\langle p_k,Q_{2k}\rangle Q_{2k+1}\rangle.
\end{align*}
So $\langle-,-\rangle$ is a brace structure on $\bfP$. \\

Let $p,p_1,\ldots,p_k$ in $\bfP$, homogeneous of respective degrees $n$, $n_1,\ldots, n_k$. Then $p\in \bfP(n+1)$ and $p_i \in \bfP(n_i+1)$ for all $i$.
Then $\langle p,p_1 \ldots p_k\rangle$ is a linear span of element of $\bfP(m)$, with:
\begin{align*}
m&=n+1-k+n_1+1+\ldots+n_k+1=n+n_1+\ldots+n_k+1.
\end{align*} 
So $\langle p,p_1\ldots p_k\rangle$ is homogeneous of degree $n+n_1+\ldots+n_k$: the brace algebra $\bfP$ is graded. \\

Let $m, n\geq 1$. If $p \in \bfP(m+1)$, and if $k>m+1$, then $\langle p,p_1\ldots p_k\rangle=0$ for all $p_1,\ldots,p_k \in \bfP$.
So $\bfP$ is $0$-bounded, with $B(m, n)=m+1$. \end{proof}\\

\textbf{Remark.} Consequently, $\bfP(1)$ is a brace subalgebra of $\bfP$. For all $p,p_1,\ldots,p_k \in\bfP(1)$,
$$\langle p,p_1\ldots p_k\rangle=\begin{cases}
p\circ p_1 \mbox{ if }k=1,\\
0\mbox{ if }k\geq 2.
\end{cases}$$
So the brace algebra $\bfP(1)$ is the associative algebra $(\bfP(1),\circ)$. \\

 By the morphism from $\prelie$ to $\brace$, we immediately obtain:

\begin{cor}\label{cor38}
Let $\bfP$ be a non-$\Sigma$ operad. It is a graded pre-Lie algebra, with:
\begin{align*}
\forall p\in \bfP(n),\: q\in \bfP,\: p\bullet q=\langle p,q\rangle=\sum_{i=1}^n p\circ_i q.
\end{align*}
Its $\petitbinfini$ brackets are given by:
\begin{align*}
\forall p\in \bfP(n),\: p_1,\ldots,p_k \in \bfP,\: \lfloor p,p_1\ldots p_k\rfloor
&=\sum_{\substack{1\leq i_1,\ldots ,i_k\leq n,\\ \mbox{\scriptsize all distincts}}} p\circ_{i_1,\ldots,i_k}(p_1,\ldots,p_k).
\end{align*}\end{cor}

\begin{proof} Indeed,  $\displaystyle \lfloor p,p_1\ldots p_k\rfloor=\langle p,p_1\shuffle \ldots \shuffle p_k \rangle 
=\sum_{\sigma \in \mathfrak{S}_k}\langle p,p_{\sigma(1)}\ldots p_{\sigma(k)}\rangle$. \end{proof}

\begin{cor}
Let $\bfP$ be an operad. Then $coinv\bfP$ is a graded pre-Lie algebra, quotient of $\bfP$.
\end{cor}

\begin{proof} Let us prove that $I=Vect(p-p^\sigma\mid p\in \bfP(n), \sigma \in \mathfrak{S}_n)$ is a pre-Lie ideal. Let $p\in \bfP(n)$,
$\sigma \in \mathfrak{S}_n$ and $q\in \bfP(m)$. There exist permutations $\sigma'_i$ such that:
\begin{align*}
(p^\sigma-p)*q&=\sum_{i=1}^n p^\sigma \circ_i q-p\circ_i q\\
&=\sum_{i=1}^n (p\circ_{\sigma(i)} q)^{\sigma'_i}-\sum_{i=1}^n p\circ_i q\\
&=\sum_{i=1}^n (p\circ_i q)^{\sigma'_{\sigma^{-1}(i)}}-p\circ_i q\in I.
\end{align*}
So $I$ is a right pre-Lie ideal. There exists permutations $\sigma''_i$ such that:
\begin{align*}
q*(p^\sigma-p)&=\sum_{i=1}^m (q\circ_i p^\sigma-q\circ_ip)=\sum_{i=1} \left((q\circ_i p)^{\sigma''_i} -q\circ_i p\right) \in I.
\end{align*}
So $I$ is also a left pre-Lie ideal. \end{proof}\\

By theorems \ref{theo13} and \ref{theo20}, and by proposition \ref{prop24}:

\begin{prop}\label{prop40}
Let $\bfP$ be an operad.
\begin{enumerate}
\item The brace structure on $\bfP$ induces a product $*=\prec+\succ$, making $\bfD_\bfP=(T(\bfP),*,\Delta_{dec})$ a dendriform bialgebra.
The graded, connected Hopf subalgebra $T(\bfP_+)$ is denoted by $\bfB_\bfP$. 
\item The pre-Lie product on $\bfP$ induces products $*$ making $(S(\bfP),*,\Delta)$, $(S(\bfP_+),*,\Delta)$,
$D_\bfP=(S(coinv\bfP),*,\Delta)$ and $B_\bfP=(S(coinv\bfP_+),*,\Delta)$ bialgebras. 
\item  There is a commutative diagram of bialgebras:
$$
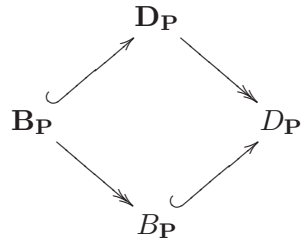
\xymatrix{&\bfD_\bfP\ar@{->>}[rd]&\\
\bfB_\bfP\ar@{^(->}[ru]\ar@{->>}[rd]&&D_\bfP\\
&B_\bfP\ar@{^(->}[ru]&}$$
\end{enumerate}\end{prop}

\textbf{Examples.} Let $p_1\in \bfP(n_1)$, $p_2\in \bfP(n_2)$, $q_1,q_2\in \bfP$. In $\bfD_\bfP$:
\begin{align*}
p_1\prec q_1&=p_1q_1+\sum_{1\leq i\leq n_1}p_1 \circ_i q_1,\\
p_1\succ q_1&=q_1p_1,\\
p_1\prec q_1q_2&=p_1q_1q_2+\sum_{1\leq i \leq n_1} (p_1\circ_i q_1)q_2+\sum_{1\leq i<j\leq n_1} p_1\circ_{i,j}(q_1,q_2),\\
p_1\succ q_1q_2&=q_1p_1q_2+q_1q_2p_1+\sum_{1\leq i \leq n_1} q_1(p_1\circ_i q_2),\\
p_1p_2\prec q_1&=p_1p_2q_1+p_1q_1p_2+\sum_{1\leq i\leq n_1} (p_1\circ_iq_1)p_2+\sum_{1\leq i\leq n_2}p_1(p_2\circ_i q_1),\\
p_1p_2\succ q_1&=q_1p_1p_2.
\end{align*}
In $S(\bfP)$:
\begin{align*}
p_1*q_1&=p_1q_1+\sum_{1\leq i\leq n_1}p_1 \circ_i q_1,\\
p_1*q_1q_2&=p_1q_1q_2
+\sum_{1\leq i \leq n_1} (p_1\circ_i q_1)q_2+\sum_{1\leq i \leq n_1} q_1(p_1\circ_i q_2)+\sum_{1\leq i\neq j\leq n_1} p_1\circ_{i,j}(q_1,q_2),\\
p_1p_2*q_1&=p_1p_2q_1+\sum_{1\leq i\leq n_1} (p_1\circ_iq_1)p_2+\sum_{1\leq i\leq n_2}p_1(p_2\circ_i q_1).
\end{align*}

\section{Interacting bialgebras from operads}

\subsection{Bialgebras in interaction}

\begin{defi}\label{defi41}
Let $A$ and $B$ be two bialgebras.
\begin{enumerate}
\item We shall say that $A$ and $B$ are in interaction if $A$ is a $B$-module-bialgebra, or equivalently if $A$
is a bialgebra in the category of $B$-modules, that is to say:
\begin{itemize}
\item $B$ is acting on $A$, via a map $\leftarrowtail:A\otimes B\longrightarrow A$.
\item  $A$ is a bialgebra in the category of $B$-modules, that is to say:
\begin{itemize}
\item For all $b\in B$, $1_A\leftarrowtail b=\epsilon(b) 1_A$.
\item For all $a\in A$, $b\in B$, $\varepsilon(a\leftarrowtail b)=\varepsilon(a)\varepsilon(b)$.
\item For all $a_1,a_2\in A$, $b\in B$, or, $(a_1a_2)\leftarrowtail b=m\left((a_1\otimes a_2)\leftarrowtail \Delta(b)\right)$ or,
with Sweedler's notation, $(a_1a_2)\leftarrowtail b=(a_1\leftarrowtail b^{(1)})(a_2\leftarrowtail b^{(2)})$.
\item For all $a\in A$, $b\in B$, $\Delta(a\leftarrowtail b)=\Delta(a)\leftarrowtail \Delta(b)$ or, with Sweedler's notation,
$\Delta(a\leftarrowtail b)=a^{(1)}\leftarrowtail b^{(1)} \otimes a^{(2)}\leftarrowtail b^{(2)}$.
\end{itemize}\end{itemize}
\item We shall say that $A$ and $B$ are in cointeraction if if $A$ is a $B$-comodule-bialgebra, or equivalently if $A$
is a bialgebra in the category of $B$-comodules, that is to say:
\begin{itemize}
\item $B$ is coacting on $A$, via a map $\rho:\left\{\begin{array}{rcl}
A&\longrightarrow&A\otimes B \\
a&\longrightarrow&\rho(a)=a_1\otimes a_0.
\end{array}\right.$
\item $A$ is a bialgebra in the category of $B$-comodules, that is to say:
\begin{itemize}
\item $\rho(1_A)=1_A\otimes 1_B$.
\item $m_{2,4}^3\circ (\rho\otimes \rho)\circ \Delta=(\Delta\otimes Id)\circ \rho$, where:
$$m_{2,4}^3:\left\{\begin{array}{rcl}
A\otimes B\otimes A\otimes B&\longrightarrow&A\otimes A\otimes B\\
a_1\otimes b_1\otimes a_2\otimes b_2&\longrightarrow&a_1\otimes a_2\otimes b_1b_2.
\end{array}\right.$$
Equivalenlty, for all $a\in A$:
$$(a^{(1)})_1 \otimes (a^{(2)})_1\otimes (a^{(1)})_0 (a^{(2)})_0=(a_1)^{(1)}\otimes (a_1)^{(2)}\otimes a_0.$$
\item For all $a,b\in A$, $\rho(ab)=\rho(a)\rho(b)$.
\item For all $a\in A$, $(\varepsilon_A\otimes Id)\circ \rho(a)=\varepsilon_A(a)1_B$.
\end{itemize}\end{itemize}\end{enumerate}\end{defi}

\textbf{Remark.} If $A$ and $B$ are in interaction, the action map $\leftarrowtail$ is a coalgebra morphism; 
if $A$ and $B$ are in cointeraction, the coaction map $\rho$ is an algebra morphism. \\

For examples and applications of cointeracting bialgebras, see
\cite{ManchonCalaque,ManchonFoissyFauvet2,FoissyEhrhart,FoissyChrom}.

\subsection{Bialgebras in interaction from operads}

\begin{prop}\label{prop42}\begin{enumerate}
\item Let $V$ be a vector space. We define the operad $\bfC_V$ by:
\begin{itemize}
\item For all $n\geq 1$, $\bfC_V(n)=End_\K(V,V^{\otimes n})$.
\item For all $f\in \bfC_V(m)$, $g\in \bfC_V(n)$ and $1\leq i\leq m$:
\begin{align*}
f\circ_i g&=(Id^{\otimes(i-1)} \otimes g\otimes Id^{\otimes (n-i)})\circ f \in \bfC_V(m+n-1).
\end{align*}
The unit is $Id_V$.
\item For all $f\in \bfC_V(n)$, $\sigma \in \mathfrak{S}_n$, and $x\in V$, if $f(x)=x_1\ldots x_n$:
\begin{align*}
f^\sigma(x)&=x_{\sigma(1)}\ldots x_{\sigma(n)}.
\end{align*}\end{itemize}
\item The tensor algebra $T(V)$ is a brace module over the brace algebra $(\bfC_V, \langle-,-\rangle)$ with,
for all $x_1,\ldots,x_n \in V$, $f_1,\ldots f_k \in \bfC_V$:
\begin{align*}
x_1\ldots x_n \leftarrow f_1\ldots f_k
&=\sum_{1\leq i_1<\ldots<i_k\leq n} x_1\ldots x_{i_1-1} f_1(x_{i_1})x_{i_1+1}\ldots x_{i_k-1}f_k(x_{i_k})x_{i_k+1}\ldots x_n. 
\end{align*}
Putting the elements of $V^{\otimes n}$ homogeneous of degree $n$, $T(V)$ is a graded brace module over $\bfC_V$.
\end{enumerate}\end{prop}

\begin{proof} We leave to the reader the proof that $\bfC_V$ is an operad and $T(V)$ is a brace module over $\bfC_V$. 
Let $x \in V$, homogeneous of degree $n$, and $f_1,\ldots, f_k \in \bfC_V$, homogeneous of respective degree $n_1,\ldots, n_k$.
Then for all $i$, $f_i \in End_\K(V,V^{\otimes (n_i+1)})$, so:
\begin{align*}
x\leftarrow f_1\ldots f_k\in V^{\otimes (n-k+n_1+1+\ldots+n_k+1)}=V^{\otimes (n+n_1+\ldots+n_k)},
\end{align*}
so it is homogeneous of degree $n+n_1+\ldots+n_k$: $T(V)$ is a graded brace module. \end{proof}\\

\textbf{Remark.} If $V$ is finite-dimensional, via the transposition, $\bfC_V$ is isomorphic to $\bfL_{V^*}$.
More generally, the transposition defines an injective operad morphism from $\bfC_V$ to $\bfL_{V^*}$.

\begin{prop}\label{prop43}
Let $\bfP$ be an operad and $V$ be a vector space. 
\begin{enumerate}
\item The following space is a graded brace module over the brace algebra associated to the operad $\bfP \otimes \bfC_V$:
$$M=\bigoplus_{n=1}^\infty \bfP(n)\otimes V^{\otimes n}.$$
\item By restriction, $M$ is a pre-Lie module on $\bfP\otimes \bfC_V$. 
This structure induces a graded pre-Lie $coinv(\bfP\otimes \bfC_V)$-module structure over the vector space $F_\bfP(V)$, 
such that, for all $p\in \bfP(k)$, for all $x_1,\ldots,x_k \in V$, for all $q\in \bfP(n)$, $f\in End_\K(V,V^{\otimes n})$:
\begin{align*}
p.(x_1\ldots x_k) \leftarrowtail \overline{q\otimes f}&=\sum_{i=1}^kp\circ_i q. (x_1\ldots x_{i-1}f(x_i)x_{i+1}\ldots x_n).
\end{align*} \end{enumerate}\end{prop}

\begin{proof} 1. As $\bfP$ is a $\N$-graded brace module over $\bfP$ and $V$ is a $\N$-graded brace module over the brace algebra $\bfC_V$,
$M$ is a graded brace module over $\bfP \otimes \bfC_V$. Moreover,  for all $p\in \bfP(k)$,
for all $x_1,\ldots,x_k \in V$, for all $q\in \bfP(n)$, $f\in End_\K(V,V^{\otimes n})$:
\begin{align*}
p\otimes x_1\ldots x_k \leftarrowtail q\otimes f&=\sum_{i=1}^k p\circ_i q\otimes (x_1\ldots x_{i-1}f(x_i)x_{i+1}\ldots x_n).
\end{align*}

2. Let  $p\in \bfP(k)$, $x_1,\ldots,x_k \in V$, $\sigma \in \mathfrak{S}_k$, $q\in \bfP(n)$, $f\in End_\K(V,V^{\otimes n})$.
There exist permutations $\sigma_i$, $\sigma'_j$ such that:
\begin{align*}
(p\otimes x_1\ldots x_k)^\sigma\leftarrowtail q \otimes f&=p^\sigma \otimes x_{\sigma(1)}\ldots x_{\sigma(k)}\leftarrowtail p\otimes f\\
&=\sum_{i=1}^k p^\sigma \circ_i q\otimes x_{\sigma(1)}\ldots f(x_{\sigma(i)})\ldots x_{\sigma(n)}\\
&=\sum_{i=1}^k (p\circ_{\sigma(i)} q \otimes x_1\ldots f(x_{\sigma(i)})\ldots x_n)^{\sigma_i}\\
&=\sum_{j=1}^k (p\circ_j q \otimes x_1\ldots f(x_j)\ldots x_n)^{\sigma'_j},
\end{align*}
so:
\begin{align*}
&((p\otimes x_1\ldots x_k)^\sigma-p\otimes x_1\ldots x_k)\leftarrowtail q\otimes f\\
&=\sum_{j=1}^k  (p\circ_j q \otimes x_1\ldots f(x_j)\ldots x_n)^{\sigma'_j}-
p\circ_j q \otimes x_1\ldots f(x_j)\ldots x_n.
\end{align*}
The pre-Lie action of $\bfP\otimes \bfC_V$ induces a pre-Lie action on the quotient of $M$ by the ideal $I$:
\begin{align*}
I&=Vect((p\otimes x_1\ldots x_k)^\sigma-p\otimes x_1\ldots x_k\mid k\geq 1, \: p\in \bfP(k), \: x_1,\ldots,x_k\in V,\: \sigma \in \mathfrak{S}_k)\\
&=Vect(p^\sigma\otimes x_{\sigma(1)}\ldots x_{\sigma(k)}-p\otimes x_1\ldots x_k
\mid k\geq 1, \: p\in \bfP(k), \: x_1,\ldots,x_k\in V,\: \sigma \in \mathfrak{S}_k)\\
&=Vect(p^\sigma\otimes x_1\ldots x_k-p\otimes x_{\sigma^{-1}(1)}\ldots x_{\sigma^{-1}(k)}
\mid k\geq 1, \: p\in \bfP(k), \: x_1,\ldots,x_k\in V,\: \sigma \in \mathfrak{S}_k).
\end{align*}
Note that the quotient $M/I=F_\bfP(V)$. 

Let  $p\in \bfP(k)$, $x_1,\ldots,x_k \in V$, $q\in \bfP(n)$, $f\in End_\K(V,V^{\otimes n})$,  $\sigma \in \mathfrak{S}_n$.
\begin{align*}
p.x_1\ldots x_k\leftarrowtail (q\otimes f)^\sigma&=\sum_{i=1}^k p\circ i q^\sigma.(x_1\ldots f^\sigma(x_i)\ldots x_k)\\
&=\sum_{i=1}^k p\circ i q.(x_1\ldots (f^\sigma)^{\sigma^-1}(x_i)\ldots x_k)\\
&=\sum_{i=1}^k p\circ i q.(x_1\ldots f(x_i)\ldots x_k)\\
&=p.x_1\ldots x_k\leftarrowtail q\otimes f.
\end{align*}
so the action of $\bfP\otimes \bfC_V$ on $F_\bfP(V)$ induces an action of $coinv(\bfP\otimes \bfC_V)$ on $F_\bfP(V)$. \end{proof}

\begin{defi}\label{defi44}
We put:
\begin{align*}
D_\bfP(V)&=D_{\bfP\otimes \bfC_V}=(S(coinv(\bfP\otimes \bfC_V)),*,\Delta),\\
B_\bfP(V)&=B_{\bfP\otimes \bfC_V}=(S(coinv(\bfP\otimes \bfC_V)_+),*,\Delta),\\
A_\bfP(V)&=S(F_\bfP(V)). 
\end{align*}\end{defi}

Note that if $V$ is one-dimensional, then $D_\bfP(V)$, respectively $B_\bfP(V)$, is isomorphic to $D_\bfP$, respectively to $B_\bfP$.

\begin{lemma}\label{lem45}
\begin{enumerate}
\item The pre-Lie action of $coinv(\bfP\otimes \bfC_V)$ on $F_\bfP(V)$ is extended to $A_\bfP(V)$:
\begin{align*}
\forall v_1,\ldots,v_k \in F_\bfP(V),\: \forall q \in coinv(\bfP \otimes \bfC_V),\:
v_1\ldots v_k\leftarrowtail q=\sum_{i=1}^k v_1\ldots (v_i \leftarrowtail q)\ldots v_k.
\end{align*}
This induces an action of $D_\bfP(V)$ on $A_\bfP(V)$, such that:
\begin{itemize}
\item For all $a\in A_\bfP(V)$, $b\in D_\bfP(V)$, $\Delta(a\leftarrow b)=a^{(1)}\leftarrow b^{(1)}\otimes a^{(2)}\leftarrow b^{(2)}$.
\item For all $b\in D_\bfP(V)$, $1\leftarrowtail b=\varepsilon(b)$.
\item For all $a_1,a_2 \in A_\bfP(V)$, $b\in D_\bfP(V)$, $a_1a_2\leftarrowtail b=(a_1\leftarrowtail b^{(1)})(a_2\leftarrowtail b^{(2)})$.
\end{itemize}
In other words, $(A_\bfP(V),m,\Delta)$ is a Hopf algebra in the category of $D_\bfP(V)$-modules. 
\item For all $p\in \bfP(n)$, $v_1,\ldots,v_n \in F_\bfP(V)$, $Q\in D_\bfP(V)$:
\begin{align*}
p.(v_1,\ldots,v_n)\leftarrowtail Q=p.(v_1\leftarrowtail Q^{(1)},\ldots,v_n\leftarrowtail Q^{(n)}).
\end{align*}
In other words, $F_\bfP(V)$ is a $\bfP$-algebra in the category of $D_\bfP(V)$-modules.
\end{enumerate}\end{lemma}

\begin{proof} 1. We consider:
\begin{align*}
X&=\{a\in A_\bfP(V)\mid \forall b\in D_\bfP(V),\: \Delta(a\leftarrowtail b)=a^{(1)}\leftarrowtail b^{(1)}\otimes a^{(2)}\leftarrowtail b^{(2)}\}.
\end{align*}
Firstly, $1\in X$: 
\begin{align*}
\forall b\in D_\bfP(V), \: \Delta(1\leftarrowtail b)&=\varepsilon(b)1\otimes 1
=\varepsilon(b^{(1)})\varepsilon(b^{(2)})1\otimes 1=1\leftarrowtail b^{(1)}\otimes 1\leftarrowtail b^{(2)}.
\end{align*}
Let $a_1,a_2 \in X$. For all $b\in D_\bfP(V)$, by the cocommutativity of $D_\bfP(V)$:
\begin{align*}
\Delta((a_1a_2)\leftarrowtail b)&=\Delta((a_1\leftarrowtail b^{(1)})(a_2\leftarrowtail b^{(2)}))\\
&=(a_1^{(1)} \leftarrowtail b^{(1)})(a_2^{(1)}\leftarrowtail b^{(3)})\otimes (a_1^{(2)} \leftarrowtail b^{(2)})(a_2^{(2)}\leftarrowtail b^{(4)})\\
&=(a_1^{(1)} \leftarrowtail b^{(1)})(a_2^{(1)}\leftarrowtail b^{(2)})\otimes (a_1^{(2)} \leftarrowtail b^{(3)})(a_2^{(2)}\leftarrowtail b^{(4)})\\
&=((a_1^{(1)} a_2^{(1)})\leftarrowtail b^{(1)})\otimes ((a_1^{(2)} a_2^{(2)})\leftarrowtail b^{(2)})\\
&=((a_1a_2)^{(1)}\leftarrowtail b^{(1)})\otimes ((a_1a_2)^{(2)}\leftarrowtail b^{(2)}).
\end{align*}
So $X$ is a subalgebra of $A_\bfP(V)$ for its usual product. Let $a\in F_\bfP(V)$ and $b\in D_\bfP(V)$. Then
$a\leftarrowtail b \in F_\bfP(V)$, so:
\begin{align*}
\Delta(a\leftarrowtail b)&=a\leftarrowtail b\otimes 1+1\otimes a\leftarrowtail b\\
&=a\leftarrowtail b^{(1)}\otimes \varepsilon(b^{(2)})1+\varepsilon(b^{(1)})1\otimes a\leftarrowtail b^{(2)}\\
&=a\leftarrowtail b^{(1)}\otimes 1\leftarrowtail b^{(2)}+1\leftarrowtail b^{(1)}\otimes a\leftarrowtail b^{(2)},
\end{align*}
so $a\in X$. As $X$ is a subalgebra containing $F_\bfP(V)$, it is  equal to $A_\bfP(V)$: $A_\bfP(V)$ is a coalgebra in the category of $D_\bfP(V)$-modules.\\

2. As this pre-Lie action comes from a brace action, if $p\in \bfP(n)$, $x_1,\ldots,x_n \in V$,
$q_1\otimes f_1,\ldots,q_k \otimes f_k \in \bfP\otimes \bfC_V$:
\begin{align*}
p.x_1\ldots x_n \leftarrowtail \overline{q_1\otimes f_1}\ldots \overline{q_k\otimes f_k}
&=\sum_{\substack{1\leq i_1,\ldots, i_k \leq n,\\\mbox{\scriptsize all distinct}}}
p\circ_{i_1,\ldots,i_k}(q_1,\ldots,q_k). (x_1\ldots f_1(x_{i_1})\ldots f_k(x_{i_k})\ldots x_n).
\end{align*}

Let us consider:
$$C=\left\{Q\in D_\bfP(V)\mid \begin{array}{c}
\forall p\in \bfP(n),\:v_1,\ldots,v_n \in F_\bfP(V),\\
p.(v_1,\ldots,v_n)\leftarrowtail Q=p.(v_1\leftarrowtail Q^{(1)},\ldots,v_n\leftarrowtail Q^{(n)})
\end{array}\right\}.$$
Obviously, $1\in C$. Let us take $Q_1,Q_2\in C$. For all $p\in \bfP(n)$, $v_1,\ldots,v_n \in F_\bfP(V)$:
\begin{align*}
p.(v_1,\ldots,v_n)\leftarrowtail Q_1*Q_2&=(p.(v_1,\ldots,v_n)\leftarrowtail Q_1)\leftarrowtail Q_2\\
&=p.((v_1\leftarrowtail Q_1^{(1)})\leftarrowtail Q_2^{(1)},\ldots,(v_n\leftarrowtail Q_1^{(n)})\leftarrowtail Q_2^{(n)})\\
&=p.(v_1\leftarrowtail (Q_1^{(1)}*Q_2^{(1)}),\ldots,v_n\leftarrowtail (Q_1^{(n)}*Q_2^{(n)}))\\
&=p.(v_1\leftarrowtail (Q_1*Q_2)^{(1)},\ldots,v_n\leftarrowtail (Q_1*Q_2)^{(n)}).
\end{align*}
So $Q_1*Q_2\in C$: $C$ is a subalgebra of $D_\bfP(V)$. 

Let us take $p\in \bfP(n)$ and $v_i=p_i.(x_{i,1},\ldots,x_{i,l_i})\in F_\bfP(V)$ for all $1\leq i\leq n$. If $q\otimes f\in \bfP\otimes \bfC_V$,
by the associativity of the operadic composition:
\begin{align*}
p.(v_1,\ldots,v_n)\leftarrowtail \overline{q\otimes f}&=
p\circ(p_1,\ldots,p_n).(x_{1,1},\ldots,x_{n,l_n})\leftarrowtail \overline{q\otimes f}\\
&=\sum_{i=1}^n \sum_{j=1}^{l_i}
p\circ(p_1,\ldots,p_i \circ_j q,\ldots,p_n).(x_{1,1},\ldots, f(x_{i,j}),\ldots,x_{n,l_n})\\
&=\sum_{i=1}^n p.(v_1,\ldots,v_i\leftarrowtail \overline{q\otimes f},\ldots, v_n). 
\end{align*}
So $coinv(\bfP\otimes \bfC_V) \subseteq C$. As $coinv(\bfP\otimes \bfC_V)$ generates $D_\bfP(V)$, $C=D_\bfP(V)$. \end{proof}

\begin{theo}\label{theo46}
Let $\theta_\bfP:\petitbinfini\longrightarrow \bfP$ be an operad morphism.
Any $\bfP$-algebra is also $\petitbinfini$, and we denote:
$$\star=\phi_{F_\bfP(V)}(\theta_\bfP(\lfloor-,-\rfloor)).$$
Then $(A_\bfP(V),\star,\Delta)$ and $D_\bfP(V)$ are two bialgebras in interaction.
\end{theo}

\begin{proof} We already proved that $(A_\bfP(V),m,\Delta)$ is a Hopf algebra in the category of $D_\bfP(V)$-modules.
Let us consider the two following maps:
\begin{align*}
\Phi_1&:\left\{\begin{array}{rcl}
A_\bfP(V)\otimes A_\bfP(V)\otimes D_\bfP(V)&\longrightarrow&A_\bfP(V)\\
a_1\otimes a_2\otimes b&\longrightarrow&(a_1*a_2)\leftarrowtail b,
\end{array}\right.\\ \\
\Phi_2&:\left\{\begin{array}{rcl}
A_\bfP(V)\otimes A_\bfP(V)\otimes D_\bfP(V)&\longrightarrow&A_\bfP(V)\\
a_1\otimes a_2\otimes b&\longrightarrow&(a_1\leftarrowtail b^{(1)})*(a_2\leftarrowtail b^{(2)}).
\end{array}\right.&
\end{align*}
For all $a_1,a_2\in A_\bfP(V)$, $b\in D_\bfP(V)$:
\begin{align*}
\Delta\circ \Phi_1(a_1\otimes a_2\otimes b)&=(a_1^{(1)}*a_2^{(1)})\leftarrowtail b^{(1)}\otimes (a_1^{(2)}*a_2^{(2)})\leftarrowtail b^{(2)}\\
&=(\Phi_1\otimes \Phi_1)\circ \Delta(a_1\otimes a_2\otimes b),\\ \\
\Delta\circ \Phi_2(a_1\otimes a_2\otimes b)&=(a_1^{(1)} \leftarrowtail b^{(1)})*(a_2^{(1)}\leftarrowtail b^{(3)})
\otimes (a_1^{(2)} \leftarrowtail b^{(2)})*(a_2^{(2)}\leftarrowtail b^{(4)})\\
&=(a_1^{(1)} \leftarrowtail b^{(1)})*(a_2^{(1)}\leftarrowtail b^{(2)})\otimes (a_1^{(2)} \leftarrowtail b^{(3)})*(a_2^{(2)}\leftarrowtail b^{(4)})\\
&=(\Phi_2\otimes \Phi_2)\circ \Delta(a_1\otimes a_2\otimes b).
\end{align*}
So both $\Phi_1$ and $\Phi_2$ are coalgebra morphisms. In order to prove that their equality,  by lemma \ref{lem21}, it is enough to prove that 
$\pi\circ \Phi_1=\pi\circ \Phi_2$, where $\pi$ is the canonical projection on $F_\bfP(V)$ in $A_\bfP(V)$.
As for all $k$, $S^k(F_\bfP(V))\leftarrowtail D_\bfP(V)\subseteq S^k(F_\bfP(V))$:
\begin{align*}
\pi\circ \Phi_1(a_1\otimes a_2\otimes b)&=\pi((a_1*a_2)\leftarrowtail b)=\pi(a_1*a_2)\leftarrowtail b=\lfloor a_1,a_2\rfloor \leftarrowtail b.
\end{align*}
The $\petitbinfini$ structure is induced by $\theta_\bfP$: denoting $q_{k,l}=\theta_\bfP(\lfloor-,-\rfloor_{k,l})$ and $\displaystyle q=\sum_{k,l\geq 0}q_{k,l}$,
for all $a_1,a_2\in A_\bfP(V)$, $\lfloor a_1,a_2\rfloor=q.(a_1,a_2)$. By lemma \ref{lem45}, for all $b\in D_\bfP(V)$:
\begin{align*}
\lfloor a_1,a_2\rfloor \leftarrowtail b&=q.(a_1,a_2)\leftarrowtail b\\
&=q.(a_1\leftarrowtail b^{(1)},a_2\leftarrowtail b^{(2)})\\
&=\lfloor a_1\leftarrowtail b^{(1)}, a_2\leftarrowtail b^{(2)}\rfloor\\
&=\pi \circ \Phi_2(a_1\otimes a_2\otimes b).
\end{align*}
As a conclusion, $\Phi_1=\Phi_2$, and $A_\bfP(V)$ is a bialgebra in the category of $D_\bfP(V)$-modules. \end{proof}

\subsection{First dual construction}

We assume now that for all $n\geq 1$, $\bfP$ is finite-dimensional. The composition can be seen as a map:
$$\circ:\bigoplus_{n\geq 1} \bfP(n)\otimes \bfP^{\otimes n}\longrightarrow \bfP.$$
Moreover, for any $n\geq 1$:
$$\circ^{-1}(\bfP(n))=\bigoplus_{p=1}^n \bigoplus_{k_1+\ldots+k_p=n} \bfP(p)\otimes \bfP(k_1)\ldots \bfP(k_p).$$
By duality, we obtain a map $\delta:\bfP^*\longrightarrow (\bfP\otimes T(\bfP))^*$, such that for any $n\geq 1$:
$$\delta(\bfP(n))\subseteq \bigoplus_{p=1}^n \bigoplus_{k_1+\ldots+k_p=n} \bfP^*(p)\otimes \bfP^*(k_1)\ldots \bfP^*(k_p).$$
So $\delta(\bfP^*)\subseteq \bfP^*\otimes T(\bfP^*)$.

\begin{prop}\begin{enumerate}
\item We define a coproduct $\Delta_*$ on $T(\bfP^*)$ as the unique algebra morphism (for the concatenation product $m_{conc}$) such that
for all $f\in \bfP^*$:
$$\Delta_*(f)=\delta(f).$$
Then $\bfD^*_\bfP=(T(\bfP^*),m_{conc},\Delta_*)$ is a bialgebra. It is graded, the elements of $\bfP^*(n)$ being homogeneous of degree $n-1$ for all $n\geq 1$.
\item There exists a nondegenerate pairing $\ll-,-\gg:T(\bfP^*)\otimes T(\bfP)\longrightarrow \K$, such that
for all $F,G \in T(\bfP^*)$, for all $X,Y\in T(\bfP)$:
\begin{align*}
\ll1,X\gg&=\varepsilon(X),&\ll F\otimes G,\Delta(X)\gg&=\ll FG,X\gg,\\
\ll F,1\gg&=\varepsilon(F),&\ll\Delta_*(F),X\otimes Y\gg&=\ll F,X*Y\gg.
\end{align*}
In other words, $\ll-,-\gg$ is a Hopf pairing between $\bfD_\bfP^*$ and $\bfD_\bfP$.
\end{enumerate}\end{prop}

\begin{proof} The following map is an algebra isomorphism:
$$\left\{\begin{array}{rcl}
T(\bfP^*_+)&\longrightarrow&T(\bfP^*)/I_0\\
f_1\ldots f_k&\longrightarrow&\overline{f_1\ldots f_k}.
\end{array}\right.$$

We define a first pairing between $T(\bfP^*)$ and $T(\bfP)$ by:
$$\forall x_1,\ldots,x_k\in \bfP,\: \forall f_1,\ldots,f_l \in \bfP^*,\:
\ll f_1\ldots f_l,x_1\ldots x_k\gg'=\begin{cases}
0\mbox{ if }k\neq l,\\
f_1(x_1)\ldots f_k(x_k)\mbox{ if }k=l.
\end{cases}$$
We shall need the completion:
$$\overline{T(\bfP)}=\prod_{k=0}^\infty \bfP^{\otimes k}.$$
We can extend the concatenation product, deconcatenation coproduct and the pairing as maps:
\begin{align*}
m_{conc}&:\overline{T(\bfP)}\otimes \overline{T(\bfP)}\longrightarrow\overline{T(\bfP)},\\
\Delta_{dec}&:\overline{T(\bfP)}\longrightarrow \overline{T(\bfP)\otimes T(\bfP)},\\
\ll-,-\gg'&:T(\bfP^*)\otimes \overline{T(\bfP)}\longrightarrow \K.
\end{align*}
We put:
$$J=\sum_{n=0}^\infty I^n=\frac{1}{1-I} \in \overline{T(\bfP)}.$$
Then:
$$\Delta(J)=\sum_{n=0}^\infty \sum_{k+l=n} I^k \otimes I^l=\sum_{k,l\geq 0} I^k \otimes I^l=J\otimes J.$$
This implies that the following map is a coalgebra morphism:
$$\phi:\left\{\begin{array}{rcl}
T(\bfP_+)&\longrightarrow&\overline{T(\bfP)}\\
x_1\ldots x_k&\longrightarrow&Jx_1J\ldots J x_k J.
\end{array} \right.$$
We now define the pairing $\ll-,-\gg$:
$$\forall F\in T(\bfP^*),\:X\in T(\bfP), \: \ll F,X\gg=\ll F,\phi(X)\gg'.$$
If $l>k$, $\ll f_1\ldots f_k,x_1\ldots x_k\gg=0$. If $k=l$:
\begin{align*}
\ll f_1\ldots f_k,x_1\ldots x_k\gg&=\ll f_1\ldots f_k,Jx_1J\ldots Jx_kJ\gg'\\
&=\ll f_1\ldots f_k,x_1\ldots x_k\gg'+0\\
&=f_1(x_1)\ldots f_k(x_k).
\end{align*}
By a triangularity argument, this pairing is non degenerate.

Let $X=x_1\ldots x_k \in T(\bfP)$. 
\begin{align*}
\ll 1,X\gg&=\begin{cases}
0\mbox{ if }k\geq 1\\
\ll 1,J\gg'\mbox{ if }k=0
\end{cases}\\
&=\begin{cases}
0\mbox{ if }k\geq 1\\
1\mbox{ if }k=0
\end{cases}\\
&=\varepsilon(X).
\end{align*}

Let $F,G\in T(\bfP^*)$, and $X\in T(\bfP)$.
\begin{align*}
\ll F\otimes G,\Delta(X)\gg&=\ll F\otimes G,(\phi\otimes \phi)\circ \Delta(X) \gg'\\
&=\ll F\otimes G,\Delta \circ \phi(X)\gg'\\
&=\ll FG,\phi(X)\gg'\\
&=\ll FG,X\gg.
\end{align*}

Let us consider now:
$$A=\{F\in T(\bfP^*)\mid \forall X,Y\in T(\bfP),\: \ll \Delta_*(F), X\otimes Y\gg=\ll F,X*Y\gg\}.$$
For all $X,Y\in T(\bfP)$:
$$\ll 1,X*Y\gg=\varepsilon(X*Y)=\varepsilon(X)\varepsilon(Y)=\ll 1\otimes 1,X\otimes Y\gg,$$
so $1\in A$. Let $F,G \in A$. For all $X,Y \in T(\bfP)$:
\begin{align*}
\ll \Delta_*(FG), X\otimes Y\gg&=\ll \Delta_*(F)\Delta_*(G),X\otimes Y\gg\\
&=\ll \Delta_*(F)\otimes \Delta_*(G),X^{(1)}\otimes Y^{(1)}\otimes X^{(2)}\otimes Y^{(2)}\gg\\
&=\ll F\otimes G,X^{(1)}*Y^{(1)}\otimes X^{(2)}*Y^{(2)}\gg\\
&=\ll F\otimes G,(X*Y)^{(1)}\otimes (X*Y)^{(2)}\gg\\
&=\ll FG,X*Y\gg.
\end{align*}
So $A$ is a subalgebra of $T(\bfP^*)$. In order to prove that $A=T(\bfP^*)$, it is now enough to prove that $\bfP^*\subseteq A$.
Let $f\in \bfP^*$, $X=x_1\ldots x_k$, $y=y_1\ldots y_l \in T(\bfP)$. 

\begin{itemize}
\item Let us assume that $k\geq 2$. As $\Delta_*(\bfP^*)\subseteq \bfP^*\otimes T(\bfP^*)$,
$\ll \Delta_*(f),X\otimes Y\gg=0$. Moreover:
$$\left(\bigoplus_{n\geq 2} \bfP^{\otimes n}\right)*T(\bfP)\subseteq \bigoplus_{n\geq 2} \bfP^{\otimes n},$$
so $\ll f,X*Y\gg=0$.
\item Let us assume that $k=0$. Then:
\begin{align*}
\ll \Delta_*(f),1\otimes Y\gg&=\ll \Delta_*(f),J\otimes \phi(Y)\gg'\\
&=\ll \delta(f), I\otimes \phi(Y)\gg'+0\\
&=\ll f,I\circ \phi(Y)\gg'\\
&=\ll f,\phi(Y)\gg'\\
&=\ll f,1*Y\gg.
\end{align*}
\item Let us finally assume that $k=1$. Then:
\begin{align*}
\ll \Delta_*(f),1\otimes Y\gg&=\ll \delta(f),x_1\otimes Jy_1J\ldots J y_lJ\gg'\\
&=\ll f,x_1\circ (Jy_1J\ldots Jy_lJ)\gg'\\
&=\ll f, \langle x_1,y_1\ldots y_l\rangle\gg'.
\end{align*}
Moreover:
\begin{align*}
\ll f,x_1*y_1\ldots y_l\gg&=\ll f,\langle x_1,y_1\ldots y_l\rangle \gg+\mbox{terms $\ll f,z_1\ldots z_l\gg'$, $l\geq 2$}\\
&=\ll f,\langle x_1,y_1\ldots y_l\rangle \gg+0\\
&=\ll f,\langle x_1,y_1\ldots y_l\rangle \gg'.
\end{align*}\end{itemize}
Finally, $A$ is equal to $T(\bfP^*)$: $\ll-,-\gg$ is a Hopf pairing.

Let $F\in T(\bfP^*)$, $X,Y,Z\in T(\bfP)$.
\begin{align*}
\ll (\Delta_*\otimes Id)\circ \Delta_*(F),X\otimes Y\otimes Z\gg&=\ll F,(X*Y)*Z\gg\\
&=\ll F,X*(Y*Z)\gg\\
&=\ll (Id \otimes \Delta_*)\circ \Delta_*(F),X\otimes Y\otimes Z\gg.
\end{align*}
As the pairing is nondegenerate, $\Delta_*$ is coassociative.  \end{proof}\\

\textbf{Remarks.} \begin{enumerate}
\item $\bfP^*(0)$ is a subcoalgebra of $\bfD_\bfP^*$; it is the dual of the algebra $(\bfP(0),\circ)$.
Its counit is denoted by $\varepsilon_0$; for all $f\in \bfP^*(0)$, $\varepsilon_0(f)=f(I)$.
\item In general, $\bfD_\bfP^*$ is not a Hopf algebra. Let us take the example where $\bfP(1)=Vect(I)$. We define $X\in \bfP^*(1)$ by
$X(I)=1$. Then $\Delta_*(X)=X\otimes X$. As $X$ has no inverse, $\bfD_\bfP^*$ is not a Hopf algebra.
\end{enumerate}

\begin{cor}
The abelianized algebra $S(\bfP^*)$ of $\bfD_\bfP^*$ inherits a coproduct $\Delta_*$, making it a bialgebra.
Moreover, $D_\bfP=S(inv\bfP^*)$ is a subbialgebra of $S(\bfP^*)$.
\end{cor}

\begin{proof} We denote by $I_{ab}$ the ideal of $T(\bfP^*)$ generated by all the commutators. Then $S(\bfP^*)=T(\bfP^*)/I_{ab}$.
Let $f\in inv\bfP^*$. We denote by $f^{(1)}\otimes f^{(2)}_1\ldots f^{(2)}_n$ the component of $\Delta_*(f)$
belonging to $\bfP^*(n)\otimes \bfP^*(k_1)\ldots \bfP^*(k_n)$. Let $\sigma \in \mathfrak{S}_n$, $\sigma_i \in \mathfrak{S}_{k_i}$.
There exists $\tau \in \mathfrak{S}_{k_1+\ldots+k_n}$ such that for all $p\in \bfP(n)$, $p_i\in \bfP(k_i)$:
\begin{align*}
&\ll (f^{(1)})^\sigma\otimes (f^{(2)}_1)^{\sigma_1}\ldots (f^{(2)}_n)^{\sigma_n},p\otimes p_1\ldots p_n\gg\\
&=\ll (f^{(1)})^\sigma\otimes (f^{(2)}_1)^{\sigma_1}\ldots (f^{(2)}_n)^{\sigma_n},p\otimes p_1\ldots p_n\gg'\\
&=\ll f^{(1)}\otimes f^{(2)}_1\ldots f^{(2)}_n, p^{\sigma^{-1}}\otimes p_1^{\sigma_1^{-1}}\ldots p_n^{\sigma_n^{-1}}\gg'\\
&=\ll f^{(1)}\otimes f^{(2)}_1\ldots f^{(2)}_n, p^{\sigma^{-1}}\otimes p_1^{\sigma_1^{-1}}\ldots p_n^{\sigma_n^{-1}}\gg'\\
&=\ll f, p^{\sigma^{-1}}\circ(p_1^{\sigma_1^{-1}},\ldots, p_n^{\sigma_n^{-1}}) \gg'\\
&=\ll f, (p\circ (p_{\sigma(1)},\ldots, p_{\sigma(n)}))^\tau\gg'\\
&=\ll f^{\tau^{-1}}, p\circ (p_{\sigma(1)},\ldots, p_{\sigma(n)})\gg'\\
&=\ll f, p\circ (p_{\sigma(1)},\ldots, p_{\sigma(n)})\gg'\\
&=\ll f^{(1)}\otimes f^{(2)}_{\sigma^{-1}(1)}\ldots f^{(2)}_{\sigma^{-1}(n)},p\otimes p_1\ldots p_n\gg.
\end{align*}
This implies that $\Delta_*(f) \in T(inv\bfP^*)\otimes T(inv\bfP^*)+T(\bfP^*) \otimes I_{ab}$.
So, in the quotient $S(\bfP^*)$, $\Delta_*(f)\in S(inv\bfP^*)\otimes S(inv\bfP^*)$, so $S(inv\bfP^*)$ is a subbialgebra of $S(\bfP^*)$. \end{proof}

\begin{cor}\begin{enumerate}
\item Let $I_0$ be the ideal of $\bfD_\bfP^*$ generated by the elements $f-\varepsilon_0(f)1$, $f\in \bfP^*(1)$.
This is a biideal; the quotient $\bfB^*_\bfP=\bfD_\bfP^*/I_0$ is a graded, connected Hopf algebra, and its graded dual is $(T(\bfP_+),*,\Delta)$.
\item Let $J_0$ be the ideal of $S(\bfP^*)$ generated by the elements $f-\varepsilon_0(f)1$, $f\in \bfP^*(1)$.
This is a biideal; the quotient $S(\bfP^*)/J_0$ is a graded, connected Hopf algebra, and its graded dual is $(S(\bfP_+),*,\Delta)$.
\item $B^*_\bfP=D_\bfP^*/J_0\cap D_\bfP^*$ is a graded, connected Hopf subalgebra of $S(\bfP^*)/J_0$, and its graded dual is $B_\bfP$.
\end{enumerate}\end{cor}

\begin{proof} 1. Firstly, observe that for all $f\in I_0$:
\begin{align*}
\Delta(f-\varepsilon_0(f)1)&=f^{(1)}\otimes f^{(2)}-\varepsilon_0(f)1\otimes 1\\
&=f^{(1)}\otimes (f^{(2)}-\varepsilon_0(f^{(2)})1)+\varepsilon_0(f^{(2)})f^{(1)}\otimes 1-\varepsilon_0(f)1\otimes 1\\
&=f^{(1)}\otimes (f^{(2)}-\varepsilon_0(f^{(2)})1)+f\otimes 1-\varepsilon_0(f)1\otimes 1\\
&=f^{(1)}\otimes (f^{(2)}-\varepsilon_0(f^{(2)})1)+(f-\varepsilon_0(f)1)\otimes 1 \in I_0\otimes \bfD_\bfP^*+\bfD_\bfP^*\otimes I_0.
\end{align*}
So $I_0$ is a biideal of $\bfD_\bfP^*$. Moreover, $\bfD_\bfP^*/I_0$ is isomorphic to the algebra $T(\bfP^*_+)$ via the morphism:
$$\left\{\begin{array}{rcl}
T(\bfP^*_+)&\longrightarrow&\bfD_\bfP^*/I_0\\
f_1\ldots f_k&\longrightarrow&\overline{f_1\ldots f_k}.
\end{array}\right.$$

Let $f\in \bfP_0$. If $x_1,\ldots,x_k \in \bfP_+$:
\begin{align*}
\ll f-\varepsilon_0(f)1,x_1\ldots x_k\gg&=\ll f-\varepsilon_0(f)1,Jx_1J\ldots J x_k J\gg'\\
&=\begin{cases}
0\mbox{ if }k\geq 2,\\
\ll f,x_1\gg' \mbox{ if }k=1,\\
\ll f,I\gg'-\varepsilon_0(f)\ll 1,1\gg'\mbox{ if }k=0;
\end{cases}\\
&=\begin{cases}
0\mbox{ if }k\geq 2,\\
0 \mbox{ if }k=1,\\
\varepsilon_0(f)-\varepsilon_0(f)=0\mbox{ if }k=0;
\end{cases}\\
&=0.
\end{align*}
So the pairing $\ll-,-\gg$ induces a Hopf pairing between $\bfD_\bfP^*/I_0$ and $\bfB_\bfP$, which is a Hopf subalgebra of $(T(\bfP),*\Delta)$.
Moreover, for all $f_1,\ldots,f_k\in \bfP^*_+$, $x_1,\ldots,x_k\in \bfP^*$:
\begin{align*}
\ll \overline{f_1\ldots f_k},x_1\ldots x_l\gg&=\ll f_1\ldots f_k, x_1\ldots x_l\gg\\
&=\ll f_1\ldots f_k,Jx_2J\ldots Jx_lJ\gg\\
&=\begin{cases}
0\mbox{ if }k>l \mbox{ (as $\ll f_i,I\gg'=0$ for all $i$)},\\
0\mbox{ if }k<l,\\
f_1(x_1)\ldots f_k(x_k)\mbox{ if }k=l.
\end{cases}
\end{align*}
Hence, we have a nondegenerate Hopf pairing between $\bfD_\bfP^*/I_0$ and $\bfB_\bfP$. It is clearly homogeneous, so
$\bfD_\bfP^*/I_0$ is isomorphic to the graded dual of $\bfB_\bfP$.\\

2. Taking the abelianization of $\bfD_\bfP^*/I_0$, we obtain the graded, connected Hopf algebra $S(\bfP^*)/J_0$.
Its graded dual is isomorphic to the largest cocommutative Hopf subalgebra of $T(\bfP_+)$, that is to say $coS(\bfP_+)$,
or up to an isomorphism $S(\bfP_+)$.\\

3. This is implied by the second point, noticing that the dual of $coinv\bfP(n)$ is $inv\bfP^*(n)$ for all $n$. \end{proof}

\begin{defi}For any finite-dimensional vector space $V$, we put:
\begin{align*}
\bfD_\bfP^*(V)&=\bfD^*_{\bfP\otimes \bfC_V},&D^*_\bfP(V)&=D^*_{\bfP\otimes \bfC_V},\\
\bfB^*_\bfP(V)&=\bfB^*_{\bfP\otimes \bfC_V},&B^*_\bfP(V)&=B^*_{\bfP\otimes \bfC_V}.
\end{align*}\end{defi}

\textbf{Remark.} If $V$ is one-dimensional:
\begin{align*}
\bfD_\bfP^*(V)&\approx\bfD_\bfP^*,&D_\bfP(V)&\approx D_\bfP,\\
\bfB^*_\bfP(V)&\approx \bfB^*_\bfP,&B^*_\bfP(V)&\approx B^*_\bfP.
\end{align*}

Let $\theta:\petitbinfini\longrightarrow \bfP$ be an operad morphism. Then $A_\bfP(V)$ is a graded Hopf algebra in the category of 
$B_\bfP(V)$-modules. Let us consider its graded dual $A^*_\bfP(V)$. Using the pairing $\ll-,-\gg$, one can define a coaction
of $D_\bfP(V)$ over the graded dual $A^*_\bfP(V)$, which we can quotient to obtain a coaction of $B_\bfP(V)$. We obtain:

\begin{cor} \label{cor51}
Let $V$ be a finite-dimensional space and $\theta:\petitbinfini\longrightarrow \bfP$ be an operad morphism.
Then $A^*_\bfP(V)$ and $D^*_\bfP(V)$ are in cointeraction, via the transposition $\rho$ of the action of $D_\bfP(V)$ over $A_\bfP(V)$.
\end{cor}

\subsection{Second dual construction}

As $\bfP$ is a $0$-bounded brace algebra, 
there exists a second coproduct $\Delta'_*$ on $T(\bfP^*)$, induced by the brace product, as defined in proposition \ref{prop17}.
It is a different from $\Delta_*$ (see section \ref{sect411} for an example), but there is a bialgebra isomorphism:

\begin{prop} \label{prop52}
The following map is a bialgebra isomorphism:
$$\Psi_\bfP:\left\{\begin{array}{rcl}
\bfD_\bfP^*=(T(\bfP^*),m_{conc},\Delta_*)&\longrightarrow&\bfD'_\bfP=(T(\bfP^*),m_{conc},\Delta'_*)\\
f\in \bfP_0^*&\longrightarrow&f-\varepsilon_0(f)1,\\
f\in \bfP_+^*&\longrightarrow&f.
\end{array}\right.$$
\end{prop}

\begin{proof} The coproducts $\Delta_*'$  and $\Delta_*$ are defined by:
\begin{align*}
\forall F\in \bfP^*,\: \forall X,Y\in \bfP,\: \ll \Delta_*'(F),X\otimes Y\gg'&=\ll F,X*Y\gg',\\
\ll \Delta_*(F),X\otimes Y\gg'&=\ll F,X*Y\gg.
\end{align*}
Let us first prove that for all $F\in T(\bfP^*)$, $X\in T(\bfP)$:
$$\ll \Psi_\bfP(F),X\gg'=\ll F,X\gg.$$
Let:
$$A=\{F\in T(\bfP^*)\mid \forall X\in T(\bfP),\: \ll \Psi_\bfP(F),X\gg'=\ll F,X\gg\}.$$
Clearly, $1\in A$. Let $F,G\in A$. For all $X\in T(\bfP)$:
\begin{align*}
\ll \Psi_\bfP(FG),X\gg'&=\ll\Psi_\bfP(F)\Psi_\bfP(G),X\gg'\\
&=\ll \Psi_\bfP(F)\otimes \Psi_\bfP(G),\Delta(X)\gg'\\
&=\ll F\otimes G,\Delta(X)\gg\\
&=\ll FG,X\gg.
\end{align*}
So $A$ is a subalgebra of $T(\bfP^*)$. Let us take $f\in \bfP^*$. For all $x_1,\ldots, x_k \in \bfP$:
\begin{align*}
\ll \Psi_\bfP(f),x_1\ldots x_k\gg'&=\ll f,x_1\ldots x_k\gg'-\varepsilon_0(f)\varepsilon(x_1\ldots x_k)\\
&=\begin{cases}
f(x_1)\mbox{ if }k=1,\\
0\mbox{ if }k\geq 2,\\
\varepsilon(f)-\varepsilon_0(f)\mbox{ if} k=0
\end{cases}\\
&=\begin{cases}
f(x_1)\mbox{ if }k=1,\\
0\mbox{ otherwise}
\end{cases}\\
&=\ll f,x_1\ldots x_k\gg.
\end{align*}
Consequently, $\bfP^*\subseteq A$, so $A=T(\bfP^*)$.\\

Let $F \in T(\bfP^*)$. For all $X,Y \in T(\bfP)$:
\begin{align*}
\ll (\Psi_\bfP\otimes \Psi_\bfP)\circ \Delta_*(F),X\otimes Y\gg'&=\ll \Delta_*(F),X\otimes Y\gg\\
&=\ll F,X*G\gg\\
&=\ll \Psi_\bfP(F),X*G\gg'\\
&=\ll \Delta_*'\circ \Psi_\bfP(F),X\otimes Y\gg'.
\end{align*}
As the pairing $\ll-,-\gg'$ is non degenerate, $(\Psi_\bfP\otimes \Psi_\bfP)\circ \Delta_*=\Delta_*'\circ \Psi_\bfP$. \end{proof}\\

Considering the abelianization:

\begin{cor} \label{cor53}
The following map is a bialgebra isomorphism:
$$\psi_\bfP:\left\{\begin{array}{rcl}
D^*_\bfP=(S(\bfP^*),m,\Delta_*)&\longrightarrow&D'_\bfP=(S(\bfP^*),m,\Delta'_*)\\
f\in \bfP_0^*&\longrightarrow&f-\varepsilon_0(f)1,\\
f\in \bfP_+^*&\longrightarrow&f.
\end{array}\right.$$
\end{cor}

\section{Associated groups and monoids}

\begin{prop}\label{prop54}\begin{enumerate}
\item \begin{enumerate}
\item The monoid  of characters of both $\bfD'_\bfP$ and $(\bfD'_\bfP)_{ab}$ 
is identified with $(\overline{P},\lozenge')$, where for all $x=\sum x_n \in \overline{\bfP}$, $y\in \overline{\bfP}$:
$$x\lozenge' y=y+\sum_{n\geq 1}\sum_{1\leq i_1<\ldots<i_k\leq n}
x_n \circ_{i_1,\ldots,i_k}(y,\ldots,y).$$
\item The monoid of characters of both $\bfD^*_\bfP$ and $(\bfD^*_\bfP)_{ab}$ is identified with $(\overline{P},\lozenge)$, 
where for all $x=\sum x_n \in \overline{\bfP}$, $y\in \overline{\bfP}$:
$$x\lozenge y=\sum_{n\geq 1} x_n \circ_{1,\ldots, n}(y,\ldots,y).$$
\item The monoids $(\overline{P},\lozenge')$ and $(\overline{P},\lozenge)$ are isomorphic.
\item $M^D_\bfP=\overline{coinv\bfP}$ is a quotient of the monoids $(\overline{\bfP},\lozenge)$, and $(\overline{\bfP},\lozenge')$,
and is identified with the monoid of characters of $D^*_\bfP$.
\end{enumerate}
\item \begin{enumerate}
\item The group of characters of both $\bfB_\bfP^*$ and $(\bfB_\bfP^*)_{ab}$ is identified with $(\overline{\bfP_+},\lozenge)$,
where for all $x=\sum x_n \in \overline{\bfP}$, $y\in \overline{\bfP}$:
$$x\lozenge y=y+\sum_{n\geq 1} x_n\circ(I+y,\ldots,I+y).$$
\item $G_\bfP^B=\overline{coinv\bfP_+}$ is a quotient group of $(\overline{\bfP_+},\lozenge)$, identified with the group of characters of $B^*_\bfP$.
\end{enumerate}\end{enumerate}\end{prop}

\begin{proof}
1. We use the description of the monoid of characters of $(\bfD'_\bfP)_{ab}=(S(\bfP^*),*,\Delta_*)$ from corollary \ref{cor32};
its product $\lozenge'$ is given by:
\begin{align*}
x\lozenge' y&=y+x\bullet e^y\\
&=y+\sum_{n\geq 1}\sum_{\substack{1\leq i_1,\ldots,i_k\leq n,\\ \mbox{\scriptsize all distincts},\\ j_1,\ldots,j_k\geq 0}}
\prod_{l=0}^\infty \frac{1}{\sharp\{p\mid j_p=l\}!}x\circ_{i_1,\ldots, i_k}(y_{j_1},\ldots,y_{j_k})\\
&=y+\sum_{n\geq 1}\sum_{\substack{1\leq i_1<\ldots<i_k\leq n,\\ j_1,\ldots,j_k\geq 0}}x\circ_{i_1,\ldots, i_k}(y_{j_1},\ldots,y_{j_k})\\
&=y+\sum_{n\geq 1}\sum_{1\leq i_1<\ldots<i_k\leq n}x_n \circ_{i_1,\ldots,i_k}(y,\ldots,y).
\end{align*}

For all $f\in \bfP^*$, we shall use the following notation for the transposition of the operadic composition:
$$\delta(f)=\sum_n f_n^{(1)}\otimes f_1^{(2)}\ldots f_n^{(2)}.$$
We identify the monoids of characters of $\bfD_\bfP$ with $\overline{\bfP}$ by the map:
$$\phi:\left\{\begin{array}{rcl}
\overline{\bfP}&\longrightarrow&M_\bfP^D\\
x&\longrightarrow&\left\{\begin{array}{rcl}
\bfD_\bfP&\longrightarrow&\K\\
f_1\ldots f_k&\longrightarrow&f_1(x)\ldots f_k(x).
\end{array}\right.\end{array}\right.$$
Let $x,y\in \overline{\bfP}$. For all $f\in \bfP^*$:
\begin{align*}
\phi_x*\phi_y(f)&=(\phi_x\otimes \phi_y)\circ \Delta_*(f)\\
&=(\phi_x\otimes \phi_y)\circ \delta_(f)\\
&=\sum_n f^{(1)}_n(x) f^{(2)}_1(y)\ldots f^{(2)}_n(y)\\
&=\sum_n f^{(1)}_n(x_n) f^{(2)}_1(y)\ldots f^{(2)}_n(y)\\
&=f(x_n\circ (y,\ldots,y))\\
&=\phi_{x\lozenge y}(f).
\end{align*}
As $\bfP^*$ generates $\bfD^*_\bfP$, $\phi_x*\phi_y=\phi_{x\lozenge y}$.
As $\bfD'_\bfP$ and $\bfD_\bfP^*$ are isomorphic, their monoids of characters are isomorphic. 
As $D_\bfP^*=S(inv\bfP^*)$ is a subbialgebra of $(\bfD_\bfP)_{ab}=S(\bfP^*)$, and as the graded dual of $inv\bfP^*$
is identified with $coinv\bfP$, the group of characters of $D_\bfP^*$ is identified with $\overline{coinv\bfP}$,
quotient of $(\overline{\bfP},\lozenge)$.  \\
  
 2. We use corollary \ref{cor32}. We obtain:
\begin{align*}
x\lozenge y&=y+x\bullet e^y\\
&=y+\sum_{n=1}^\infty \sum_{\substack{1\leq i_1,\ldots,i_k\leq n,\\ \mbox{\scriptsize all distincts},\\ j_1,\ldots,j_k\geq 1}}
\prod_{l=0}^\infty \frac{1}{\sharp\{p\mid j_p=l\}!}x\circ_{i_1,\ldots, i_k}(y_{j_1},\ldots,y_{j_k})\\
&=y+\sum_{n\geq 1}\sum_{\substack{1\leq i_1<\ldots<i_k\leq n,\\ j_1,\ldots,j_k\geq 1}}x\circ_{i_1,\ldots, i_k}(y_{j_1},\ldots,y_{j_k})\\
&=y+\sum_{n\geq 1}\sum_{1\leq i_1<\ldots<i_k\leq n}x_n \circ_{i_1,\ldots,i_k}(y,\ldots,y)\\
&=y+\sum_{n\geq 1} x\circ (I+y,\ldots,I+y).
\end{align*}
As $B_\bfP^*=S(inv\bfP^*_+)$ is a Hopf subalgebra of $(\bfB^*_\bfP)_{ab}=S(\bfP^*_+)$, and as the graded dual of $inv\bfP_+^*$
is identified with $coinv\bfP_+$, the group of characters of $B_\bfP^*$ is identified with $\overline{coinv\bfP_+}$,
quotient of $(\overline{\bfP_+},\lozenge)$. \end{proof}\\

\textbf{Remarks.} \begin{enumerate}
\item $G_\bfP^B$ can be seen a submonoid of $(M_\bfP^D,\lozenge)$ and $(M_\bfP^D,\lozenge')$, via the injections:
\begin{align*}
&\left\{\begin{array}{rcl}
G_\bfP^B&\longrightarrow&(M_\bfP^D,\lozenge)\\
x&\longrightarrow&I+x,
\end{array}\right.&
&\left\{\begin{array}{rcl}
G_\bfP^B&\longrightarrow&(M_\bfP^D,\lozenge')\\
x&\longrightarrow&x.
\end{array}\right.\end{align*}
\item Using the isomorphism $\psi_\bfP$, we obtain an explict isomorphism of monoids:
\begin{align*}
\left\{\begin{array}{rcl}
(M_\bfP^D,\lozenge')&\longrightarrow&(M_\bfP^D,\lozenge)\\
x&\longrightarrow&x+I,
\end{array}\right.&
&\left\{\begin{array}{rcl}
(M_\bfP^D,\lozenge)&\longrightarrow&(M_\bfP^D,\lozenge')\\
x&\longrightarrow&x-I.
\end{array}\right. \end{align*}

\end{enumerate}

\begin{defi} 
Let $V$ be a vector space and $\bfP$ be an operad. We put:
\begin{align*}
M^D_\bfP(V)&=M^D_{\bfP\otimes \bfC_V},&G^B_\bfP(V)&=G^B_{\bfP\otimes \bfC_V}.
\end{align*}\end{defi}

Let $\theta_\bfP:\petitbinfini\longrightarrow\bfP$ be an operad morphism. We obtain:
\begin{itemize}
\item A group structure on $G^A_\bfP(V)=\overline{F_\bfP(V)}$, given by:
$$\forall x,y\in G^A_\bfP(V),\: x\blacklozenge y=\lfloor e^x,e^y\rfloor.$$
In particular, if the morphism $\theta_\bfP$ is trivial, that is to say:
$$\theta_\bfP(\lfloor-,-\rfloor_{k,l})
=\begin{cases} 
I\mbox{ if }(k,l)=(1,0)\mbox{ or }(0,1),\\
0\mbox{ otherwise},
\end{cases}$$
then $x\blacklozenge y=x+y$.
\item The monoid $M^D_\bfP=(\overline{coinv(\bfP\otimes \bfC_V)},\lozenge)$ and the group
$G^B_\bfP=(\overline{coinv(\bfP\otimes \bfC_V)_+},\lozenge)$.
\item By theorem \ref{theo46} and proposition \ref{prop36}, there exists right actions $\vartriangleleft$ of $(M^D_\bfP(V),\diamond)$ 
and  $\vartriangleleft'$ of $(M^D_\bfP(V),\diamond')$ on $G^A_\bfP(V)$ by group endomorphisms; 
by restriction, there exists a right action $\vartriangleleft$ of $G^B_\bfP(V)$ on $G^A_\bfP(V)$ by group automorphisms.
\end{itemize}
Let us first describe these two actions.

\begin{prop}
The vector spaces $M^D_\bfP(V)$ and $\overline{F_\bfP(V)}\otimes V^*$ are canonically isomorphic.
For all $x=\sum p_n.(x_1,\ldots,x_n) \in G^A_\bfP(V)$, for all $y=q\otimes f \in M^D_\bfP(V)$, with $q\in \overline{F_\bfP(V)}$ and $f\in V^*$:
\begin{align*}
x\vartriangleleft' y&=\sum_{n\geq 1} \sum_{1\leq i_1<\ldots <i_k \leq n}
f(x_{i_1})\ldots f(x_{i_k}) p_n.(x_1,\ldots, x_{i_1-1},q,x_{i_1+1},\ldots,x_{i_k-1},q,x_{i_k+1},\ldots,x_n),\\
x\vartriangleleft y&=\sum_{n\geq 1} f(x_1)\ldots f(x_n) p.(q,\ldots,q).
\end{align*}
\end{prop}

\begin{proof} We naturally identify $End_\K(V,V^{\otimes n} )$ with $V^{\otimes n}\otimes V^*$. For all $n\geq 1$:
$$coinv\bfP\otimes \bfC_V(n)=\bfP(n)\otimes_{\mathfrak{S}_n} (V^{\otimes n}\otimes V^*)
=(\bfP(n)\otimes_{\mathfrak{S}_n} V^{\otimes n})\otimes V^*=F_\bfP(V)(n)\otimes V^*,$$
so:
$$M^D_\bfP(V)=\prod_{n\geq 1} F_\bfP(V)(n)\otimes V^*=
\left(\prod_{n\geq 1} F_\bfP(V)(n)\right)\otimes V^*=\overline{F_\bfP(V)}\otimes V^*.$$
The first formula comes from proposition \ref{prop36}, as $x\vartriangleleft' y=x\leftarrowtail e^y$.
The second formula is obtained by the application of the isomorphism between $D_\bfP(V)$ and $D'_\bfP(V)$, inducing the isomorphism
between $(M^D_\bfP(V),\lozenge')$ and $(M^D_\bfP(V),\lozenge)$. \end{proof}\\

The graduation of $F_\bfP(V)$ induces a distance $d$ on $F_\bfP(V)$: denoting $val$ the valutation associated to this graduation,
$$\forall x,y\in F_\bfP(V),\: d(x,y)=2^{-val(x-y)}.$$

\begin{prop}\label{prop57}
For all $y\in M^D_\bfP(V)$, we consider:
$$\phi_y:\left\{\begin{array}{rcl}
\overline{F_\bfP(V)}&\longrightarrow&\overline{F_\bfP(V)}\\
x&\longrightarrow&\phi_y(x)=x\vartriangleleft y.
\end{array}\right.$$
Then $\phi_y$ is a continuous endomorphism of $\bfP$-algebras. Moreover:
$$\forall y,z\in M^D_\bfP(V),\:\phi_y\circ \phi_z=\phi_{z\lozenge y}.$$
\end{prop}

\begin{proof}   For all $x\in G^A_\bfP(V)$, $y\in M^D_\bfP(V)$, $val(x\vartriangleleft y)\geq val(x)$, so $\phi_y$ is continuous. \\

Un to the isomorphism between $(M^D_\bfP(V),\diamond)$ and $(M^D_\bfP(V),\diamond')$, we work with the action $\vartriangleleft'$.
For all $y\in M^D_\bfP(V)$, $x\in \overline{F_\bfP(V)}$, we put $\phi_y(x)=x\vartriangleleft y$.
By lemma \ref{lem45}, for all $p\in \bfP(n)$, for all $v_1,\ldots,v_n \in F_\bfP(V)$, as $y$ is primitive, $e^y$ is a group-like element and:
\begin{align*}
\phi'_y(p.(v_1,\ldots,v_n))&=p.(v_1,\ldots,b_n)\leftarrowtail e^y\\
&=p.(v_1\leftarrowtail e^y,\ldots,v_n\leftarrowtail e^y)\\
&=p.(\phi'_y(v_1),\ldots,\phi'_y(v_n)).
\end{align*}
By continuity of $\phi'_y$, this is still true if $v_1,\ldots,v_n \in \overline{F_\bfP(V)}$, so $\phi'_y$ is indeed a continuous morphism
of $\bfP$-algebras. Up to an automorphism, this is also the case for $\phi_y$.\\

For all $x \in G^A_\bfP(V)$, $y,z\in M^D_\bfP(V)$:
\begin{align*}
\phi_y\circ \phi_z(x)&=(x\vartriangleleft z)\vartriangleleft y=x\vartriangleleft (z\lozenge y)=\phi_{z\lozenge y}(x),
\end{align*}
so $\phi_y\circ \phi_z=\phi_{z\lozenge y}$. \end{proof}\\

\textbf{Notations.} \begin{enumerate}
\item We denote by $\calM_\bfP(V)$ the monoid of continuous $\bfP$-algebra endomorphisms of $\overline{F_\bfP(V)}$
and by $\calG_\bfP(V)$ the group of continuous $\bfP$-algebra automorphisms of $\overline{F_\bfP(V)}$.
\item We put:
$$\overline{F_\bfP(V)}_2=\prod_{n\geq 2} \bfP(n). V^{\otimes n}.$$
If $\phi\in \calM_\bfP(V)$, then $\phi(\overline{F_\bfP(V)}_2)\subseteq \overline{F_\bfP(V)}_2$, so $\phi$ induces a map:
$$\phi':\left\{\begin{array}{rcl}
V=\overline{F_\bfP(V)}/\overline{F_\bfP(V)}_2&\longrightarrow&\overline{F_\bfP(V)}/\overline{F_\bfP(V)}_2\\
\overline{v}&\longrightarrow&\overline{\phi(v)}.
\end{array}\right.$$
We obtain in this way a monoid morphism $\varpi:\calM_\bfP(V)\longrightarrow End_\K(V)$,
and by restriction a group morphism $\varpi:\calG_\bfP(V)\longrightarrow GL(V)$.
The kernel of this morphism is denoted by $\calG_\bfP^{(1)}(V)$: this is the group of continuous automorphisms of $\overline{F_\bfP(V)}$
tangent to the identity. 
\item If $\varphi \in GL(V)$, let $\iota(\varphi)$ be the unique continuous endomorphism of $\bfP$-algebras of $\overline{F_\bfP(V)}$
such that $\iota(\varphi)(x)=\varphi(x)$ for all $x\in V$. Then $\iota:GL(V)\longrightarrow \calG_\bfP(V)$ is a group morphism,
such that $\varpi \circ \iota=Id_{GL(V)}$. So:
$$\calG_\bfP(V)=\calG_\bfP^{(1)}(V)\rtimes GL(V).$$
\end{enumerate}

\begin{prop}\label{prop58}
\begin{enumerate}
\item The morphism $\phi$ is an anti-isomorphism from $M^D_\bfP(V)$ to $\calM_\bfP(V)$.
\item Its restriction to $G^D_\bfP(V)$ is an anti-isomorphism from $G^B_\bfP(V)$ to $\calG^{(1)}_\bfP(V)$.
\end{enumerate} \end{prop}

\begin{proof} By proposition \ref{prop57}, $\phi$ is antimorphism from $M^B_\bfP(V)$ to $\calM_\bfP(V)$.
Let $y=q\otimes f \in M^D_\bfP(V)$, with $q\in \overline{F_\bfP(V)}$ and $f\in V^*$. For all $x\in V$,
$$\phi_y(x)=f(x)q,$$
so $\phi$ is injective. Let $\phi\in \calM_\bfP(V)$. We fix a basis $(x_i)_{i\in I}$ of $V$. For all $i\in I$,
we put $q_i=\phi(x_i)$, and $y=\sum q_j\otimes x_j^*$. For all $i\in I$:
$$\phi_y(x_i)=\sum_j x_j^*(x_i) q_j=q_i=\phi(x_i).$$
As both $\phi_y$ and $\phi$ are continuous morphisms of $\bfP$-algebras, they are equal, so $\phi$ is surjective.

Moreover, for all $y=q\otimes f\in M^D_\bfP(V)$:
\begin{align*}
\phi_y\in \calG^{(1)}_\bfP(V)&\Longleftrightarrow \forall x\in V,\: \phi_y(x)-x \in \overline{F_\bfP(V)}_2\\
&\Longleftrightarrow \forall x\in V, f(x)q-x\in \overline{F_\bfP(V)}_2\\
&\Longleftrightarrow y_1=Id_V\\
&\Longleftrightarrow y\in G^B_\bfP(V).
\end{align*}
So $\phi$ induces an anti-isomorphism from $G^B_\bfP(V)$ to $\calG^{(1)}_\bfP(V)$. \end{proof}

\section{Set bases}

In numerous cases, there exists,  for all $n \geq 1$, a basis $\calB_n$ of $\bfP(n)$ such that for for all $x\in \calB_n$, for all $\sigma\in \mathfrak{S}_n$,
$x^\sigma \in \calB_n$. We shall say in this case that the basis $\calB=\bigsqcup \calB_n$ of $\bfP$ is a set basis, and we shall identify 
the vector spaces $\bfP$ and $\bfP^*$ through the map:
$$\left\{\begin{array}{rcl}
\bfP&\longrightarrow&\bfP^*\\
x\in \calB&\longrightarrow&x^*\in \calB^*.
\end{array}\right.$$
We do not expect these bases to be stable under the operadic composition (we shall not always work here with set operads).
For example, $(e_n)_{n\geq 1}$ and $(\sigma)_{\sigma \in \sqcup \mathfrak{S}_n}$ are set bases of respectively $\com$ and $\ass$.\\

Let us consider a set basis $\calB$ of $\bfP$.
\begin{itemize}
\item For all $n\geq 1$, we denote by $\calO_n$ the set of orbits of the action of $\mathfrak{S}_n$ on $\calB_n$.
This is the set of isoclasses of elements of $\calB_n$.
\item For all $x\in \calB_n$, its orbit will be denoted by $\widehat{x}\in \calO_n$.
\item For all $\omega \in \calO_n$, we denote by $s_\omega$ the quotient cardinal of $N!/|\omega|$ (number of symmetries of $\omega$).
\end{itemize}
Let us fixe a system $(x_\omega)_{\omega \in \calO_n}$ of representants of the orbits. For all $\omega \in \calO_n$, we denote by
$\overline{\omega}$ the class of $x_\omega$ in $coinv\bfP(n)$: this does not depend of the choice of $x_\omega$, and 
$(\overline{\omega})_{\omega \in \calO_n}$ is a basis of $coinv(\bfP)$.
We denote by $\calB^*_n$ the dual basis of $\bfP^*(n)$. For all $\omega \in \calO_n$, we put:
$$f_\omega=\sum_{\sigma \in \mathfrak{S}_n}  (x_\omega^{\sigma})^*=s_\omega \sum_{x\in \omega} x^*.$$
Then $(f_\omega)_{\omega \in \sqcup \calO_n}$ is a basis of $inv\bfP^*$.
For all $\omega$, $\omega'\in \calO_n$:
$$f_\omega(\overline{\omega'})=\delta_{\omega,\omega'}  s_\omega.$$

We now fix a finite-dimensional vector space $V$. Let us choose a basis $(X_1,\ldots,X_N)$ of $V$ and let us denote by
$(\epsilon_1,\ldots,\epsilon_N)$ the dual basis of $V^*$. 
A basis of $\bfP\otimes V^{\otimes n}$ is given by $(x\otimes X_{i_1}\ldots X_{i_n})_{b\in \calB_n,1 \leq i_1,\ldots, i_n\leq N}$,
which can be seen as the set of elements of $\calB_n$ decorated by elements of $[N]$. The action of the symmetric group is given by:
$$(x\otimes X_{i_1}\ldots X_{i_n})^\sigma=x^\sigma\otimes X_{i_{\sigma(1)}}\ldots  X_{i_{\sigma(n}},$$
so this is also a set basis. The set of orbits of this action is interpreted as the set of isoclasses of elements of $\calB_n$ decorated by $[N]$.
It is a basis of the homogeneous component of degree $n$ of $F_\bfP(V)$.\\

A basis of $End_\K(V,V^{\otimes n})$ is given by 
$(\epsilon_j X_{i_1}\ldots X_{i_n})_{1\leq i_1,\ldots,i_n,j\leq N}$, where for all $v\in V$:
$$\epsilon_j X_{i_1}\ldots X_{i_n}(v)=\epsilon_j(v)X_{i_1}\ldots X_{i_n}.$$
A basis of $\bfP \otimes \bfC_V(n)$ is given by  $(x\otimes \epsilon_jX_{i_1}\ldots X_{i_n})_{b\in \calB_n,1 \leq j, i_1,\ldots, i_n\leq N}$.
The action of the symmetric group is given by:
$$(x\otimes \epsilon_j X_{i_1}\ldots X_{i_n})^\sigma=x^\sigma\otimes \epsilon_jX_{i_{\sigma(1)}}\ldots  X_{i_{\sigma(n}},$$
so it is a set basis. Moreover, the orbits of this action can be seen as pairs $(\omega,j)$, where $\omega$ is an isoclasse of elements
of $\calB$ decorated by $[N]$ and $j\in [N]$.

\chapter{Examples and applications}

Let us now give examples of these constructions. We start with classical operads $\com$, $\ass$ and $\prelie$. 
We obtain that $A^*_\com$ is the coordinate Hopf algebra of the group $G=(\K[[X]]_+,+)$ and
$B^*_\com$ is the coordinate bialgebra of the Faà di Bruno monoid of formal continuous maps $M=(\K[[X]]_+,\circ)$;
the coaction of $B^*_\com$ and $A^*_\com$ corresponds to the action of $M$ on $G$ by composition.
Similarly, for $\ass$, we obtain groups and monoids of non-commutative formal series.
Moreover, $A^*_\prelie$ is the Connes-Kreimer Hopf algebra \cite{Connes,Hoffman,Foissy1,Panaite}, 
and $A_\prelie$ is the Grossman-Larson Hopf algebra \cite{Grossman1,Grossman2,Grossman3}, both of them based on trees;
$D^*_\prelie$ is another bialgebra of rooted trees, whose coproduct is given by extraction-contraction operations,
defined in \cite{ManchonCalaque}, as well as the coaction of $D^*_\prelie$ on $A^*_\prelie$.\\

We then introduce an operadic structure on Feynman graphs, inducing operadic structures on other combinatorial objects
as simple graphs, simple graphs without cycle, posets. All these operads give pairs of (co)-interacting bialgebras,
as well as non-commutative versions of them; we recover in particular in this process the bialgebras on graphs without cycle of \cite{Manchon2},
or the bialgebras of quasiposets used in \cite{FoissyEhrhart}.

\section{Operads $\com$,  $\ass$ and $\prelie$}

\subsection{The operads $\com$ and $\ass$}

\label{sect411}
The brace and pre-Lie structures of $\com$ are given by:
\begin{align*}
\langle e_i, e_{j_1}\ldots e_{j_k}\rangle&=\binom{i+1}{k}e_{i+j_1+\ldots+j_k},&e_i\bullet e_j&=(i+1)e_{i+1}.
\end{align*}
Let us fix the vector space $V=(X_1,\ldots,X_N)$. We denote by $(\epsilon_i)_{i\geq 1}$ the dual basis of $(X_1,\ldots,X_N)$. 
We consider the morphism $\theta_\com:\petitbinfini\longrightarrow\com$ given in section \ref{sect223}. 
Then:
$$A_\com(V)=S(F_\com(V))=S(\K[X_1,\ldots,X_N]_+),$$
with the quasi-shuffle product $*$ induced by the product of $\K[X_1,\ldots,X_N]_+$.
For all $\alpha \in \N^N$, we put $X_\alpha=X_1^{\alpha_1}\ldots X_N^{\alpha_N}$. Then, for example, if $\alpha,\beta,\gamma,\delta \in \N^N-\{0\}$:
\begin{align*}
X_\alpha*X_\beta&=X_\alpha X_\beta+X_{\alpha+\beta},\\
X_\alpha*X_\beta X_\Gamma&=X_\alpha X_\beta X_\gamma+X_{\alpha+\beta}X_\gamma+X_\alpha X_{\beta+\gamma},\\
X_\alpha X_\beta*X_\gamma X_\delta&=X_\alpha X_\beta X_\gamma X_\delta+X_{\alpha+\gamma}X_\beta X_\gamma
+X_{\alpha+\delta}X_\beta X_\gamma\\
&+X_\alpha X_{\beta+\gamma}X_\delta+X_\alpha X_{\beta+\delta}X_\gamma+X_{\alpha+\gamma}X_{\beta+\delta}+X_{\alpha+\delta}X_{\beta+\gamma}.
\end{align*}

Dually, $A^*_\com(V)$ is identified with $ S(\K[X_1,\ldots,X_N]_+)$ as an algebra. Its coproduct $\Delta_*$ is given by:
$$\forall \alpha \in \N^N,\:\Delta(X_\alpha)=\sum_{\alpha=\beta+\gamma} X_\beta\otimes X_\gamma.$$ 
Its group of characters is isomorphic to:
$$(\{1+P(X_1,\ldots,X_N)\mid P\in \K[[X_1,\ldots,X_N]]_+\},\cdot).$$
Moreover, $B^*_\com(V)$ is the Faà di Bruno Hopf algebra on $N$ variables, which group of characters is the 
group of formal diffeomorphisms of $\K[[X_1,\ldots,X_N]]$ which are tangent to the identity, that is to say:
$$(\{(X_i+P(X_1\ldots,X_N))_{i\in [N]}\mid P_i(X_1,\ldots,X_N)\in \K[[X_1,\ldots,X_N]]_{\geq 2}\},\circ),$$
where $\K[[X_1,\ldots,X_N]]_{\geq 2}$ is the subspace of formal series in $\K[[X_1,\ldots, X_N]]$ of valuation $\geq 2$.\\

Here are examples of coproducts $\Delta_*$ and $\Delta_*'$ on $D^*_\com$:
\begin{align*}
\Delta'_*(e_1)&=e_1\otimes 1+1\otimes e_1+e_1\otimes e_1,&\Delta_*(e_1)&=e_1\otimes e_1,\\
\Delta'_*(e_2)&=e_2\otimes 1+1\otimes e_2+e_2\otimes e_1e_1+e_1\otimes e_2+2e_2\otimes e_1,&
\Delta_*(e_2)&=e_2\otimes e_1e_1+e_1\otimes e_2.
\end{align*}

Let us describe $\bfD^*_\ass$.

\begin{defi}
Let $\sigma \in \mathfrak{S}_n$.
\begin{enumerate}
\item We shall write $I_1\sqcup \ldots \sqcup I_k=_\sigma [n]$ if:
\begin{itemize}
\item For all $p\in [k]$, both $I_p$ and $\sigma(I_p)$ are intervals of $[n]$.
\item $[n]=I_1\sqcup \ldots \sqcup I_k$.
\item For all $1\leq p<q\leq k$, for all $i\in I_p$, $j\in I_q$, $i<j$.
\end{itemize} 
\item Let us assume that $I_1\sqcup \ldots \sqcup I_k=_\sigma [n]$. As $\sigma(I_1),\ldots,\sigma(I_k)$ are intervals, there exists a unique 
permutation $\tau\in \mathfrak{S}_k$ such that
$\sigma(I_{\tau(1)})\sqcup \ldots \sqcup \sigma(I_{\tau(k)})=_{\sigma^{-1}} [n]$.
We denote $\sigma/(I_1,\ldots,I_k)=\tau^{-1}$.
\end{enumerate}\end{defi}

The bialgebra $\bfD^*_\ass$ is freely generated by the set $\displaystyle \bigsqcup_{n\geq 1} \mathfrak{S}_n$. For all permutation $\sigma\in \mathfrak{S}_n$,
$$\Delta_*(\sigma)=\sum_{I_1\sqcup \ldots \sqcup I_k=_\sigma [n]}
\sigma/(I_1,\ldots,I_k)\otimes Std(\sigma_{\mid I_1}) \ldots Std(\sigma_{\mid I_k}),$$
where $Std$ is the usual standardization of permutations. For example:
\begin{align*}
\Delta_*((1))&=(1)\otimes (1),\\
\Delta_*((12))&=(12)\otimes (1)(1)+(1)\otimes (12),\\
\Delta_*((21))&=(21)\otimes (1)(1)+(1)\otimes (21),\\
\Delta_*((123))&=(123)\otimes (1)(1)(1)+(1)\otimes (123)+(12)\otimes (12)(1)+(12)\otimes (1)(12),\\
\Delta_*((132))&=(132)\otimes (1)(1)(1)+(1)\otimes (132)+(12)\otimes (1)(21),\\
\Delta_*((213))&=(213)\otimes (1)(1)(1)+(1)\otimes (213)+(12)\otimes (21)(1),\\
\Delta_*((231))&=(231)\otimes (1)(1)(1)+(1)\otimes (231)+(21)\otimes (12)(1),\\
\Delta_*((312))&=(312)\otimes (1)(1)(1)+(1)\otimes (312)+(21)\otimes (1)(12),\\
\Delta_*((321))&=(321)\otimes (1)(1)(1)+(1)\otimes (321)+(21)\otimes (21)(1)+(21)\otimes (1)(21).
\end{align*}

Let us consider the vector space
$V=Vect(X_1,\ldots,X_N)$. Then $D^*_\ass(V)$ is generated by the elements $(X_{i_1}\ldots X_{i_k}\epsilon_j)_{k\geq 1, i_1,\ldots,i_k,j\in [N]}$. 
For all word $w$ in letters $X_1,\ldots,X_N$, for all $i\in [N]$:
$$\Delta_*(w\epsilon_i)=\sum_{k\geq 0}\sum_{\substack{w=u_0v_1u_1\ldots v_ku_k,\\ i_1,\ldots,i_k \in [n]}}
u_0X_{i_1}u_1\ldots u_{k-1} X_{i_k}u_k\epsilon_i\otimes (v_1\epsilon_{i_1})\ldots(v_k\epsilon_k).$$
Note that the abelianization of $\bfD^*_\com(V)$ is $D^*_\ass(V)$.\\

\textbf{Examples.} In  $\bfD^*_\com(V)$ or in $D^*_\ass(V)$, if $i,j,k,l\in [N]$:
\begin{align*}
\Delta_*(X_i\epsilon_j)&=\sum_{p=1}^N X_p\epsilon_j\otimes X_i\epsilon_p,\\
\Delta_*(X_iX_j\epsilon_k)&=\sum_{p=1}^N X_p\epsilon_k\otimes X_iX_j\epsilon_p+\sum_{p,q=1}^N
X_pX_q\epsilon_k\otimes (X_i\epsilon_p)(X_j\epsilon_q),\\
\Delta_*(X_iX_jX_k\epsilon_l)&=\sum_{p=1}^N X_p\epsilon_l\otimes X_iX_jX_k\epsilon_p+\sum_{p,q=1}^N
X_pX_q\epsilon_l\otimes (X_iX_j\epsilon_p)(X_k\epsilon_q)\\
&+\sum_{p,q=1}^NX_pX_q\epsilon_l\otimes (X_i\epsilon_p)(X_jX_k\epsilon_q)
+\sum_{p,q,r=1}^NX_pX_qX_r\epsilon_l\otimes (X_i\epsilon_p)(X_j\epsilon_q)(X_k\epsilon_r).
\end{align*}
 In order to obtain the Hopf algebra $\bfB^*_\ass(V)$, we quotient by the relations $X_i\epsilon_j=\delta_{i,j}1$. The coproduct becomes:
\begin{align*}
\Delta_*(X_iX_j\epsilon_k)&=1\otimes X_iX_j\epsilon_p+X_iX_j\epsilon_k\otimes 1,\\
\Delta_*(X_iX_jX_k\epsilon_l)&=1\otimes X_iX_jX_k\epsilon_l+\sum_{p=1}^NX_pX_k\epsilon_l\otimes X_iX_j\epsilon_p
+\sum_{q=1}^NX_iX_q\epsilon_l\otimes X_jX_k\epsilon_q+X_iX_jX_k\epsilon_l\otimes 1.
\end{align*}
 
We consider the morphism $\theta_\ass:\petitbinfini\longrightarrow \ass$ defined in section \ref{sect223}.  
The Hopf algebra $A^*_\ass(V)$ is generated by the elements $(X_{i_1}\ldots X_{i_k})_{k\geq 1, i_1,\ldots,i_k\in [N]}$.
For all word $w$ in letters $X_1,\ldots,X_N$, for all $i\in [N]$:
$$\Delta_\star(w)=w\otimes 1+1\otimes w+\sum_{w=uv,\: u,v\neq \emptyset}u\otimes v.$$
The coaction of $D^*_\ass(V)$ over $A^*_\ass(V)$ is given by:
$$\rho(w)=\sum_{k\geq 0}\sum_{\substack{w=u_0v_1u_1\ldots v_ku_k,\\ i_1,\ldots,i_k \in [n]}}
u_0X_{i_1}u_1\ldots u_{k-1} X_{i_k}u_k\otimes (v_1\epsilon_{i_1})\cdot \ldots \cdot (v_k\epsilon_k).$$
For example:
\begin{align*}
\rho(X_i)&=\sum_{p=1}^N X_p\otimes X_i\epsilon_p,\\
\rho(X_iX_j)&=\sum_{p=1}^N X_p\otimes X_iX_j\epsilon_p+\sum_{p,q=1}^N
X_pX_q\otimes (X_i\epsilon_p)(X_j\epsilon_q),\\
\rho(X_iX_jX_k)&=\sum_{p=1}^N X_p\otimes X_iX_jX_k\epsilon_p+\sum_{p,q=1}^NX_pX_q\otimes (X_iX_j\epsilon_p)(X_k\epsilon_q)\\
&+\sum_{p,q=1}^NX_pX_q\otimes (X_i\epsilon_p)(X_jX_k\epsilon_q)+\sum_{p,q,r=1}^NX_pX_qX_r\otimes (X_i\epsilon_p)(X_j\epsilon_q)(X_k\epsilon_r).
\end{align*}
Here are examples of the coaction of $B^*_\ass(V)$:
\begin{align*}
\rho(X_i)&=X_i\otimes 1,\\
\rho(X_iX_j)&=X_iX_j\otimes 1+\sum_{p=1}^N X_p\otimes X_i\epsilon_p,\\
\rho(X_iX_jX_k)&=X_iX_jX_k\otimes 1+\sum_{p=1}^NX_pX_k\otimes X_iX_j\epsilon_p+\sum_{q=1}^NX_iX_q\otimes X_jX_k\epsilon_q
+\sum_{p=1}^N X_p\otimes X_iX_jX_k\epsilon_p.
\end{align*}

\subsection{The operad $\prelie$}

We now consider the operad $\prelie$, as described in \cite{Chapoton1,Chapoton2}. This operad comes from an operadic species; 
for all finite set $A$, $\prelie(A)$ is the vector space generated by rooted trees whose set of vertices is $A$. For example:
\begin{align*}
\prelie(\{1\})&=Vect(\tdun{$1$}),\\
\prelie(\{1,2\})&=Vect(\tddeux{$1$}{$2$},\tddeux{$2$}{$1$}),\\
\prelie(\{1,2,3\})&=Vect(\tdtroisdeux{$1$}{$2$}{$3$},\tdtroisdeux{$1$}{$3$}{$2$},\tdtroisdeux{$2$}{$1$}{$3$},
\tdtroisdeux{$2$}{$3$}{$1$},\tdtroisdeux{$3$}{$1$}{$2$},\tdtroisdeux{$3$}{$2$}{$1$},
\tdtroisun{$1$}{$3$}{$2$},\tdtroisun{$2$}{$3$}{$1$},\tdtroisun{$3$}{$2$}{$1$}).
\end{align*}
The composition is given by insertion at vertices in all possible ways. For example:
\begin{align*}
\tddeux{$1$}{$2$}\circ_1 \tddeux{$3$}{$4$}&=\tdtroisun{$3$}{$2$}{$4$}+\tdtroisdeux{$3$}{$4$}{$2$},&
\tddeux{$1$}{$2$}\circ_2 \tddeux{$3$}{$4$}&=\tdtroisdeux{$1$}{$3$}{$4$}.
\end{align*}

The morphism $\theta_\prelie:\petitbinfini\longrightarrow \prelie$ is described in section \ref{sect223}.\\

Let us fix $V=Vect(X_1,\ldots,X_N)$.
\begin{itemize}
\item A basis of $A_\prelie(V)$  is given by forests of rooted trees decorated by $[N]$;
in particular, if $i\in [N]$, $X_i$ is identified with $\tdun{$i$}$. The product is given by graftings; for example,
if $i,j,k \in [N]$:
\begin{align*}
\tdun{$i$}*\tdun{$j$}&=\tdun{$i$}\tdun{$j$}+\tddeux{$i$}{$j$},\\
\tdun{$i$}*\tddeux{$j$}{$k$}&=\tdun{$i$}\tddeux{$j$}{$k$}+\tdtroisdeux{$i$}{$j$}{$k$},&
\tddeux{$i$}{$j$}*\tdun{$k$}&=\tddeux{$i$}{$j$}\tdun{$k$}+\tdtroisun{$i$}{$k$}{$j$}+\tdtroisdeux{$i$}{$j$}{$k$},\\
\tdun{$i$}\tdun{$j$}*\tdun{$k$}&=\tdun{$i$}\tdun{$j$}\tdun{$k$}+\tddeux{$i$}{$k$}\tdun{$j$}+\tdun{$i$}\tddeux{$j$}{$k$},&
\tdun{$i$}*\tdun{$j$}\tdun{$k$}&=\tdun{$i$}\tdun{$j$}\tdun{$k$}+\tddeux{$i$}{$j$}\tdun{$k$}+\tddeux{$i$}{$k$}\tdun{$j$}+\tdtroisun{$i$}{$k$}{$j$}.
\end{align*}
In other terms, this is the Grossman-Larson Hopf algebra of decorated rooted trees \cite{Grossman1,Grossman2,Grossman3}. 
Its dual is (the coopposite of)  the Connes-Kreimer Hopf algebra of decorated rooted trees \cite{Connes,Foissy1,Panaite,Hoffman}, 
which coproduct is given by admissible cuts. If $i,j,k\in [N]$:
\begin{align*}
\Delta_\star(\tdun{$i$})&=\tdun{$i$}\otimes 1+1\otimes \tdun{$i$},\\
\Delta_\star(\tddeux{$i$}{$j$})&=\tddeux{$i$}{$j$}\otimes 1+1\otimes \tddeux{$i$}{$j$}+\tdun{$i$}\otimes \tdun{$j$},\\
\Delta_\star(\tdtroisun{$i$}{$k$}{$j$})&=\tdtroisun{$i$}{$k$}{$j$}\otimes 1+1\otimes \tdtroisun{$i$}{$k$}{$j$}+
\tddeux{$i$}{$j$}\otimes \tdun{$k$}+\tddeux{$i$}{$k$}\otimes \tdun{$j$}+\tdun{$i$}\otimes \tdun{$j$}\tdun{$k$},\\
\Delta_\star(\tdtroisdeux{$i$}{$j$}{$k$})&=\tdtroisdeux{$i$}{$j$}{$k$}\otimes 1+1\otimes \tdtroisdeux{$i$}{$j$}{$k$}+
\tddeux{$i$}{$j$}\otimes \tdun{$k$}+\tdun{$i$}\otimes \tddeux{$j$}{$k$}.
\end{align*}
\item A basis of $B_\prelie(V)$ is given by forests of pairs $(t,j)$, where $t$ is a rooted tree decorated by $[N]$ and $j\in [N]$.
The underlying pre-Lie product $\bullet$ is given by insertion at a vertex, as the operadic composition is. For example, if $i,j,k\in [N]$ and $N\geq 2$:
\begin{align*}
(\tddeux{$1$}{$2$},i)\bullet (\tddeux{$j$}{$k$},1)&=(\tdtroisun{$j$}{$2$}{$k$},i)+(\tdtroisdeux{$j$}{$k$}{$2$},i),&
(\tddeux{$1$}{$2$},i)\bullet (\tddeux{$j$}{$k$},2)&=(\tdtroisdeux{$1$}{$j$}{$k$},i),\\
(\tddeux{$1$}{$1$},i)\bullet (\tddeux{$j$}{$k$},1)&=(\tdtroisun{$j$}{$1$}{$k$},i)+(\tdtroisdeux{$j$}{$k$}{$1$},i)+(\tdtroisdeux{$1$}{$j$}{$k$},i),&
(\tddeux{$1$}{$1$},i)\bullet (\tddeux{$j$}{$k$},2)&=0.
\end{align*}
\item The bialgebra $D^*_\prelie(V)$ has the same basis. Its coproduct is given by extraction-contraction of subtrees. For example, 
in the non decorated case (or equivalently if $N=1$):
\begin{align*}
\Delta_*(\tun)&=\tun\otimes \tun,\\
\Delta_*(\tdeux)&=\tdeux\otimes \tun \tun+\tun \otimes \tdeux,\\
\Delta_*(\ttroisun)&=\ttroisun\otimes \tun\tun\tun+\tdeux\otimes \tdeux\tun+\tdeux\otimes \tun\tdeux+\tun \otimes \ttroisun,\\
\Delta_*(\ttroisdeux)&=\ttroisdeux\otimes \tun\tun\tun+\tdeux\otimes \tdeux \tun+\tdeux\otimes \tun \tdeux+\tun \otimes \ttroisdeux.
\end{align*}
This is the extraction-contraction coproduct of \cite{ManchonCalaque}.
More generally, in $D^*_\prelie(V)$, if $a,b,c,d\in [N]$: 
\begin{align*}
\Delta_*((\tdun{$a$},d))&=\sum_{p=1}^N (\tdun{$p$},d)\otimes (\tdun{$a$},p),\\
\Delta_*((\tddeux{$a$}{$b$},d))&=\sum_{p,q=1}^N (\tddeux{$p$}{$q$},d)\otimes (\tdun{$a$},p)(\tdun{$b$},q)
+\sum_{p=1}^N (\tdun{$p$},d) \otimes (\tddeux{$a$}{$b$},p),\\
\Delta_*((\tdtroisun{$a$}{$c$}{$b$},d))&=\sum_{p,q,r=1}^N (\tdtroisun{$p$}{$r$}{$q$},d)\otimes (\tdun{$a$},p)(\tdun{$b$},q)(\tdun{$c$},r)
+\sum_{p,q=1}^N (\tddeux{$p$}{$q$},d) \otimes (\tddeux{$a$}{$b$},p)(\tdun{$c$},q)\\
&+\sum_{p,q=1}^N (\tddeux{$p$}{$q$},d) \otimes (\tddeux{$a$}{$c$},p)(\tdun{$b$},q)
+\sum_{p=1}^N (\tdun{$p$},d)\otimes (\tdtroisun{$a$}{$c$}{$b$},p),\\
\Delta_*((\tdtroisdeux{$a$}{$b$}{$c$},d))&=\sum_{p,q,r=1}^N (\tdtroisdeux{$p$}{$q$}{$r$},d)\otimes (\tdun{$a$},p)(\tdun{$b$},q)(\tdun{$c$},r)
+\sum_{p,q=1}^N (\tddeux{$p$}{$q$},d) \otimes (\tddeux{$a$}{$b$},p)(\tdun{$c$},q)\\
&+\sum_{p,q=1}^N (\tddeux{$p$}{$q$},d) \otimes(\tdun{$a$},q) (\tddeux{$b$}{$c$},p)
+\sum_{p=1}^N (\tdun{$p$},d)\otimes (\tdtroisdeux{$a$}{$b$}{$c$},p).
\end{align*}
After taking the quotient by $\tdun{$i$}\epsilon_j X-\delta_{i,j}1$, in $B^*_\prelie(V)$:
\begin{align*}
\Delta_*((\tddeux{$a$}{$b$},d))&= (\tddeux{$a$}{$b$},d)\otimes 1+1\otimes (\tddeux{$a$}{$b$},d),\\
\Delta_*((\tdtroisun{$a$}{$c$}{$b$},d))&=(\tdtroisun{$a$}{$b$}{$c$},d)\otimes 1
+\sum_{p=1}^N (\tddeux{$p$}{$c$},d) \otimes (\tddeux{$a$}{$b$},p)+\sum_{p=1}^N (\tddeux{$p$}{$b$},d) \otimes (\tddeux{$a$}{$c$},p)
+1\otimes (\tdtroisun{$a$}{$c$}{$b$},d),\\
\Delta_*((\tdtroisdeux{$a$}{$b$}{$c$},d))&=(\tdtroisdeux{$a$}{$b$}{$c$},d)\otimes 1+\sum_{p=1}^N (\tddeux{$p$}{$c$},d) \otimes (\tddeux{$a$}{$b$},p)
+\sum_{q=1}^N (\tddeux{$a$}{$q$},d) \otimes (\tddeux{$b$}{$c$},p)+1\otimes (\tdtroisdeux{$a$}{$b$}{$c$},d).
\end{align*}
The coaction of $D^*_\prelie(V)$ over $A^*_\prelie(V)$ is given in a similar way. For example:
\begin{align*}
\rho(\tdun{$a$})&=\sum_{p=1}^N \tdun{$p$}\otimes (\tdun{$a$},p),\\
\rho(\tddeux{$a$}{$b$})&=\sum_{p,q=1}^N \tddeux{$p$}{$q$}\otimes (\tdun{$a$},p)(\tdun{$b$},q)
+\sum_{p=1}^N \tdun{$p$} \otimes (\tddeux{$a$}{$b$},p),\\
\rho(\tdtroisun{$a$}{$c$}{$b$})&=\sum_{p,q,r=1}^N \tdtroisun{$p$}{$r$}{$q$}\otimes (\tdun{$a$},p)(\tdun{$b$},q)(\tdun{$c$},r)
+\sum_{p,q=1}^N \tddeux{$p$}{$q$} \otimes (\tddeux{$a$}{$b$},p)(\tdun{$c$},q)\\
&+\sum_{p,q=1}^N \tddeux{$p$}{$q$} \otimes (\tddeux{$a$}{$c$},p)(\tdun{$b$},q)
+\sum_{p=1}^N \tdun{$p$}\otimes (\tdtroisun{$a$}{$c$}{$b$},p),\\
\rho(\tdtroisdeux{$a$}{$b$}{$c$})&=\sum_{p,q,r=1}^N \tdtroisdeux{$p$}{$q$}{$r$}\otimes (\tdun{$a$},p)(\tdun{$b$},q)(\tdun{$c$},r)
+\sum_{p,q=1}^N \tddeux{$p$}{$q$} \otimes (\tddeux{$a$}{$b$},p)(\tdun{$c$},q)\\
&+\sum_{p,q=1}^N \tddeux{$p$}{$q$} \otimes(\tdun{$a$},q) (\tddeux{$b$}{$c$},p)
+\sum_{p=1}^N \tdun{$p$}\otimes (\tdtroisdeux{$a$}{$b$}{$c$},p).
\end{align*}\end{itemize}

\section{Feynman Graphs}

\subsection{Oriented Feynman graphs}

We shall use the following formalism for oriented Feynman graphs :

\begin{defi}\label{defi60}
\begin{enumerate}
\item A Feynman graph is a family
$$\Gamma=(V(\Gamma),Int(\Gamma), OutExt(\Gamma), InExt(\Gamma),S_\Gamma,T_\Gamma),$$
where:
\begin{itemize}
\item $V(\Gamma)$ is a finite, non-empty set, called the set of vertices of $\Gamma$.
\item $Int(\Gamma)$ is a finite set, called the set of internal edges of $\Gamma$.
\item $OutExt(\Gamma)$ is a finite set, called the set of external outgoing edges of $\Gamma$.
\item $InExt(\Gamma)$ is a finite set, called the set of internal ingoing edges of $\Gamma$.
\item $S_\Gamma:Int(\Gamma)\sqcup OutExt(\Gamma)\longrightarrow V(\Gamma)$ is the source map.
\item $T_\Gamma:Int(\Gamma)\sqcup InExt(\Gamma)\longrightarrow V(\Gamma)$ is the target map.
\end{itemize}
\item Let $\Gamma$ and $\Gamma'$ be two Feynman graphs. We shall say that $\Gamma$ and $\Gamma'$ are equivalent if the following conditions hold:
\begin{itemize}
\item $V(\Gamma)=V(\Gamma')$.
\item There exist bijections: 
\begin{align*}
\phi_{Int}&:Int(\Gamma)\longrightarrow Int(\Gamma'),\\
\phi_{OutExt}&:OutExt(\Gamma)\longrightarrow OutExt(\Gamma'),\\
\phi_{InExt}&:InExt(\Gamma)\longrightarrow IntExt(\Gamma'),
\end{align*}
such that:
\begin{align*}
\forall e\in Int(\Gamma),\: S_{\Gamma}(e)&=S_{\Gamma'}\circ \phi_{Int}(e)\mbox{ and }T_{\Gamma}(e)=T_{\Gamma'}\circ \phi_{Int}(e),\\
\forall e\in OutExt(\Gamma),\: S_{\Gamma}(e)&=S_{\Gamma'}\circ \phi_{OutExt}(e),\\
\forall e\in InExt(\Gamma),\: T_{\Gamma}(e)&=T_{\Gamma'}\circ \phi_{InExt}(e).
\end{align*} \end{itemize}
Roughly speaking, two Feynman graphs $\Gamma$ and $\Gamma'$ are equivalent if one obtains $\Gamma$ from $\Gamma'$ by changing the names of
its edges.
\item Let $A$ be a finite set. The set of all equivalence classes of Feynman graphs $\Gamma$ such that $V(\Gamma)$ is denoted by $\fg(A)$
and the space generated by $\fg(A)$ will be denoted by $\bffg(A)$.
\end{enumerate}\end{defi}

We shall work only with equivalence classes of Feynman graphs, which we now simply call Feynman graphs.\\

\textbf{Remarks.} \begin{enumerate}
\item $Int(\Gamma)$, $OutExt(\Gamma)$ or $InExt(\Gamma)$ may be empty.
\item Restricting to $Int(\Gamma)$, Feynman graphs are also oriented graphs, possibly with multiple edges and loops.
\end{enumerate}

We shall represent Feynman graphs by a diagram, such as:
$$
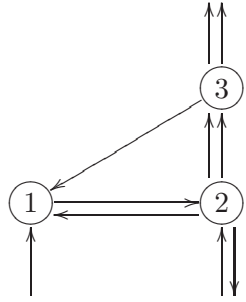
\xymatrix{&&&\\
&&\rond{3}\ar@<1ex>[u] \ar[u] \ar[lld]\\
\rond{1}\ar[rr]&&\rond{2}\ar@<1ex>[ll] \ar[u] \ar@<1ex>[u] \ar@<1ex>[d]\\
\ar[u]&&\ar[u]}$$

\begin{defi}
Let $\Gamma$ be a Feynman graph, and $I\subseteq V(\Gamma)$, non-empty. 
\begin{enumerate}
\item (Extraction). We define the Feynman graph $\Gamma_{\mid I}$ by:
\begin{align*}
V(\Gamma_{\mid I})&=I,\\
Int(\Gamma_{\mid I})&=\{e\in Int(\Gamma)\mid S_\Gamma(e) \in I,\: T_\Gamma(e)\in I\},\\
OutExt(\Gamma_{\mid I})&=\{e\in Int(\Gamma)\mid S_\Gamma(e) \in I,\: T_\Gamma(e)\notin I\}
\sqcup \{e\in OutExt(\Gamma)\mid S_\Gamma(e)\in I\},\\
InExt(\Gamma_{\mid I})&=\{e\in Int(\Gamma)\mid S_\Gamma(e) \notin I,\: T_\Gamma(e)\in I\}
\sqcup \{e\in InExt(\Gamma)\mid T_\Gamma(e)\in I\}.
\end{align*}
For all $e\in Int(\Gamma_{\mid I}) \sqcup OutExt(\Gamma_{\mid I})$, for all $f\in Int(\Gamma_{\mid I}) \sqcup IntExt(\Gamma_{\mid I})$: 
\begin{align*}
S_{\Gamma_{\mid I}}(e)&=S_\Gamma(e),&T_{\Gamma_{\mid I}}(f)&=T_\Gamma(f).
\end{align*}
Roughly speaking, $\Gamma_{\mid I}$ is the Feynman graph obtained by taking all the vertices in $I$ and the half edges attacted to them.
\item (Contraction). Let $b\notin V(\Gamma)$. We define the Feynman graph $\Gamma/I\rightarrow b$ by:
\begin{align*}
V(\Gamma/I\rightarrow b)&=(V(\Gamma) \setminus I)\sqcup\{b\},&OutExt(\Gamma/I\rightarrow b)&=OutExt(\Gamma),\\
Int(\Gamma/I\rightarrow b)&=Int(\Gamma) \setminus Int(\Gamma_{\mid I}),&InExt(\Gamma/I\rightarrow b)&=InExt(\Gamma).
\end{align*}
For all $e\in Int(\Gamma/I\rightarrow b)\sqcup OutExt(\Gamma/I\rightarrow b)$:
$$S_{\Gamma/I\rightarrow b}(e)=\begin{cases}
S_\Gamma(e)\mbox{ if }S_\Gamma(e)\notin I,\\
b\mbox{ if }otherwise.
\end{cases}$$
For all $e\in Int(\Gamma/I\rightarrow b)\sqcup InExt(\Gamma/I\rightarrow b)$:
$$T_{\Gamma/I\rightarrow b}(e)=\begin{cases}
T_\Gamma(e)\mbox{ if }T_\Gamma(e)\notin I,\\
b\mbox{ if }otherwise.
\end{cases}$$
Roughly speaking, $\Gamma/I\rightarrow b$ is obtained by contracting all the vertices of $I$ and the internal edges between them to a single vertex $b$.
\item We shall say that $I$ is $\Gamma$-convex if for any oriented path $x\rightarrow y_1\rightarrow\ldots \rightarrow y_k\rightarrow z$ in $\Gamma$:
$$x,z\in I \Longrightarrow y_1,\ldots,y_k\in I.$$
\end{enumerate}\end{defi}

\subsection{Lemmas on extraction-contraction}

\begin{lemma}\label{lem62}
Let $\Gamma$ be a Feynman graph, $I_a,I_b\subseteq V(\Gamma)$, non-empty and disjoint.
\begin{enumerate}
\item $(\Gamma/I_a\rightarrow a)_{\mid I_b}=\Gamma_{\mid I_b}$.
\item $(\Gamma/I_a\rightarrow a)/I_b\rightarrow b=(\Gamma/I_b\rightarrow b)/I_a\rightarrow a$.
\item The following conditions are equivalent:
\begin{enumerate}
\item $I_a$ is $\Gamma$-convex and $I_b$ is $(\Gamma/I_a\rightarrow a)$-convex.
\item $I_b$ is $\Gamma$-convex and $I_a$ is $(\Gamma/I_b\rightarrow b)$-convex.
\end{enumerate}\end{enumerate}\end{lemma}

\begin{proof} 1. Let us put $\Gamma'=(\Gamma/I_a\rightarrow a)_{\mid I_b}$. Then:
\begin{align*}
V(\Gamma')&=I_b,\\
Int(\Gamma')&=\{e\in Int(\Gamma) \mid S_\Gamma(e)\in I_b, \: T_\Gamma(e)\in I_b\},\\
OutExt(\Gamma')&=\{e\in Int(\Gamma) \mid S_\Gamma(e)\in I_b, \: T_\Gamma(e)\notin I_b\}
\sqcup \{e\in OutExt(\Gamma)\mid S_\Gamma(e)\in I_b\},\\
InExt(\Gamma')&=\{e\in Int(\Gamma) \mid S_\Gamma(e)\notin I_b, \: T_\Gamma(e)\in I_b\}
\sqcup \{e\in OutExt(\Gamma)\mid T_\Gamma(e)\in I_b\}.
\end{align*}
For all $e\in Int(\Gamma')\sqcup OutExt(\Gamma')$, $S_{\Gamma'}(e)=S_\Gamma(e)$.
For all $e\in Int(\Gamma')\sqcup InExt(\Gamma')$, $T_{\Gamma'}(e)=T_\Gamma(e)$. So $\Gamma'=\Gamma_{\mid I_b}$.\\

2. Let us put $\Gamma''=(\Gamma/I_a\rightarrow a)/I_b\rightarrow b$. Then:
\begin{align*}
V(\gamma'')&=V(\gamma) \sqcup \{a,\b\}\setminus (I_a\sqcup I_b),\\
Int(\Gamma'')&=\{e\in Int(\Gamma)\mid S_\Gamma(e)\notin I_a\sqcup I_b \mbox{ or }T_\Gamma(e)\notin I_a\sqcup I_b\},\\
OutExt(\Gamma'')&=OutExt(\Gamma),\\
InExt(\Gamma'')&=InExt(\Gamma).
\end{align*}
For all $e\in Int(\Gamma'')\sqcup OutExt(\Gamma'')$:
$$S_{\Gamma''}(e)=\begin{cases}
a\mbox{ if }S_\Gamma(e)\in I_a,\\
b\mbox{ if }S_\Gamma(e)\in I_b,\\
S_\Gamma(e)\mbox{ otherwise}.
\end{cases}$$
For all $e\in Int(\Gamma'')\sqcup IntExt(\Gamma'')$:
$$T_{\Gamma''}(e)=\begin{cases}
a\mbox{ if }T_\Gamma(e)\in I_a,\\
b\mbox{ if }T_\Gamma(e)\in I_b,\\
T_\Gamma(e)\mbox{ otherwise}.
\end{cases}$$
By symmetry between $a$ and $b$, $\Gamma''=(\Gamma/I_b\rightarrow b)/I_a\rightarrow a$.\\

3. $\Longrightarrow$. Let us assume that $x\rightarrow y_1\rightarrow\ldots \rightarrow y_k\rightarrow z$ in $\Gamma$,
with $x,z\in I_b$. For all $y\in V(\Gamma)$, we put:
$$\overline{y}=\begin{cases}
a\mbox{ if }y\in I_a,\\
y\mbox{ otherwise}.
\end{cases}$$
Then $x\rightarrow \overline{y_1}\rightarrow\ldots \rightarrow \overline{y_k}\rightarrow z$ in $\Gamma/I_a\rightarrow a$.
As $I_b$ is $(\Gamma/I_a\rightarrow a)$-convex, all the $\overline{y_i}$ belong to $I_b$, so are different from $a$:
hence, $y_1,\ldots,y_k \in I_b$.

Let us assume $x\rightarrow y_1\rightarrow\ldots \rightarrow y_k\rightarrow z$ in $\Gamma/I_b\rightarrow b$, with $x,z\in I_a$.
If at least one of the $y_p$ is equal to $b$, let us consider the smallest index $i$ such that $x_i=b$ and the greatest index $j$ such that $x_j=b$.
There exists $y'_i,y''_j \in I_b$, such that $x\rightarrow y_1\rightarrow \ldots \rightarrow y_{i-1}\rightarrow y'_i$ and
$y''_j\rightarrow y_{j+1}\rightarrow \ldots \rightarrow z$ in $\Gamma$. As a consequence, in $\Gamma/I_a\rightarrow a$:
$$y''_j\rightarrow y_{j+1}\rightarrow\ldots \rightarrow y_k \rightarrow a\rightarrow y_1\rightarrow\ldots \rightarrow y_{i-1}\rightarrow y''_i.$$
As $I_b$ is $(\Gamma/I_a\rightarrow a)$-convex, $a\in I_b$, which is absurd. So none of the $y_p$ is equal to $b$, which implies that
$x\rightarrow y_1\rightarrow\ldots \rightarrow y_k\rightarrow z$ in $\Gamma$. As $I_a$ is $\Gamma$-convex,
$y_1,\ldots,y_p\in I_a$.\\

$\Longleftarrow$: by symmetry between $a$ and $b$. \end{proof}

\begin{lemma} \label{lem63}
Let $\Gamma$ be a Feynman graph, and $I_a\subseteq I_b \subseteq V(G)$ be non-empty sets.
\begin{enumerate}
\item $(\Gamma/I_a\rightarrow a)/(I_b\sqcup\{a\}\setminus I_a)\rightarrow b=\Gamma/I_b\rightarrow b$.
\item $(\Gamma/I_a \rightarrow a)_{\mid I_b \sqcup\{a\}\setminus I_a}=(\Gamma_{\mid I_b})/I_a\rightarrow a$.
\item $(\Gamma_{\mid I_b})_{\mid I_a}=\Gamma_{\mid I_a}$.
\item The following conditions are equivalent:
\begin{enumerate}
\item $I_b$ is $\Gamma$-convex and $I_a$ is $\Gamma_{\mid I_b}$-convex.
\item $I_a$ is $\Gamma$-convex and $I_b \sqcup \{a\}\setminus I_a$ is $\Gamma/I_a\rightarrow a$-convex.
\end{enumerate}\end{enumerate}\end{lemma}

\begin{proof} 1. Let $\Gamma'=(\Gamma/I_a\rightarrow a)/ (I_b\sqcup\{a\}\setminus I_a)\rightarrow b$. Then:
\begin{align*}
V(\Gamma')&=V(\Gamma)\sqcup\{a\}\setminus I_b,\\
Int(\Gamma')&=\{e\in Int(\Gamma)\mid S_\Gamma(e)\notin I_b\mbox{ or } T_\Gamma(e)\notin I_b\},\\
OutExt(\Gamma')&=OutExt(\Gamma),\\
InExt(\Gamma')&=InExt(\Gamma).
\end{align*}
For all $e\in Int(\Gamma')\sqcup OutExt(\Gamma')$:
$$S_{\Gamma'}(e)=\begin{cases}
S_\Gamma(e)\mbox{ if } S_\Gamma(e)\notin I_b,\\
b\mbox{ otherwise}.
\end{cases}$$
For all $e\in Int(\Gamma')\sqcup InExt(\Gamma')$:
$$T_{\Gamma'}(e)=\begin{cases}
T_\Gamma(e)\mbox{ if } T_\Gamma(e)\notin I_b,\\
b\mbox{ otherwise}.
\end{cases}$$
So $\Gamma'=\Gamma/I_b\rightarrow b$.\\

2. Let $\Gamma''=(\Gamma/I_a \rightarrow a)_{\mid I_b \sqcup\{a\}\setminus I_a}$. Then:
\begin{align*}
V(\Gamma'')&=I_b\sqcup\{a\}\setminus I_a,\\
Int(\Gamma'')&=\{e\in Int(\Gamma)\mid S_\Gamma(e)\in I_b\mbox{ and } T_\Gamma(e)\in I_b\}\\
&\hspace{1cm}\setminus\{e\in Int(\Gamma)\mid S_\Gamma(e)\in I_a\mbox{ and } T_\Gamma(e)\in I_a\},\\
OutExt(\Gamma'')&=\{e\in OutExt(\Gamma)\mid S_\Gamma(e)\in I_b\}\sqcup
\{e\in Int(\Gamma)\mid S_\Gamma(e)\in I_b \mbox{ and } T_\Gamma(e)\notin I_b\},\\
InExt(\Gamma'')&=\{e\in InExt(\Gamma)\mid T_\Gamma(e)\in I_b\}\sqcup
\{e\in Int(\Gamma)\mid S_\Gamma(e)\notin I_b \mbox{ and } T_\Gamma(e)\in I_b\}.
\end{align*}
For all $e\in Int(\Gamma'')\sqcup OutExt(\Gamma'')$:
$$S_{\Gamma''}(e)=\begin{cases}
S_\Gamma(e)\mbox{ if } S_\Gamma(e)\in I_b\setminus I_a,\\
a\mbox{ otherwise}.
\end{cases}$$
For all $e\in Int(\Gamma'')\sqcup InExt(\Gamma'')$:
$$T_{\Gamma''}(e)=\begin{cases}
T_\Gamma(e)\mbox{ if } T_\Gamma(e) \in I_b\setminus I_a,\\
a\mbox{ otherwise}.
\end{cases}$$
So $\Gamma''=(\Gamma_{\mid I_b})/I_a\rightarrow a$.\\

3. Immediate.\\

4. $\Longrightarrow$. Let $x\rightarrow y_1\rightarrow\ldots \rightarrow y_k\rightarrow z$ in $\Gamma$, with $x,z\in I_a$.
Then $x,z\in I_b$. As $I_b$ is $\Gamma$-convex, $y_1,\ldots,y_k\in I_b$. As $I_a$ is $\Gamma_{\mid I_b}$-convex, $y_1,\ldots,y_k\in I_a$.

Let $x\rightarrow y_1\rightarrow\ldots \rightarrow y_k\rightarrow z$ in $\Gamma/I_a\rightarrow a$, with $x,z\in I_b\sqcup \{a\}\setminus I_a$.
Let $i_1<\ldots<i_l$ be the indices such that $y_i=a$. There exists elements $y'_{i_p}$, $y''_{i_p} \in I_a$, such that, in $\Gamma$:
$$x\rightarrow \ldots \rightarrow y'_{i_1}, y''_{i_1}\rightarrow y_{i_1+1}\rightarrow \ldots \rightarrow y'_{i_p},y''_{i_p} \rightarrow \ldots \rightarrow z.$$
As $I_a\subseteq I_b$ and $I_b$ is $\Gamma$-convex, all the $y_i$ except the $y_{i_p}$ are elements of $I_b$.
So $y_1,\ldots,y_k \in I_b\sqcup\{a\}\setminus I_a$.

$\Longleftarrow$. Let $x\rightarrow y_1\rightarrow\ldots \rightarrow y_k\rightarrow z$ in $\Gamma$, with $x,z\in I_b$.
For all $y\in V(\Gamma)$, we put:
$$\overline{y}=\begin{cases}
y\mbox{ if }y\notin I_a,\\
a\mbox{ otherwise}.
\end{cases}$$
Then  $\overline{x}\rightarrow \overline{y_1}\rightarrow\ldots \rightarrow \overline{y_k}\rightarrow \overline{z}$ in $\Gamma/I_a\rightarrow a$.
As $I_b\sqcup\{a\}\setminus I_a$ is $\Gamma/I_a\rightarrow a$-convex, $\overline{y_1},\ldots,\overline{y_k}\in I_b\sqcup\{a\}\setminus I_a$,
so $y_1,\ldots,y_k \in I_b$.

Let $x\rightarrow y_1\rightarrow\ldots \rightarrow y_k\rightarrow z$ in $\Gamma_{\mid I_b}$, with $x,z\in I_a$.
Then Let $x\rightarrow y_1\rightarrow\ldots \rightarrow y_k\rightarrow z$ in $\Gamma$; as $I_a$ is $\Gamma$-convex,
$y_1,\ldots,y_k\in I_a$. \end{proof}

\begin{defi}
Let $\Gamma$ be a Feynman graph.
\begin{enumerate}
\item We shall say that $\Gamma$ is \emph{simple} if the two following conditions hold:
\begin{itemize}
\item For all $v,v'\in V(\Gamma)$, there exists at most one internal edge $e$ in $\Gamma$ with $S_\Gamma(e)=v$ and $T_\Gamma(e)=v'$.
\item For all $e\in Int(\Gamma)$, $S_\Gamma(e)\neq T_\Gamma(e)$.
\end{itemize}\end{enumerate}\end{defi}

\begin{lemma} \label{lem65}
Let $\Gamma$ be a Feynman graph, $I\subseteq V(\Gamma)$, non-empty.
\begin{enumerate}
\item If $\Gamma_{\mid I}$ and $\Gamma/I\rightarrow a$ has no oriented cycle, then $\Gamma$ has no oriented cycle and $I$ is $\Gamma$-connex.
\item We assume that $I$ is $\Gamma$-convex. If $\Gamma_{\mid I}$ or $\Gamma/I \rightarrow a$ has an oriented cycle,
then $\Gamma$ has an oriented cycle. 
\end{enumerate}\end{lemma}

\begin{proof} 1. Let $x\rightarrow y_1\rightarrow \ldots \rightarrow y_k\rightarrow z$ in $\Gamma$, with $x,z\in I$.
We assume that at least one of the $y_i$ is not in $I$. Let $i$ be the smallest index such that $y_i \notin I$
and $j$ be the smallest index greater than $i$ such that $y_j \in I$, with the convention $y_{k+1}=z$.
Then $a\rightarrow y_i\rightarrow \ldots \rightarrow y_{j-1} \rightarrow a$ in $\Gamma/I\rightarrow a$, and $\Gamma/I\rightarrow a$
has an oriented cycle: this is a contradiction, so $y_1,\ldots,y_k \in I$ and $I$ is $\Gamma$-convex.

Let us consider an oriented cycle $x_1\rightarrow \ldots \rightarrow x_k\rightarrow x_1$ in $\Gamma$. If one of the $x_i$ belongs to $I$, 
as $I$ is $\Gamma$-convex, all the $x_i$ belongs to $I$, so $\Gamma_{\mid I}$ has an oriented cycle: this is a contradiction.
Il none of the $x_i$ belong to $I$, they form an oriented cycle in $\Gamma/I\rightarrow a$: this is a contradiction.
As a conclusion, $\Gamma$ has no oriented cycle.\\



2. If $\Gamma_{\mid I}$ has an oriented cycle, obviously $\Gamma$ has an oriented cycle.
Let us assume that $\Gamma/I\rightarrow a$ has an oriented cycle. If this cycle does not contain $a$, then obviously
$\Gamma$ has an oriented cycle. If not, there exists an oriented cycle
$a\rightarrow y_1\rightarrow \ldots \rightarrow y_k\rightarrow a$ in $\Gamma/I\rightarrow a$, with $y_1,\ldots,y_k\notin V(\Gamma)\setminus I$.
Note that $k\geq 1$, by definition of $\Gamma/I\rightarrow a$.
Hence, there exists $x,z\in I$, such that $x\rightarrow y_1\rightarrow \ldots \rightarrow y_k\rightarrow z$ in $\Gamma$:
$I$ is not $\Gamma$-convex. \end{proof}\\

\textbf{Remark.} If $I$ is not $\Gamma$-convex,  $\Gamma$ may have no oriented cycle, whereas $\Gamma/I\rightarrow a$ may have one.
Take for example:
\vspace{.5cm}
$$\Gamma=\xymatrix{\rond{1}\ar[r] \ar@/^2pc/[rr]&\rond{2}\ar[r]&\rond{3}},$$
If $I=\{1,3\}$, then:
$$\Gamma/I\rightarrow a=\xymatrix{\rond{a} \ar@/^1pc/[r]&\rond{2}\ar@/^1pc/[l]}.$$

\subsection{The operad of Feynman graphs}

\begin{prop} \label{prop66}
Let $\Gamma \in \fg(A)$, $\Gamma'\in \fg(B)$, and $b\in A$. We put:
\begin{align*}
\Gamma \nabla_b\Gamma'&=\sum_{\substack{\Gamma''\in \fg(A\sqcup B\setminus \{b\}),\\
\Gamma''_{\mid B}=\Gamma',\: \Gamma''/B\rightarrow b=\Gamma}} \Gamma'',&
\Gamma \circ_b\Gamma'&=\sum_{\substack{\Gamma''\in \fg(A\sqcup B\setminus \{b\}),\\
\Gamma''_{\mid B}=\Gamma',\: \Gamma''/B\rightarrow b=\Gamma,\\ \mbox{\scriptsize $B$ $\Gamma''$-convex}}} \Gamma''.
\end{align*}
For all $\Gamma \in \fg(A)$, $\Gamma'\in \fg(B)$, $\Gamma''\in \fg(C)$, if $b\neq c\in A$:
\begin{align*}
(\Gamma \nabla_b\Gamma')\nabla_c\Gamma''&=(\Gamma \nabla_c\Gamma'')\nabla_b\Gamma',&
(\Gamma \circ_b\Gamma')\circ_c\Gamma''&=(\Gamma \circ_c\Gamma'')\circ_b\Gamma'.
\end{align*}
If $b\in B$ and $c\in C$:
\begin{align*}
(\Gamma \nabla_b\Gamma')\nabla_c\Gamma''&=\Gamma \nabla_b(\Gamma'\nabla_c\Gamma''),&
(\Gamma \circ_b\Gamma')\circ_c\Gamma''&=\Gamma \circ_b(\Gamma'\circ_c\Gamma'').
\end{align*}\end{prop}

\begin{proof} If $b\neq c \in A$:
\begin{align*}
(\Gamma \nabla_b\Gamma')\nabla_c\Gamma''&=\sum_{\substack{\Upsilon\in \fg(A\sqcup B\sqcup C\setminus\{b,c\}),\\
\Upsilon_{\mid C}=\Gamma'',\: (\Upsilon/C\rightarrow C)_{\mid B}=\Gamma',\\
(\Upsilon/C\rightarrow c)/B\rightarrow b=\Gamma}} \Upsilon=\sum_{\substack{\Upsilon\in \fg(A\sqcup B\sqcup C\setminus\{b,c\}),\\
\Upsilon_{\mid C}=\Gamma'',\: \Upsilon_{\mid B}=\Gamma',\\
(\Upsilon/B\rightarrow b)/C\rightarrow c=\Gamma}} \Upsilon=(\Gamma \nabla_c\Gamma'')\nabla_b\Gamma';\\ \\
(\Gamma \circ_b\Gamma')\circ_c\Gamma''&=\sum_{\substack{\Upsilon\in \fg(A\sqcup B\sqcup C\setminus\{b,c\}),\\
\Upsilon_{\mid C}=\Gamma'',\: (\Upsilon/C\rightarrow C)_{\mid B}=\Gamma',\\
(\Upsilon/C\rightarrow c)/B\rightarrow b=\Gamma,\\
\mbox{\scriptsize $C$ $\Upsilon$-convex, $B$ $\Upsilon/C\rightarrow c$-convex}}} \Upsilon
=\sum_{\substack{\Upsilon\in \fg(A\sqcup B\sqcup C\setminus\{b,c\}),\\
\Upsilon_{\mid C}=\Gamma'',\: \Upsilon_{\mid B}=\Gamma',\\
(\Upsilon/B\rightarrow b)/C\rightarrow c=\Gamma,\\
\mbox{\scriptsize $B$ $\Upsilon$-convex, $C$ $\Upsilon/B\rightarrow b$-convex}}} \Upsilon=(\Gamma \circ_c\Gamma'')\circ_b\Gamma'.
\end{align*}
We used lemma \ref{lem62}, with $I_a=B$ and $I_b=C$.\\

Let $b\in A$ and $c\in B$.
\begin{align*}
(\Gamma \nabla_b\Gamma')\nabla_c\Gamma''&=\sum_{\substack{\Upsilon \in \fg(A\sqcup B\sqcup C\setminus\{b,c\}),\\
\Upsilon_{\mid C}=\Gamma'',\: (\Upsilon/C\rightarrow c)_{\mid B}=\Gamma'',\\
(\Upsilon/C\rightarrow c)/B\rightarrow b=\Gamma}} \Upsilon
=\sum_{\substack{\Upsilon \in \fg(A\sqcup B\sqcup C\setminus\{b,c\}),\\
(\Upsilon_{\mid B\sqcup C\setminus \{c\}})_{\mid C}=\Gamma'',\\ (\Upsilon_{\mid B\sqcup C\setminus \{c\}})/C\rightarrow c=\Gamma'',\\
\Upsilon/B\sqcup C\setminus \{c\}\rightarrow b=\Gamma}} \Upsilon
=\Gamma \nabla_b(\Gamma'\nabla_c\Gamma'');\\ \\
(\Gamma \circ_b\Gamma')\circ_c\Gamma''&=\sum_{\substack{\Upsilon \in \fg(A\sqcup B\sqcup C\setminus\{b,c\}),\\
\Upsilon_{\mid C}=\Gamma'',\: (\Upsilon/C\rightarrow c)_{\mid B}=\Gamma'',\\
(\Upsilon/C\rightarrow c)/B\rightarrow b=\Gamma,\\
\mbox{\scriptsize $C$ $\Upsilon$-convex, $B$ $\Upsilon_{\mid C}$-convex}}} \Upsilon
=\sum_{\substack{\Upsilon \in \fg(A\sqcup B\sqcup C\setminus\{b,c\}),\\
(\Upsilon_{\mid B\sqcup C\setminus \{c\}})_{\mid C}=\Gamma'',\\ (\Upsilon_{\mid B\sqcup C\setminus \{c\}})/C\rightarrow c=\Gamma'',\\
\Upsilon/B\sqcup C\setminus \{c\}\rightarrow b=\Gamma,\\
\mbox{\scriptsize $B\sqcup C\setminus \{c\}$ $\Upsilon$-convex},\\
\mbox{\scriptsize $C$ $\Upsilon_{\mid B\sqcup C\setminus \{c\}}$-convex}}} \Upsilon
=\Gamma \circ_b(\Gamma'\circ_c\Gamma'').
\end{align*}
 We used lemma \ref{lem63}, with $I_a=C$ and $I_b=B\sqcup C\setminus \{c\}$. \end{proof}\\
 
 Although the associativity of the composition is satisfied, $\bffg$ is not an operad: it contains no unit.
 In order to obtain it, we must take a completion. For all finite set $A$, we put:
$$\overline{\bffg}(A)=\prod_{\Gamma \in \fg(A)} \K\Gamma.$$
It contains $\bffg(A)$. Its elements will be denoted under the form:
$$\sum_{\Gamma \in \fg(A)} a_\Gamma \Gamma.$$

\begin{theo}\label{theo67}
The compositions $\nabla$ and $\circ$ are naturally extended to $\overline{\bffg}$, making it an operad. Its unit is:
$$I=\sum_{\Gamma \in \fg(\{1\})}\Gamma.$$
\end{theo}

\begin{proof}
Let $X=\sum a_\Gamma \Gamma \in \bffg(A)$ and $Y=\sum b_\Gamma \Gamma$ in $\bffg(B)$ and $b \in A$. Then:
\begin{align*}
X\nabla_b Y&=\sum_{\Upsilon \in \fg(A\sqcup B\setminus \{b\})} a_{\Upsilon/ B\rightarrow b} b_{\Upsilon_{\mid B}} \Upsilon,&
X\circ_b Y&=\sum_{\substack{\Upsilon \in \fg(A\sqcup B\setminus \{b\}),\\
\mbox{\scriptsize $B$ $\Upsilon$-convex}}} a_{\Upsilon/ B\rightarrow b} b_{\Upsilon_{\mid B}} \Upsilon.
\end{align*}
These formulas also make sense if $X\in \overline{\bffg}(A)$ and $Y\in \overline{\bffg}(B)$.
Proposition \ref{prop66} implies the associativity of $\nabla$ and $\circ$. \\

Let $\Gamma \in \fg(A)$ and $b\in A$. If $\Gamma' \in \fg(\{b\})$:
$$\Gamma \nabla_b \Gamma'=\begin{cases}
\Gamma\mbox{ if $InExt(\Gamma_{\mid \{b\}})=InExt(\Gamma')$ and $OutExt(\Gamma_{\mid \{b\}})=OutExt(\Gamma')$},\\
0\mbox{ otherwise}.
\end{cases}$$
Summing over all possible $\Gamma'$, $\Gamma \nabla_b I=\Gamma$. Hence, for all $X \in \overline{\bffg}(A)$, $X\nabla_b I=X$.\\

Let $\Gamma \in \fg(\{1\})$ and $\Gamma'\in \fg(A)$. Then:
$$\Gamma \nabla_1\Gamma'=\begin{cases}
\Gamma \mbox{ if $InExt(\Gamma')=InExt(\Gamma)$ and $OutExt(\Gamma)=OutExt(\Gamma')$},\\
0\mbox{ otherwise}.
\end{cases}$$
Summing over all possible $\Gamma$, $I\nabla_1\Gamma'=\Gamma'$. Hence, for all $X\in \overline{\bffg}(A)$, $I\nabla_1 X=X$,
so $I$ is the unit of the operad $(\overline{\bffg},\nabla)$. The proof is similar for $(\overline{\bffg},\circ)$. \end{proof}\\

\textbf{Remark.} The unit of $\overline{\bffg}$ is:
$$I=\sum_{i,j,k\geq 0} \xymatrix{\ar[r]|i&\rond{1}\ar[r]|j \ar@(ul,ur)|k&},$$
where the integers on the edges and half-edges indicate their multiplicity.

\subsection{Suboperads and quotients}

\begin{prop}\begin{enumerate}
\item For all finite space $A$, we denote by $\wcfg(A)$ the set of Feynman graphs $\Gamma \in \fg(A)$ with no oriented cycle.
We also put:
\begin{align*}
\bfwcfg(A)&=\bigoplus_{\Gamma \in \wcfg(A)} \K\Gamma,&
\overline{\bfwcfg}(A)&=\prod_{\Gamma \in \wcfg(A)} \K\Gamma.
\end{align*}
Then $\overline{\bfwcfg}$ is a suboperad of $(\overline{\bffg},\nabla)$ and $(\overline{\bffg},\circ)$. Moreover,
$(\overline{\bfwcfg},\nabla)=(\overline{\bfwcfg},\circ)$.
\item For all finite set $A$, we put:
$$I(A)=\prod_{\Gamma \in \fg(A)\setminus \wcfg(A)} \K\Gamma.$$
Then $I$ is an operadic ideal of $(\overline{\bffg},\circ)$. Moreover, $(\overline{\bffg}/I,\circ)$ is isomorphic
to $(\overline{\bfwcfg},\circ)$. 
\end{enumerate}\end{prop}

\begin{proof} This is a direct consequence of lemma \ref{lem65}. \end{proof}

\begin{defi}
Let $\Gamma \in \fg(A)$. We define an equivalence relation on $Int(\Gamma)$: 
$$\forall e,f\in Int(\Gamma),\: e\sim f\Longleftrightarrow (S_\Gamma(e),T_\Gamma(e))=(S_\Gamma(f),T_\Gamma(f)).$$
We define a Feynman graph $s(\Gamma)$ by:
\begin{align*}
V(s(\Gamma))&=V(\Gamma),&OutExt(s(\Gamma))&=OutExt(\Gamma),\\
Int(s(\Gamma))&=\{e\in Int(\Gamma)\mid S_\Gamma(e)\neq T_\Gamma(e)\}/\sim,&InExt(s(\Gamma))&=InExt(\Gamma).
\end{align*}
For all $\overline{e}\in Int(s(\Gamma))$:
\begin{align*}
S_{s(\Gamma)}(\overline{e})&=S_\Gamma(e),&T_{s(\Gamma)}(\overline{e})&=T_\Gamma(e).
\end{align*}
For all $e\in OutExt(s(\Gamma))$, for all $f\in InExt(s(\Gamma))$:
\begin{align*}
S_{s(\Gamma)}(e)&=S_\Gamma(e),&T_{s(\Gamma)}(f)&=T_\Gamma(f).
\end{align*}
Roughly speaking, $s(\Gamma)$ is obtained by deleting the loops of $\Gamma$ and reducing multiple edges to single edges.
Note that $s(\Gamma)$ is a simple Feynman graph.
\end{defi}

\begin{prop}
For all finite set $A$, we denote by $\sfg(A)$ the set of simple Feynman graphs $\Gamma$ such that $V(\Gamma)=A$. We also put:
\begin{align*}
\bfsfg(A)&=\bigoplus_{\Gamma \in \sfg(A)} \K\Gamma,&
\overline{\bfsfg}(A)&=\prod_{\Gamma \in \sfg(A)} \K\Gamma.
\end{align*}
We define two operadic composition on $\overline{\bfsfg}$: if $\Gamma\in \sfg(A)$, $\Gamma'\in \sfg(B)$ and $b\in A$:
\begin{align*}
\Gamma \nabla_b \Gamma'&=\sum_{\substack{\Gamma''\in \sfg(A\sqcup B\setminus\{b\}),\\
\Gamma''_{\mid B}=\Gamma',\: s(\Gamma''/B\rightarrow b)=\Gamma}}\Gamma'',&
\Gamma \circ_b \Gamma'&=\sum_{\substack{\Gamma''\in \sfg(A\sqcup B\setminus\{b\}),\\
\Gamma''_{\mid B}=\Gamma',\: s(\Gamma''/B\rightarrow b)=\Gamma,\\
\mbox{\scriptsize $B$ $\Gamma''$-convex}}}\Gamma''.
\end{align*}
The following map is an injective operad morphism from $(\overline{\bfsfg},\nabla)$ to $(\overline{\bffg},\nabla)$ and
from $(\overline{\bfsfg},\circ)$ to $(\overline{\bffg},\circ)$:
$$\psi:\left\{\begin{array}{rcl}
\overline{\bfsfg}(A)&\longrightarrow&\overline{\bffg}(A)\\
\Gamma&\longrightarrow&\displaystyle \sum_{\Gamma'\in \fg(A),\:s(\Gamma')=\Gamma}\Gamma'.
\end{array}\right.$$
\end{prop}

\begin{proof} The map $\psi$ is clearly injective. Let $\Gamma\in \sfg(A)$, $\Gamma'\in \sfg(B)$, $b\in A$.
Note that $\psi(\Gamma)\nabla_b \psi(\Gamma')$ is a sum of Feynman graphs with multiplicity $1$, more precisely:
$$\psi(\Gamma)\nabla_b\psi(\Gamma')=\sum_{\substack{\Upsilon \in \fg(A\sqcup B\setminus \{b\}),\\
s(\Upsilon_{\mid B})=\Gamma,\: s(\Upsilon/B\rightarrow b)=\Gamma'}} \Upsilon.$$
Let us assume that $\Upsilon$ appears in $\psi(\Gamma)\nabla_b \psi(\Gamma)$ and that $s(\Upsilon)=s(\Upsilon')$.
Then, obviously, $s(\Upsilon'_{\mid B})=s(\Upsilon_{\mid B})=\Gamma'$, and $g(\Upsilon'/B\rightarrow b)=g(\Upsilon/B\rightarrow b)=\Gamma$,
so $\Upsilon'$ also appears in $\psi(\Gamma)\nabla_b\psi(\Gamma')$. Therefore:
$$\psi(\Gamma)\nabla_b \psi(\Gamma')
=\psi\left(\sum_{\substack{\Upsilon\in \fg(A\sqcup B\setminus\{b\}),\\
s(\Upsilon_{\mid B})=\Gamma',\: \Upsilon/B\rightarrow b=\Gamma}} \Upsilon\right).$$
As $\psi$ is injective, this defines an operadic composition on $\overline{\bfsfg}$, making $\psi$ an operad morphism.
The proof is similar for $\circ$, observing that if $\Gamma,\Gamma'\in \fg(A\sqcup B\setminus \{b\})$ are such that $s(\Gamma)=s(\Gamma')$,
then $B$ is $\Gamma$-convex if, and only if, $B$ is $\Gamma'$-convex. \end{proof}\\

\textbf{Remark.} The unit of $\bfsfg$ is:
$$I=\sum_{i,j\geq 0} \xymatrix{\ar[r]|i&\rond{1} \ar[r]|j&}.$$

Restricting to Feynman graphs with no oriented cycle:

\begin{prop} \begin{enumerate}
\item For all finite set $A$, we denote by $\wcsfg(A)$ the set of simple Feynman graphs $\Gamma$ with no oriented cycle such that $V(\Gamma)=A$. 
We also put:
\begin{align*}
\bfwcsfg(A)&=\bigoplus_{\Gamma \in \wcsfg(A)} \K\Gamma,&
\overline{\bfwcsfg}(A)&=\prod_{\Gamma \in \wcsfg(A)} \K\Gamma.
\end{align*}
Then $\overline{\bfwcsfg}$ is a suboperad of both $(\overline{\bfsfg},\nabla)$ and $(\overline{\bfsfg},\circ)$,
and $\psi(\overline{\bfwcsfg})\subseteq \overline{\bfwcfg}$.
Moreover, $(\overline{\bfwcsfg},\nabla)=(\overline{\bfwcsfg},\circ)$.
\item For all finite set $A$, we put:
$$J(A)=\prod_{\Gamma \in \sfg(A)\setminus \wcsfg(A)} \K\Gamma.$$
Then $J$ is an operadic ideal of $(\overline{\bfsfg},\circ)$. Moreover, $(\overline{\bfsfg}/J,\circ)$ is isomorphic to $(\overline{\bfwcsfg},\circ)$.
\end{enumerate}\end{prop}

We obtain two commutative diagrams of operads. 
\begin{align*}
&\mbox{For $\nabla$:}&&\mbox{For $\circ$:}\\
&\xymatrix{&\overline{\bffg}&\\
\overline{\bfsfg}\ar@{^(->}[ru]&&\overline{\bfwcfg}\ar@{_(->}[lu]\\
&\overline{\bfwcsfg}\ar@{^(->}[ru]\ar@{_(->}[lu]}&
&\xymatrix{&\overline{\bffg}\ar@{->>}[r]&\overline{\bfwcfg}\\
\overline{\bfsfg}\ar@{->>}[r]\ar@{^(->}[ru]&\overline{\bfwcsfg}&\overline{\bfwcfg}\ar@{_(->}[lu]\ar[u]|{Id}\\
&\overline{\bfwcsfg}\ar@{^(->}[ru]\ar@{_(->}[lu]\ar[u]|{Id}&}
\end{align*}

\section{Oriented graphs, posets, finite topologies}

\subsection{Oriented graphs}

We shall identify oriented graphs (possibly with loops and multiple edges) with Feynman graphs with no external edge.
For all finite set $A$, we denote by $\gr(A)$ the set of graphs $G$ such that $V(G)=A$ and we put:
\begin{align*}
\bfgr(A)&=\bigoplus_{G\in \gr(A)}\K G,&\overline{\bfgr}(A)&=\prod_{G\in \gr(A)}\K G.
\end{align*}

\begin{defi}
Let $\Gamma$ be a Feynman graph. The graph $g(\Gamma)$ is defined by:
\begin{itemize}
\item $V(g(\Gamma))=V(\Gamma)$.
\item $Int(g(\Gamma))=Int(\Gamma)$.
\item $S_{g(\Gamma)}=(S_\Gamma)_{\mid Int(\Gamma)}$.
\item $T_{g(\Gamma)}=(T_\Gamma)_{\mid Int(\Gamma)}$.
\end{itemize}
Roughly speaking, one deletes the external edges of $\Gamma$ to obtain $g(\Gamma)$.
\end{defi}

Note that if $G\in \gr(A)$ and $I\subseteq A$, non-empty, then $G/I\rightarrow a$ is also a graph, whereas $G_{\mid I}$ is not always a graph
(external edges may appear).

\begin{theo}
We define two operadic composition on $\overline{\bfgr}$: if $G\in \gr(A)$, $G'\in \gr(B)$ and $b\in A$,
\begin{align*}
G\nabla_a G'&=\sum_{\substack{G''\in \gr(A\sqcup B\setminus\{b\}),\\
g(G''_{\mid B})=G',\: G''/B\rightarrow b=G}} G'',&
G\circ_a G'&=\sum_{\substack{G''\in \gr(A\sqcup B\setminus\{b\}),\\
g(G''_{\mid B})=G',\: G''/B\rightarrow b=G,\\
\mbox{\scriptsize $B$ is $G$-convex}}} G''.
\end{align*}
The unit is:
$$I=\sum_{G\in \gr(\{1\})} G.$$
Moreover, the following map is an injective morphism from $(\overline{\bfgr},\nabla)$ to $(\overline{\bffg},\nabla)$
and from $(\overline{\bfgr},\circ)$ to $(\overline{\bffg},\circ)$:
$$\phi:\left\{\begin{array}{rcl}
\overline{\bfgr(A)}&\longrightarrow&\overline{\bffg(A)}\\
G&\longrightarrow&\displaystyle \sum_{\Gamma \in \fg(A),\: g(\Gamma)=G} \Gamma.
\end{array}\right.$$
\end{theo}

\begin{proof} The map $\phi$ is clearly injective. Let $G\in \gr(A)$, $G'\in \gr(B)$, $b\in A$.
Note that $\phi(G)\nabla_b \phi(G')$ is a sum of Feynman graphs with multiplicity $1$, more precisely:
$$\phi(G)\nabla_b\phi(G')=\sum_{\substack{\Gamma \in \fg(A\sqcup B\setminus \{b\}),\\
g(\Gamma_{\mid B})=G,\: g(\Gamma/B\rightarrow b)=G'}} \Gamma.$$
Let us assume that $\Gamma$ appears in $\phi(G)\nabla_b \phi(G)$ and that $g(\Gamma)=g(\Gamma')$.
Then, obviously, $g(\Gamma'_{\mid B})=g(\Gamma_{\mid B})=G'$, and $g(\Gamma'/B\rightarrow b)=g(\Gamma/B\rightarrow b)=G$,
so $\Gamma'$ also appears in $\phi(G)\nabla_b\phi(G')$.Hence:
$$\phi(G)\nabla_b \phi(G')
=\phi\left(\sum_{\substack{G''\in \gr(A\sqcup B\setminus\{b\}),\\
g(G''_{\mid B})=G',\: G''/B\rightarrow b=G}} G''\right).$$
As $\phi$ is injective, this defines an operadic composition on $\overline{\bfgr}$, making $\phi$ an operad morphism.
The proof is similar for $\circ$, observing that if $\Gamma,\Gamma'\in \fg(A\sqcup B\setminus \{b\})$ are such that $g(\Gamma)=g(\Gamma')$,
then $B$ is $\Gamma$-convex if, and only if, $B$ is $\Gamma'$-convex. \end{proof}\\

\textbf{Remark.} The unit of $\bfgr$ is:
$$I=\sum_{k\geq 0}\xymatrix{\rond{1} \ar@(ul,ur)|k}.$$

\begin{defi} 
Let $A$ be a finite set.
\begin{enumerate}
\item We denote by $\wcgr(A)$ the set of graphs $G$ with no oriented cycle such that $V(G)=A$. We also put:
\begin{align*}
\bfwcgr(A)&=\bigoplus_{G\in \wcgr(A)} \K G,&\overline{\bfwcgr}(A)&=\prod_{G\in \wcgr(A)} \K G.
\end{align*}
\item We denote by $\sgr(A)$ the set of simple graphs $G$ such that $V(G)=A$. We also put:
\begin{align*}
\bfsgr(A)&=\bigoplus_{G\in \sgr(A)} \K G.
\end{align*}
This is a finite-dimensional space, of dimension $2^{|A|(|A|-1)}$.
\item We denote by $\wcsgr(A)$ the set of simple graphs $G$ with no oriented cycle such that $V(G)=A$. We also put:
\begin{align*}
\bfwcsgr(A)&=\bigoplus_{G\in \wcsgr(A)} \K G.
\end{align*}
This is a finite-dimensional space.
\end{enumerate}\end{defi}

By restriction of $\phi$, we obtain:

\begin{cor}
$\overline{\bfwcgr}$, $\bfsgr$ and $\bfwcsgr$ are operads for $\nabla$ and $\circ$.
\end{cor}

\textbf{Remark.} The unit of $\bfsgr$ is $I=\xymatrix{\rond{1}}$.\\

We obtain two commutative diagrams of operads. For $\nabla$:
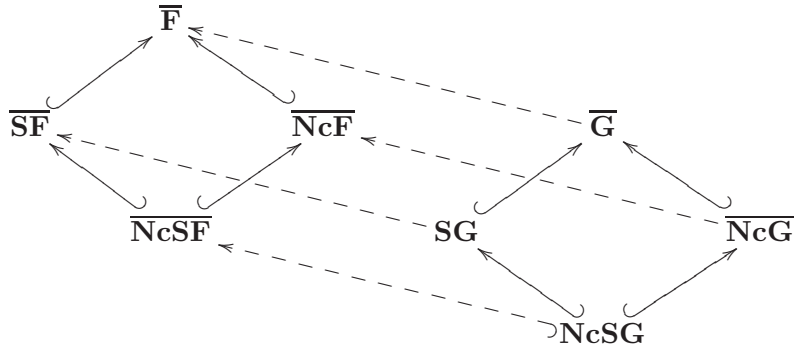
$$\xymatrix{&\overline{\bffg}&&&&\\
\overline{\bfsfg}\ar@{^(->}[ru]&&\overline{\bfwcfg}\ar@{_(->}[lu]&&\overline{\bfgr}\ar@{-->}[lllu]&\\
&\overline{\bfwcsfg}\ar@{^(->}[ru]\ar@{_(->}[lu]&&\bfsgr\ar@{^(->}[ru]\ar@{-->}|(.49)\hole[lllu]&&\overline{\bfwcgr}\ar@{_(->}[lu]\ar@{-->}|(.52)\hole[lllu]\\
&&&&\bfwcsgr\ar@{^(->}[ru]\ar@{_(->}[lu]\ar@{^(-->}[lllu]}$$
For $\circ$:
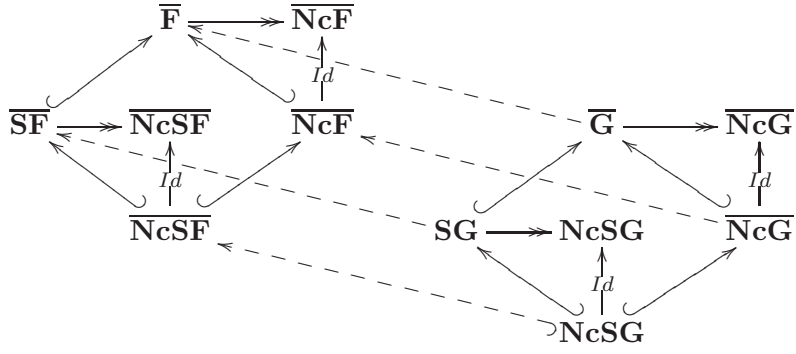
$$\xymatrix{&\overline{\bffg}\ar@{->>}[r]&\overline{\bfwcfg}&&&\\
\overline{\bfsfg}\ar@{->>}[r]\ar@{^(->}[ru]&\overline{\bfwcsfg}&\overline{\bfwcfg}\ar@{_(->}[lu]\ar[u]|{Id}&&\overline{\bfgr}
\ar@{->>}[r]\ar@{-->}|(.635)\hole[lllu]&\overline{\bfwcgr}\\
&\overline{\bfwcsfg}\ar@{^(->}[ru]\ar@{_(->}[lu]\ar[u]|{Id}&&\bfsgr\ar@{->>}[r]\ar@{^(->}[ru]\ar@{-->}|(.49)\hole|(.66)\hole[lllu]&\bfwcsgr
&\overline{\bfwcgr}\ar@{_(->}[lu]\ar[u]|{Id}\ar@{-->}|(.52)\hole[lllu]\\
&&&&\bfwcsgr\ar@{^(->}[ru]\ar@{_(->}[lu]\ar[u]|{Id}\ar@{^(-->}[lllu]}$$

\subsection{Quasi-orders and orders}

A  quasi-order on a set $A$ is a transitive, reflexive relation on $A$. By Alexandroff's theorem, if $A$ is finite, 
there is a bijective correspondence relation between  quasi-orders on $A$ and topologies on $A$ \cite{Alexandroff,Erne}.

\begin{defi}
Let $A$ be a finite set. 
\begin{enumerate}
\item We denote by $\qo(A)$ the set of  quasi-orders on $A$ and we denote by $\bfqo(A)$ the space generated by $\qo(A)$.
\item We denote by $\od(A)$ the set of orders on $A$ and we denote by $\bfod(A)$ the space generated by $\od(A)$. 
\end{enumerate} \end{defi}

The elements of $\qo(A)$ will be denoted under the form $P=(A,\leq_P)$, where $\leq_P$ the considered  quasi-order defined on the set $A$.

\begin{defi}
Let $\Gamma \in \fg(A)$. We define a quasi-order on $A$ by:
$$\forall x,y \in A, x\leq_\Gamma y\mbox{ if there exists an oriented path from $x$ to $y$ in $\Gamma$}.$$
\end{defi}

\textbf{Remark.} The  quasi-order $\leq_\Gamma$ is an order if, and only if, $\Gamma$ has no oriented cycle.

\begin{defi}
Let $P\in \qo(A)$ and $I\subseteq A$, non-empty.
\begin{enumerate}
\item We shall say that $I$ is $P$-convex if, and only if, for all $x,y,z\in A$:
$$x,z\in I\mbox{ and }x\leq_P y \leq_P z\Longrightarrow y\in I.$$
\item The  quasi-order $\leq_{P/I \rightarrow a}$ is defined on $A\sqcup \{a\}\setminus I$: if $x,y\in A\setminus I$,
\begin{itemize}
\item $x\leq_{P/I\rightarrow a} y$ if $x\leq_P y$ or if there exists $x',y'\in I$, $x\leq_P y'$ and $x\leq_P y'$.
\item $x\leq_{P/I\rightarrow a} a$ if there exists $y'\in I$ such that $x\leq_P y'$.
\item $a\leq_{P/I\rightarrow a} y$ if there exists $x'\in I$, such that $x'\leq_P y$.
\end{itemize} \end{enumerate}\end{defi}

\textbf{Remark.} If $\Gamma \in \fg(A)$ and $I\subseteq A$, then $I$ is $\Gamma$-convex if, and only if, $I$ is $\leq_\Gamma$-convex.
Moreover, $\leq_{\Gamma/I\rightarrow a}=(\leq_\Gamma)_{/I\rightarrow a}$.\\

The following operad is described in \cite{ManchonFoissyFauvet2}:

\begin{theo}
We define an operadic composition $\circ$ on $\bfqo$: if $P\in \qo(A)$, $Q\in \qo(B)$ and $b\in A$,
\begin{align*}
P \circ_a Q&=\sum_{\substack{R\in \qo(A\sqcup B\setminus\{b\}),\\
R_{\mid B}=Q,\: R/I\rightarrow a=P,\\
\mbox{\scriptsize $B$ $R$-convex}}}R.
\end{align*} 
The following map is an injective operad morphism from $(\bfqo,\circ)$ to $(\bfsgr,\circ)$:
$$\kappa:\left\{\begin{array}{rcl}
\bfqo(A)&\longrightarrow&\bfsgr(A)\\
P&\longrightarrow&\displaystyle \sum_{G\in \sgr(A),\: \leq_G=P} G.
\end{array}\right.$$
\end{theo}

\begin{proof} Let $P\in \qo(A)$, $Q\in \qo(B)$ and $b\in B$. Then:
$$\kappa(P)\circ_b \kappa(Q)=\sum_{\substack{G \in \bfsgr(A\sqcup B\setminus\{b\}),\\
\leq_{G_{\mid B}}=Q,\: \leq_{G/\rightarrow b}=P,\\
\mbox{\scriptsize $B$ $G$-convex}}}G.$$
Let us assume that $G$ appears in $\kappa(P)\circ_b \kappa(Q)$ and that $\leq_{G'}=\leq_{G}$.
Then, as $B$ is $G$-convex, it is $\leq_G$-convex, hence $\leq_{G'}$-convex, hence $G'$-convex.
Moreover:
\begin{align*}
\leq_{G'_{\mid B}}&=(\leq_{G'})_{\mid B}=(\leq_G)_{\mid B}=\leq_{G_{\mid B}}=P,\\
\leq_{G'/B\rightarrow b}&=(\leq_{G'})_{/B\rightarrow b}=(\leq_G)_{/B\rightarrow b}=\leq_{G/B\rightarrow b}=Q,
\end{align*}
so $G'$ appears in $\kappa(G)\circ_b\kappa(G')$. Therefore:
\begin{align*}
\kappa(P) \circ_a \kappa(Q)&=\kappa\left(\sum_{\substack{R\in \qo(A\sqcup B\setminus\{b\}),\\
R_{\mid B}=Q,\: R/I\rightarrow a=P,\\
\mbox{\scriptsize $B$ $R$-convex}}}R\right).
\end{align*} 
Moreover, $\kappa$ is injective: indeed, if $P\in \qo(A)$, the arrow graph $G_P$ of this  quasi-order satisfies $\leq_{G_P}=\leq_P$.
Therefore, this defines an operadic composition on  quasi-orders. \end{proof}\\

The following operad $\bfod$ is described in \cite{ManchonFoissyFauvet}:

\begin{cor}\label{cor80}
$\bfod$ is a suboperad of $(\bfqo,\circ)$. Moreover, if, for any finite set $A$, we denote by $J(A)$ the space generated by the elements of
$\qo(A) \setminus \od(A)$, then $J$ is an ideal of $(\bfqo,\circ)$ and the quotient $\bfqo/I$ is isomorphic to $\bfod$.
\end{cor}

\begin{proof} This is implied by $\bfod=\kappa^{-1}(\kappa(\bfqo)\cap \bfwcsgr)$. \end{proof}\\

\textbf{Examples.}  In $\bfod$:
\begin{align*}
\tddeux{$1$}{$2$}\circ_1 \tddeux{$1$}{$2$}&=\tdtroisdeux{$1$}{$2$}{$3$}+\tdtroisun{$1$}{$3$}{$2$}&
\tddeux{$1$}{$2$}\circ_2 \tddeux{$1$}{$2$}&=\tdtroisdeux{$1$}{$2$}{$3$}+\pdtroisun{$3$}{$1$}{$2$}\\
\tddeux{$1$}{$2$}\circ_1 \tddeux{$2$}{$1$}&=\tdtroisdeux{$2$}{$1$}{$3$}+\tdtroisun{$2$}{$3$}{$1$}&
\tddeux{$1$}{$2$}\circ_2 \tddeux{$2$}{$1$}&=\tdtroisdeux{$1$}{$3$}{$2$}+\pdtroisun{$2$}{$1$}{$3$}\\
\tddeux{$1$}{$2$}\circ_1 \tdun{$1$}\tdun{$2$}&=\pdtroisun{$3$}{$1$}{$2$}+\tddeux{$1$}{$3$}\tdun{$2$}+\tdun{$1$}\tddeux{$2$}{$3$}&
\tddeux{$1$}{$2$}\circ_2 \tdun{$1$}\tdun{$2$}&=\tdtroisun{$1$}{$3$}{$2$}+\tddeux{$1$}{$2$}\tdun{$3$}+\tddeux{$1$}{$3$}\tdun{$2$}\\
\tddeux{$1$}{$2$}\circ_1\tdun{$1,2$}\hspace{.25cm}&=\tddeux{$1,2$}{$3$}\hspace{.25cm}&
\tddeux{$1$}{$2$}\circ_2\tdun{$1,2$}\hspace{.25cm}&=\tddeux{$1$}{$2,3$}\hspace{.25cm}\\
\\
\tddeux{$2$}{$1$}\circ_1 \tddeux{$1$}{$2$}&=\pdtroisun{$2$}{$1$}{$3$}+\tdtroisdeux{$3$}{$1$}{$2$}&
\tddeux{$2$}{$1$}\circ_2 \tddeux{$1$}{$2$}&=\tdtroisdeux{$2$}{$3$}{$1$}+\tdtroisun{$2$}{$3$}{$1$}\\
\tddeux{$2$}{$1$}\circ_1 \tddeux{$2$}{$1$}&=\pdtroisun{$1$}{$2$}{$3$}+\tdtroisdeux{$3$}{$2$}{$1$}&
\tddeux{$2$}{$1$}\circ_2 \tddeux{$2$}{$1$}&=\tdtroisdeux{$3$}{$2$}{$1$}+\tdtroisun{$3$}{$2$}{$1$}\\
\tddeux{$2$}{$1$}\circ_1 \tdun{$1$}\tdun{$2$}&=\tdtroisun{$2$}{$3$}{$1$}+\tddeux{$2$}{$1$}\tdun{$3$}+\tdun{$1$}\tddeux{$2$}{$3$}&
\tddeux{$2$}{$1$}\circ_2 \tdun{$1$}\tdun{$2$}&=\pdtroisun{$1$}{$2$}{$3$}+\tddeux{$2$}{$1$}\tdun{$3$}+\tddeux{$3$}{$1$}\tdun{$2$}\\
\tddeux{$2$}{$1$}\circ_1\tdun{$1,2$}\hspace{.25cm}&=\tddeux{$3$}{$1,2$}\hspace{.25cm}&
\tddeux{$2$}{$1$}\circ_2\tdun{$1,2$}\hspace{.25cm}&=\tddeux{$2,3$}{$1$}\hspace{.25cm}\\
\\
\tdun{$1$}\tdun{$2$}\circ_1 \tddeux{$1$}{$2$}&=\tddeux{$1$}{$2$}\tdun{$3$}&
\tdun{$1$}\tdun{$2$}\circ_2 \tddeux{$1$}{$2$}&=\tdun{$1$}\tddeux{$2$}{$3$}\\
\tdun{$1$}\tdun{$2$}\circ_1 \tddeux{$1$}{$2$}&=\tddeux{$2$}{$1$}\tdun{$3$}&
\tdun{$1$}\tdun{$2$}\circ_2 \tddeux{$1$}{$2$}&=\tdun{$1$}\tddeux{$3$}{$2$}\\
\tdun{$1$}\tdun{$2$}\circ_1 \tdun{$1$}\tdun{$2$}&=\tdun{$1$}\tdun{$2$} \tdun{$3$}&
\tdun{$1$}\tdun{$2$}\circ_2 \tdun{$1$}\tdun{$2$}&=\tdun{$1$}\tdun{$2$} \tdun{$3$}\\
\tdun{$1$}\tdun{$2$}\circ_1\tdun{$1,2$}\hspace{.25cm}&=\tdun{$1,2$}\hspace{.25cm}\tdun{$3$}&
\tdun{$1$}\tdun{$2$}\circ_2\tdun{$1,2$}\hspace{.25cm}&=\tdun{$1$}\tdun{$2,3$}\hspace{.25cm}\\
\\
\tdun{$1,2$}\hspace{.25cm}\circ_1 \tddeux{$1$}{$2$}&=0&
\tdun{$1,2$}\hspace{.25cm}\circ_2 \tddeux{$1$}{$2$}&=0\\
\tdun{$1,2$}\hspace{.25cm}\circ_1 \tddeux{$1$}{$2$}&=0&
\tdun{$1,2$}\hspace{.25cm}\circ_2 \tddeux{$1$}{$2$}&=0\\
\tdun{$1,2$}\hspace{.25cm}\circ_1 \tdun{$1,2$}\hspace{.25cm}&=0&
\tdun{$1,2$}\hspace{.25cm}\circ_2 \tdun{$1,2$}\hspace{.25cm}&=0\\
\tdun{$1,2$}\hspace{.25cm}\circ_1\tdun{$1,2$}\hspace{.25cm}&=0&
\tdun{$1,2$}\hspace{.25cm}\circ_2\tdun{$1,2$}\hspace{.25cm}&=0
\end{align*}

We obtain a diagram of operads (for $\circ$):
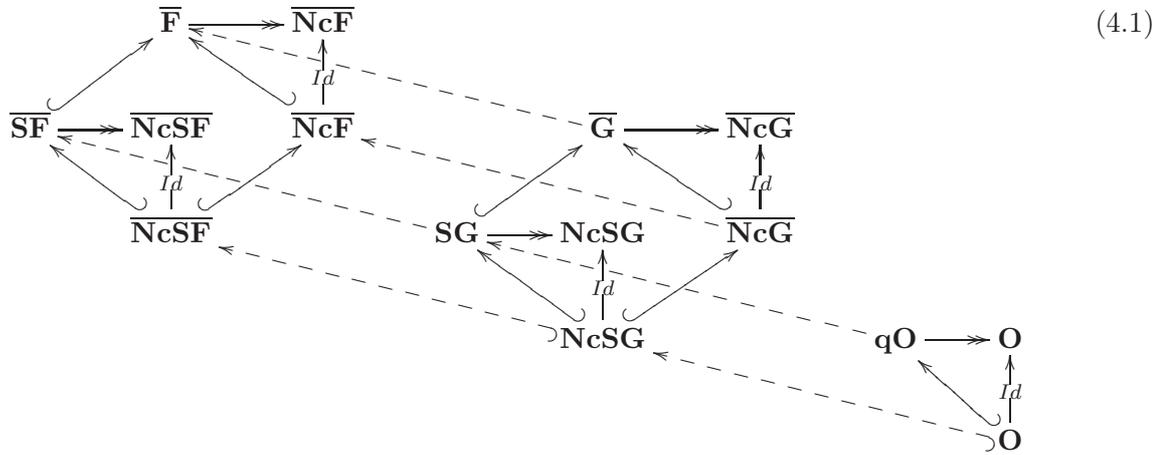
\begin{align} \label{EQ4.1}
\xymatrix{&\overline{\bffg}\ar@{->>}[r]&\overline{\bfwcfg}&&&&&\\
\overline{\bfsfg}\ar@{->>}[r]\ar@{^(->}[ru]&\overline{\bfwcsfg}&\overline{\bfwcfg}\ar@{_(->}[lu]\ar[u]|{Id}&&\overline{\bfgr}
\ar@{->>}[r]\ar@{-->}|(.635)\hole[lllu]&\overline{\bfwcgr}&&\\
&\overline{\bfwcsfg}\ar@{^(->}[ru]\ar@{_(->}[lu]\ar[u]|{Id}&&\bfsgr\ar@{->>}[r]\ar@{^(->}[ru]\ar@{-->}|(.49)\hole|(.66)\hole[lllu]&\bfwcsgr
&\overline{\bfwcgr}\ar@{_(->}[lu]\ar[u]|{Id}\ar@{-->}|(.52)\hole[lllu]&&\\
&&&&\bfwcsgr\ar@{^(->}[ru]\ar@{_(->}[lu]\ar[u]|{Id}\ar@{^(-->}[lllu]&&\bfqo\ar@{->>}[r]\ar@{-->}|(.497)\hole|(.67)\hole[lllu]&\bfod\\
&&&&&&&\bfod\ar@{_(->}[lu]\ar[u]|{Id}\ar@{^(-->}[lllu]}
\end{align}
By composition, we obtain morphism from $\bfqo$ to several other operads; for example:
\begin{align*}
&\left\{\begin{array}{rcl}
\bfqo(A)&\longrightarrow&\overline{\bfgr}(A)\\
P&\longrightarrow&\displaystyle \sum_{G\in \gr(A),\: \leq_G=P} G,
\end{array}\right.&&\left\{\begin{array}{rcl}
\bfqo(A)&\longrightarrow&\overline{\bfsfg}(A)\\
P&\longrightarrow&\displaystyle \sum_{G\in \sfg(A),\: \leq_G=P} G,
\end{array}\right.\\ \\
&\left\{\begin{array}{rcl}
\bfqo(A)&\longrightarrow&\overline{\bffg}(A)\\
P&\longrightarrow&\displaystyle \sum_{G\in \fg(A),\: \leq_G=P} G.
\end{array}\right.
\end{align*}

\textbf{Remark}. The image of $\kappa$ is not a suboperad of $(\bfsgr,\nabla)$.
For example, let us take $P$ and $Q$ be the  quasi-orders associated to the following graphs:
\begin{align*}
&\xymatrix{\rond{1} \ar@/^/[r]&\rond{b}\ar@/^/[l]};&&\xymatrix{\rond{2}&\rond{3}}.
\end{align*}
Then $G$ appears in $\kappa(P)\nabla_b \kappa(Q)$, and $G'$ not:
\begin{align*}
G&=\xymatrix{\rond{2}\ar[r]&\rond{1} \ar@/^/[r]&\rond{3}\ar@/^/[l]},&
G'&=\xymatrix{\rond{2}\ar[r]\ar@/^2pc/[rr]&\rond{1} \ar@/^/[r]&\rond{3}\ar@/^/[l]},
\end{align*}
although $\leq_G=\leq_{G'}$.

\section{$\petitbinfini$ structures}

\subsection{Operad morphisms and associated products}

\begin{theo}
There exists a unique operad morphism:
$$\left\{\begin{array}{rcl}
\ascom&\longrightarrow&\bfod\\
m&\longrightarrow&\tdun{$1$}\tdun{$2$},\\
\star&\longrightarrow&\tdun{$1$}\tdun{$2$}+\tddeux{$1$}{$2$}.
\end{array}\right.$$
\end{theo}

\begin{proof} In $\bfod$:
\begin{align*}
\tdun{$1$}\tdun{$2$}^{(12)}&=\tdun{$1$}\tdun{$2$},\\
\tdun{$1$}\tdun{$2$}\circ_1\tdun{$1$}\tdun{$2$}&=\tdun{$1$}\tdun{$2$}\circ_2\tdun{$1$}\tdun{$2$}=\tdun{$1$}\tdun{$2$}\tdun{$3$},\\
(\tddeux{$1$}{$2$}+\tdun{$1$}\tdun{$2$})\circ_1(\tddeux{$1$}{$2$}+\tdun{$1$}\tdun{$2$})
&=(\tddeux{$1$}{$2$}+\tdun{$1$}\tdun{$2$})\circ_2(\tddeux{$1$}{$2$}+\tdun{$1$}\tdun{$2$})\\
&=\tdtroisdeux{$1$}{$2$}{$3$}+\tdtroisun{$1$}{$3$}{$2$}+\pdtroisun{$3$}{$1$}{$2$}+\tdun{$1$}\tddeux{$2$}{$3$}
+\tddeux{$1$}{$2$}\tdun{$3$}+\tddeux{$1$}{$3$}\tdun{$2$}+\tdun{$1$}\tdun{$2$}\tdun{$3$}.
\end{align*}
So these elements define a morphism from $\ascom$ to $\bfqo$. \end{proof}

By composition, we obtain morphisms from $\ascom$ to any operad of the commutative diagram (\ref{EQ4.1}). 
We always denote by $m$ and $\star$ the image of the two products of $\ascom$ in all these operads.

\begin{defi} Let $A$ a finite set and $I\subseteq A$.
\begin{enumerate}
\item Let $\Gamma \in \fg(A)$. We shall say that $I$ is an \emph{ideal} of $\Gamma$ if for all $x,y\in A$:
$$x\leq_\Gamma y\mbox{ and }x\in I\Longrightarrow y\in I.$$
\item Let $\leq\in \qo(A)$. We shall say that $I$ is an \emph{ideal} of $\leq$ if  for all $x,y\in A$:
$$x\leq y\mbox{ and }x\in I\Longrightarrow y\in I.$$
\end{enumerate}\end{defi}

\begin{prop}
Let $A$ and $B$ be disjoint finite sets. If $\mathcal{P}\in \{\fg,\sfg,\gr,\sgr,\qo\}$, for all $\Gamma \in \mathcal{P}(A)$, $\Gamma' \in \mathcal{P}(B)$:
\begin{align*}
m\circ(\Gamma,\Gamma')&=\Gamma\Gamma',&
\star\circ(\Gamma,\Gamma')&=\sum_{\substack{\Upsilon \in \mathcal{P}(A\sqcup B),\\
\Upsilon_{\mid A}=\Gamma,\: \Upsilon_{\mid B}=\Gamma',\\
\mbox{\scriptsize $B$ ideal of $\Upsilon$}}}\Upsilon.
\end{align*} \end{prop}

\begin{proof} We prove it for $\mathcal{P}=\fg$; the other cases are proves similarly 
In $\overline{\bffg}$, $m$ is the sum over all Feynman graphs $\Gamma$ on $\{1,2\}$, with no internal edge between
$1$ and $2$ and no internal edge between $2$ and $1$.
Then $m\circ_1 \Gamma$ is the sum over all Feynman graphs $\Upsilon$ over $A\sqcup\{2\}$,
with $\Upsilon_{\mid A}=\Gamma$ and no edge between any vertex of $A$ and $2$ and no edge between $2$ and any vertex of $A$.
In other words, $m\circ_1 \Gamma$ is the sum over all Feynman graphs $\Upsilon=\Gamma \Upsilon'$, where $\Upsilon'$ is a Feynman
graph on $\{2\}$. Consequently, $m\circ (\Gamma,\Gamma')=(m\circ_1\Gamma)\circ_2\Gamma'$ is the sum over all Feynman graphs
$\Upsilon$ on $A\sqcup B$, with $\Upsilon_{\mid A}=\Gamma$ and $\Upsilon_{\mid B}=\Gamma'$, and no internal edge
between $A$ and $B$ an no internal edge between $B$ and $A$, that is to say $\Upsilon=\Gamma\Gamma'$.
The proof is similar for $\star$: $\star\circ (\Gamma,\Gamma')$ is the sum over all Feynman graphs $\Upsilon$ on $A\sqcup B$,
such that $\Upsilon_{\mid A}=\Gamma$ and $\Upsilon_{\mid B}=\Gamma'$, and no internal edge between $B$ and $A$.\end{proof}

\subsection{Associated Hopf algebras}

\begin{cor}
The vector space $\bfcfg$ generated by connected Feynman graph is a suboperad of $(\bffg,\nabla)$ and $(\bffg,\circ)$.
\end{cor}

\begin{proof} Let $\Gamma\in \fg(A)$ and $I\subseteq A$, non-empty. Let us prove that if 
$\Gamma_{\mid I}$ and $\Gamma/I\rightarrow a$ are connected, then  $\Gamma$ is connected.
For all $x\in A$, let us denote by $CC(x)$ the connected component of $x$ in $\Gamma$.
As $\Gamma_{\mid I}$ is connected, for all $x\in I$, $I\subseteq CC(x)$.
Let $x\notin I$. As $\Gamma/I\rightarrow a$ is connected, there exists an non oriented path from $x$ to $a$ in $\Gamma/I\rightarrow a$,
so there exists a non-oriented path from $x$ to a vertex $y\in I$ in $\Gamma$: we obtain that $I\subseteq CC(y)=CC(x)$ for all $x\notin I$.
So $\Gamma$ has only one connected component. \end{proof}\\

Following corollary \ref{cor80}, it is possible to define suboperad of connected objects for all the operads in the commutative diagram.
As the product $m$ is, in all cases, the disjoint union, the morphism from $\petitbinfini$ to any of these operads obtained 
by restriction of the morphism from $\ascom$ takes its image in the suboperad of connected objects. For example:
\begin{align*}
\theta_\bfod(\lfloor-,-\rfloor_{1,1})&=\tddeux{$1$}{$2$},&
\theta_\bfod(\lfloor-,-\rfloor_{1,2})&=\tdtroisun{$1$}{$3$}{$2$},\\
\theta_\bfod(\lfloor-,-\rfloor_{2,1})&=\pdtroisun{$1$}{$2$}{$3$},&
\theta_\bfod(\lfloor-,-\rfloor_{2,2})&=_1^3\hspace{-1.5mm}\pquatresept_2^4+_1^3\hspace{-1.5mm}\pquatresix_2^4
+_1^4\hspace{-1.5mm}\pquatresix_2^3+_1^3\hspace{-1.5mm}\pquatrecinq_2^4+_1^4\hspace{-1.5mm}\pquatrecinq_2^3.
\end{align*}
For all $k,l\geq 1$, $\theta_\bfod(\lfloor-,-\rfloor_{k,l})$ is the sum of all connected bipartite graphs with blocks $\{1,\ldots,k\}$
and $\{k+1,\ldots,k+l\}$. The number of such graphs is given by sequence A227322 of the OEIS.\\

For any vector space $V$, for any $(\bfP,\mathbf{CP})$ in the following set:
$$\left\{\begin{array}{c}
(\bffg,\bfcfg),(\bfwcfg,\bfcwcfg),(\bfsfg,\bfcsfg),(\bfwcsfg,\bfcwcsfg),(\bfgr,\bfcgr),\\
(\bfwcgr,\bfcwcgr),(\bfsgr,\bfcsgr),(\bfwcsgr,\bfcwcsgr),(\bfqo,\bfcqo),(\bfod,\bfcod)
\end{array}\right\},$$
we have:
$$F_\bfP(V)=S(F_{\mathbf{Cp}}(V)).$$
Let us describe the product $\star$ induced by the $\petitbinfini$ structure on all these Hopf algebras. 
We restrict ourselves to $F_\bffg(V)$, the other cases are similar. We fix $V=Vect(X_1,\ldots,X_N)$.
\begin{itemize}
\item As a vector space, $A_\bfcfg(V)=F_\bffg(V)$ is generated by isoclasses $\widehat{\Gamma}$ of Feynman graphs $\Gamma$
whose vertices are decorated by elements of $[N]$. 
\item For any Feynman graph $\Gamma$ whose vertices are decorated by $[N]$:
$$\Delta(\widehat{\Gamma})=\sum_{\Gamma=\Gamma_1\Gamma_2} \widehat{\Gamma_1}\otimes \widehat{\Gamma_2}.$$
\item For any Feynman graphs $\Gamma \in \fg(A)$, $\Gamma'\in \fg(B)$, whose vertices are decorated by $[N]$:
$$\widehat{\Gamma}*\widehat{\Gamma'}=\sum_{\substack{\Gamma'' \in \fg(A\sqcup B),\\
\Gamma''_{\mid A}=\Gamma,\;\Gamma''_{\mid B}=\Gamma',\\
\mbox{\scriptsize $B$ ideal of $\Gamma''$}}} \widehat{\Gamma''}.$$
\end{itemize}
The dual Hopf algebra $A^*_\bfcfg(V)$ has the same basis, up to an identification. Moreover:
\begin{itemize}
\item For any Feynman graph $\Gamma\in \fg(A)$, whose vertices are decorated by $[N]$:
$$\Delta_*(\widehat{\Gamma})=\sum_{\mbox{\scriptsize $I$ ideal of $\Gamma$}}
\widehat{\Gamma_{\mid V(\Gamma)\setminus I}}\otimes \widehat{\Gamma_{\mid I}}.$$
\item For any Feynman graphs $\Gamma,\Gamma'$, whose vertices are decorated by $[N]$:
$$\widehat{\Gamma}.\widehat{\Gamma'}=\widehat{\Gamma\Gamma'}.$$
\end{itemize}

By functoriality, we obtain a diagram of Hopf algebra morphisms:
$$
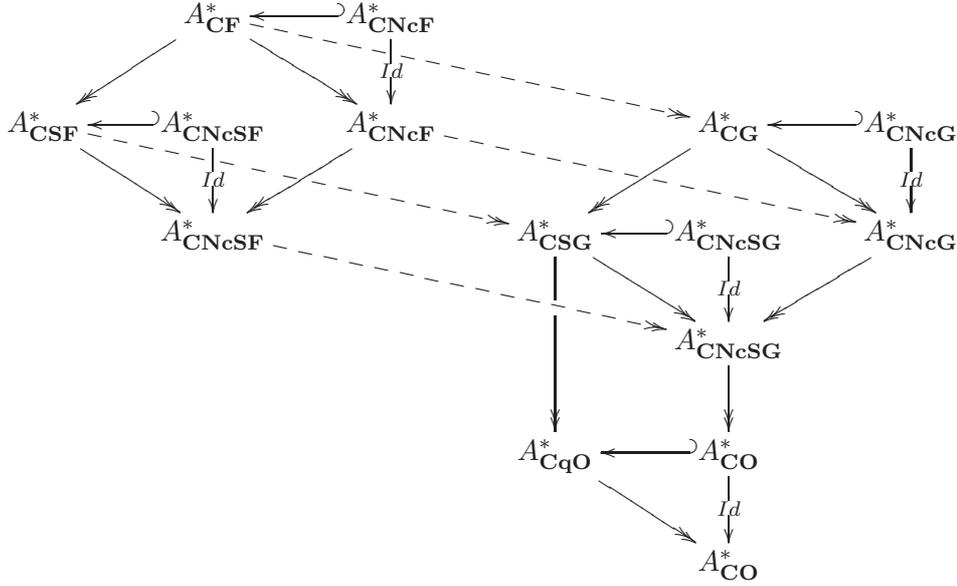
\xymatrix{&A^*_\bfcfg\ar@{->>}[ld]\ar@{->>}[rd]\ar@{-->>}|(.35)\hole[rrrd]&A^*_\bfcwcfg\ar|{Id}[d]\ar@{_(->}[l]&&&\\
A^*_\bfcsfg\ar@{->>}[rd]\ar@{-->>}|(.33)\hole|(.51)\hole[rrrd]&A^*_\bfcwcsfg\ar|{Id}[d]\ar@{_(->}[l]
&A^*_\bfcwcfg\ar@{->>}[ld]\ar@{-->>}|(.49)\hole[rrrd]&&A^*_\bfcgr\ar@{->>}[ld]\ar@{->>}[rd]&A^*_\bfcwcgr\ar|{Id}[d]\ar@{_(->}[l]\\
&A^*_\bfcwcsfg\ar@{-->>}[rrrd]&&A^*_\bfcsgr\ar@{->>}[rd]\ar@{->>}|(.34)\hole[dd]&A^*_\bfcwcsgr\ar|{Id}[d]\ar@{_(->}[l]&A^*_\bfcwcgr\ar@{->>}[ld]\\
&&&&A^*_\bfcwcsgr\ar@{->>}[d]&&\\
&&&A^*_\bfcqo\ar@{->>}[rd]&A^*_\bfcod\ar|{Id}[d]\ar@{_(->}[l]&\\
&&&&A^*_\bfcod&}$$

Here are examples of morphisms in this diagram:
\begin{align*}
&\left\{\begin{array}{rcl}
A^*_\bfcfg(V)&\longrightarrow&A^*_\bfcsfg(V)\\
\widehat{\Gamma}&\longrightarrow&\widehat{s(\Gamma)},
\end{array}\right.&
&\left\{\begin{array}{rcl}
A^*_\bfcfg(V)&\longrightarrow&A^*_\bfcwcfg(V)\\
\widehat{\Gamma}&\longrightarrow&
\begin{cases}
\widehat{\Gamma} \mbox{ if $\Gamma$ has no oriented cycle},\\
0\mbox{ otherwise},\end{cases}
\end{array}\right.\\ \\
&\left\{\begin{array}{rcl}
A^*_\bfcfg(V)&\longrightarrow&A^*_\bfcgr(V)\\
\widehat{\Gamma}&\longrightarrow&\widehat{g(\Gamma)},
\end{array}\right.&
&\left\{\begin{array}{rcl}
A^*_\bfcgr(V)&\longrightarrow&A^*_\bfcqo(V)\\
\widehat{G}&\longrightarrow&\widehat{\leq_G}.
\end{array}\right.&
\end{align*}

Let us describe the bialgebra $D^*_\bffg(V)$. It is generated by pairs $(\widehat{\Gamma},d)$, where $\Gamma$ is a connected 
Feynman graph decorated by $[N]$, and $d\in [N]$. We shall need the following notions: 

\begin{defi}
Let $\Gamma$ be a Feynman graph, with $V(\Gamma)=A$, and let $\{A_1,\ldots,A_k\}$ be a partition of $A$.
\begin{enumerate}
\item We shall say that this partition is $\Gamma$-admissible if:
\begin{itemize}
\item For all $i$, $\Gamma_{\mid A_i}$ is connected.
\item \begin{itemize}
\item $A_1$ is $\Gamma$-convex.
\item $A_2$ is $(\Gamma/A_1\rightarrow 1)$-convex.
\item $\vdots$
\item $A_k$ is $((\ldots(\Gamma/A_1\rightarrow 1)/\ldots)/A_{k-1}\rightarrow (k-1))$-convex.
\end{itemize}
By lemma \ref{lem62}-3, this does not depend on the choice of the order on $A_1,\ldots,A_k$.
\end{itemize}
\item Let us assume that $\{A_1,\ldots,A_k\}$ is $\Gamma$-admissible. Let $D:[k]\longrightarrow [N]$.
\begin{enumerate}
\item We obtain a Feynman graph on $[k]$:
$$\Gamma/\{A_1,\ldots,A_k\}=(\ldots(\Gamma/A_1\rightarrow 1)/\ldots)/A_k\rightarrow k.$$
Its isoclasse does not depend on the order chosen on the partition $A_1\sqcup \ldots \sqcup A_k$ by lemma \ref{lem62}.
If $\Gamma$ is connected, then $\Gamma/\{A_1,\ldots,A_k\}$ is connected. Moreover, $(\Gamma/\{A_1,\ldots,A_k\},D)$ 
is a Feynman graph decorated by $[N]$.
\item $\Gamma_{\mid A_1}\ldots \Gamma_{\mid A_k}$ is a decorated Feynman graph, with $k$ connected components, namely $A_1,\ldots,A_k$.
\end{enumerate}\end{enumerate}\end{defi}

For all connected Feynman graphs $\Gamma$ and $d\in [N]$, in $D^*_\bffg(V)$:
$$\Delta_*((\widehat{\Gamma},d))=\sum_{\substack{\mbox{\scriptsize $\{A_1,\ldots,A_k\}$ $\Gamma$-admissible},\\ D:[k]\longrightarrow N}}
(\widehat{(\Gamma/\{A_1,\ldots,A_k\},D)},d) \otimes (\widehat{(\Gamma_{\mid A_1}},D(1))\ldots (\widehat{\Gamma_{\mid A_k}},D(k)).$$
For any Feynman graph decorated by $[N]$:
 $$\rho(\widehat{\Gamma})=\sum_{\substack{\mbox{\scriptsize $\{A_1,\ldots,A_k\}$ $\Gamma$-admissible},\\ D:[k]\longrightarrow N}}
(\widehat{\Gamma/\{A_1,\ldots,A_k\},D)}\otimes (\widehat{\Gamma_{\mid A_1}},D(1))\ldots(\widehat{\Gamma_{\mid A_k}},D(k)).$$

Similar formulas can be given for the other operads. For example, if $a,b,c,d\in [N]$, in $D^*_\bfqo(V)$:
\begin{align*}
\Delta_*((\tdun{$a$},d))&=\sum_{p=1}^N (\tdun{$p$},d)\otimes (\tdun{$a$},p),\\
\Delta_*((\tddeux{$a$}{$b$},d))&=\sum_{p,q=1}^N (\tddeux{$p$}{$q$},d)\otimes (\tdun{$a$},p)(\tdun{$b$},q)
+\sum_{p=1}^N (\tdun{$p$},d) \otimes (\tddeux{$a$}{$b$},p),\\
\Delta_*((\tdtroisun{$a$}{$c$}{$b$},d))&=\sum_{p,q,r=1}^N (\tdtroisun{$p$}{$r$}{$q$},d)\otimes (\tdun{$a$},p)(\tdun{$b$},q)(\tdun{$c$},r)
+\sum_{p,q=1}^N (\tddeux{$p$}{$q$},d) \otimes (\tddeux{$a$}{$b$},p)(\tdun{$c$},q)\\
&+\sum_{p,q=1}^N (\tddeux{$p$}{$q$},d) \otimes (\tddeux{$a$}{$c$},p)(\tdun{$b$},q)
+\sum_{p=1}^N (\tdun{$p$},d)\otimes (\tdtroisun{$a$}{$c$}{$b$},p),\\
\Delta_*((\pdtroisun{$a$}{$b$}{$c$},d))&=\sum_{p,q,r=1}^N (\pdtroisun{$p$}{$q$}{$r$},d)\otimes (\tdun{$a$},p)(\tdun{$b$},q)(\tdun{$c$},r)
+\sum_{p,q=1}^N (\tddeux{$p$}{$q$},d) \otimes (\tdun{$c$},p)(\tddeux{$b$}{$a$},q)\\
&+\sum_{p,q=1}^N (\tddeux{$p$}{$q$},d) \otimes (\tdun{$b$},p) (\tddeux{$c$}{$a$},q)
+\sum_{p=1}^N (\tdun{$p$},d)\otimes (\pdtroisun{$a$}{$b$}{$c$},p),\\
\Delta_*((\tdtroisdeux{$a$}{$b$}{$c$},d))&=\sum_{p,q,r=1}^N (\tdtroisdeux{$p$}{$q$}{$r$},d)\otimes (\tdun{$a$},p)(\tdun{$b$},q)(\tdun{$c$},r)
+\sum_{p,q=1}^N (\tddeux{$p$}{$q$},d) \otimes (\tddeux{$a$}{$b$},p)(\tdun{$c$},q)\\
&+\sum_{p,q=1}^N (\tddeux{$p$}{$q$},d) \otimes(\tdun{$a$},q) (\tddeux{$b$}{$c$},p)
+\sum_{p=1}^N (\tdun{$p$},d)\otimes (\tdtroisdeux{$a$}{$b$}{$c$},p).
\end{align*}
Note that the subalgebra of $D^*_\bfcqo(V)$ generated by rooted trees is a subbialgebra. This comes from the sujective operad morphism
from $\bfcqo$ to $\prelie$, sending any rooted tree to itself and the other quasi-order to $0$.\\

For the coaction:
\begin{align*}
\rho(\tdun{$a$})&=\sum_{p=1}^N \tdun{$p$}\otimes (\tdun{$a$},p),\\
\rho(\tddeux{$a$}{$b$})&=\sum_{p,q=1}^N \tddeux{$p$}{$q$}\otimes (\tdun{$a$},p)(\tdun{$b$},q)
+\sum_{p=1}^N \tdun{$p$} \otimes (\tddeux{$a$}{$b$},p),\\
\rho(\tdtroisun{$a$}{$c$}{$b$})&=\sum_{p,q,r=1}^N \tdtroisun{$p$}{$r$}{$q$}\otimes (\tdun{$a$},p)(\tdun{$b$},q)(\tdun{$c$},r)
+\sum_{p,q=1}^N \tddeux{$p$}{$q$} \otimes (\tddeux{$a$}{$b$},p)(\tdun{$c$},q)\\
&+\sum_{p,q=1}^N \tddeux{$p$}{$q$} \otimes (\tddeux{$a$}{$c$},p)(\tdun{$b$},q)
+\sum_{p=1}^N \tdun{$p$}\otimes (\tdtroisun{$a$}{$c$}{$b$},p),\\
\rho(\pdtroisun{$a$}{$b$}{$c$})&=\sum_{p,q,r=1}^N \pdtroisun{$p$}{$q$}{$r$}\otimes (\tdun{$a$},p)(\tdun{$b$},q)(\tdun{$c$},r)
+\sum_{p,q=1}^N \tddeux{$p$}{$q$} \otimes (\tdun{$c$},p)(\tddeux{$b$}{$a$},q)\\
&+\sum_{p,q=1}^N \tddeux{$p$}{$q$} \otimes (\tdun{$b$},p) (\tddeux{$c$}{$a$},q)
+\sum_{p=1}^N \tdun{$p$}\otimes (\pdtroisun{$a$}{$b$}{$c$},p),\\
\rho(\tdtroisdeux{$a$}{$b$}{$c$})&=\sum_{p,q,r=1}^N \tdtroisdeux{$p$}{$q$}{$r$}\otimes (\tdun{$a$},p)(\tdun{$b$},q)(\tdun{$c$},r)
+\sum_{p,q=1}^N \tddeux{$p$}{$q$} \otimes (\tddeux{$a$}{$b$},p)(\tdun{$c$},q)\\
&+\sum_{p,q=1}^N \tddeux{$p$}{$q$} \otimes(\tdun{$a$},q) (\tddeux{$b$}{$c$},p)
+\sum_{p=1}^N \tdun{$p$}\otimes (\tdtroisdeux{$a$}{$b$}{$c$},p).
\end{align*}

In order to obtain the coproduct of $B^*_\bfqo(V)$, let us quotient by relations $(\tdun{$i$},j)=\delta_{i,j}$.
In $B^*_\bfcqo(V)$:
\begin{align*}
\Delta_*((\tddeux{$a$}{$b$},d))&= (\tddeux{$a$}{$b$},d)\otimes 1+1\otimes (\tddeux{$a$}{$b$},d),\\
\Delta_*((\tdtroisun{$a$}{$c$}{$b$},d))&=(\tdtroisun{$a$}{$b$}{$c$},d)\otimes 1
+\sum_{p=1}^N (\tddeux{$p$}{$c$},d) \otimes (\tddeux{$a$}{$b$},p)+\sum_{p=1}^N (\tddeux{$p$}{$b$},d) \otimes (\tddeux{$a$}{$c$},p)
+1\otimes (\tdtroisun{$a$}{$c$}{$b$},d),\\
\Delta_*((\pdtroisun{$a$}{$b$}{$c$},d))&=(\pdtroisun{$a$}{$b$}{$c$},d)\otimes 1+\sum_{q=1}^N (\tddeux{$c$}{$q$},d) \otimes (\tddeux{$b$}{$a$},q)
+\sum_{q=1}^N (\tddeux{$b$}{$q$},d) \otimes (\tddeux{$c$}{$a$},q)+1\otimes (\pdtroisun{$a$}{$b$}{$c$},d),\\
\Delta_*((\tdtroisdeux{$a$}{$b$}{$c$},d))&=(\tdtroisdeux{$a$}{$b$}{$c$},d)\otimes 1+\sum_{p=1}^N (\tddeux{$p$}{$c$},d) \otimes (\tddeux{$a$}{$b$},p)
+\sum_{q=1}^N (\tddeux{$a$}{$q$},d) \otimes (\tddeux{$b$}{$c$},p)+1\otimes (\tdtroisdeux{$a$}{$b$}{$c$},d).
\end{align*}

We obtain a commutative diagram of Hopf algebra morphisms:
$$
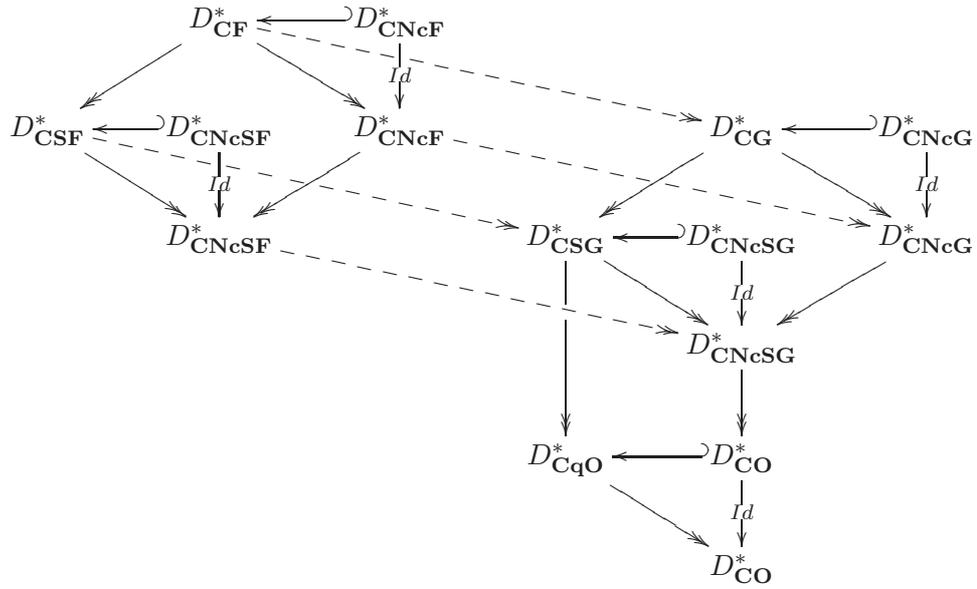
\xymatrix{&D^*_\bfcfg\ar@{->>}[ld]\ar@{->>}[rd]\ar@{-->>}|(.35)\hole[rrrd]&D^*_\bfcwcfg\ar|{Id}[d]\ar@{_(->}[l]&&&\\
D^*_\bfcsfg\ar@{->>}[rd]\ar@{-->>}|(.33)\hole|(.51)\hole[rrrd]&D^*_\bfcwcsfg\ar|{Id}[d]\ar@{_(->}[l]
&D^*_\bfcwcfg\ar@{->>}[ld]\ar@{-->>}|(.49)\hole[rrrd]&&D^*_\bfcgr\ar@{->>}[ld]\ar@{->>}[rd]&D^*_\bfcwcgr\ar|{Id}[d]\ar@{_(->}[l]\\
&D^*_\bfcwcsfg\ar@{-->>}[rrrd]&&D^*_\bfcsgr\ar@{->>}[rd]\ar@{->>}|(.34)\hole[dd]&D^*_\bfcwcsgr\ar|{Id}[d]\ar@{_(->}[l]&D^*_\bfcwcgr\ar@{->>}[ld]\\
&&&&D^*_\bfcwcsgr\ar@{->>}[d]&&\\
&&&D^*_\bfcqo\ar@{->>}[rd]&D^*_\bfcod\ar|{Id}[d]\ar@{_(->}[l]&\\
&&&&D^*_\bfcod&}$$

\textbf{Remarks}. The components of $\bfcfg$ are not finite-dimensional, in order to obtain $D^*_\bfcfg(V)$, we use
the duality between $\bfcfg$ and $\bfcfg$ such that for every Feynman graphs $\Gamma,\Gamma'$:
$$\ll \Gamma,\Gamma'\gg=\delta_{\Gamma,\Gamma'}.$$
The formulas given in the finite-dimensional case also make sense for this duality.
The same remark holds for other operads here appearing, such as $\overline{\bfcgr}$.

\chapter{Summary}

 Let $\bfP$ be an operad, such that $\bfP(0)=(0)$ and, for all $n\geq 1$, $\bfP(n)$ is finite dimensional. 
\begin{enumerate}
\item \begin{enumerate}
\item $\bfP$ is a graded, non connected brace algebra, with a bracket denoted by $\langle-,-\rangle$.
Moreover, $\bfP_+$ is a graded and connected brace subalgebra of $\bfP$.
\item This induces a graded, non connected pre-Lie algebra structure on $\bfP$, which pre-Lie product is denoted by $\bullet$. 
The following diagram of pre-Lie algebras is commutative:
$$
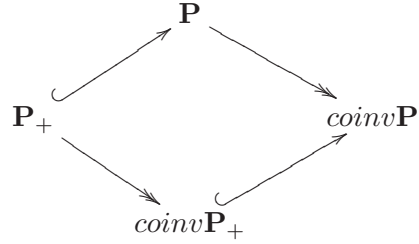
\xymatrix{&\bfP\ar@{->>}[rd]&\\
\bfP_+\ar@{->>}[rd] \ar@{^(->}[ru]&&coinv\bfP\\
&coinv\bfP_+ \ar@{^(->}[ru]&}$$
\item This induces a monoid product $\lozenge$ on $\overline{\bfP}$ and a group product $\lozenge$ on $\overline{\bfP_+}$.
The following diagram of monoids is commutative:
$$
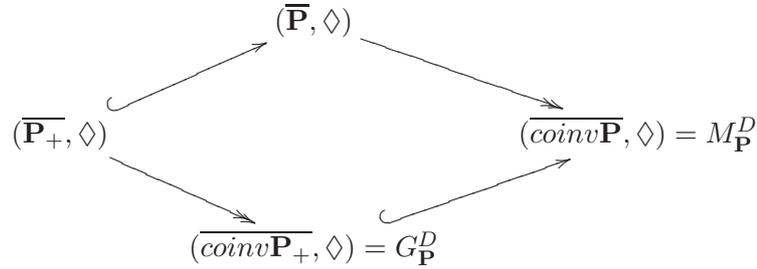
\xymatrix{&(\overline{\bfP},\lozenge)\ar@{->>}[rd]&\\
(\overline{\bfP_+},\lozenge)\ar@{->>}[rd] \ar@{^(->}[ru]&&(\overline{coinv\bfP},\lozenge)=M_\bfP^D\\
&(\overline{coinv\bfP_+},\lozenge)=G_\bfP^D \ar@{^(->}[ru]&}$$
\end{enumerate}
\item  \begin{enumerate}
\item There exist products $*$, induced by the operadic composition of $\bfP$, 
making the following diagram of graded bialgebras commutative:
$$
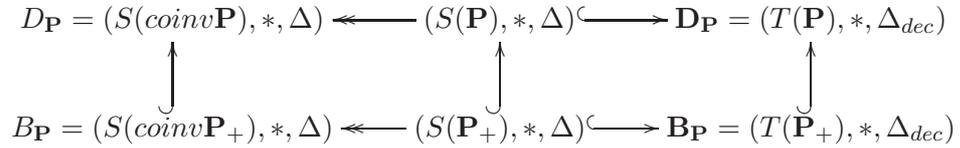
\xymatrix{D_\bfP=(S(coinv\bfP),*,\Delta)&(S(\bfP),*,\Delta) \ar@{^(->}[r] \ar@{->>}[l]&\bfD_\bfP=(T(\bfP),*,\Delta_{dec})\\
B_\bfP=(S(coinv\bfP_+),*,\Delta)\ar@{^(->}[u] &(S(\bfP_+),*,\Delta)\ar@{^(->}[u] \ar@{^(->}[r] \ar@{->>}[l] 
&\bfB_\bfP=(T(\bfP_+),*,\Delta_{dec})\ar@{^(->}[u]}$$
\item They are all graded; the three bialgebras on the bottow row are graded Hopf algebras.
\end{enumerate}
\item \begin{enumerate}
\item There exist coproducts $\Delta_*$ making the following diagram of graded bialgebras commutative:
$$
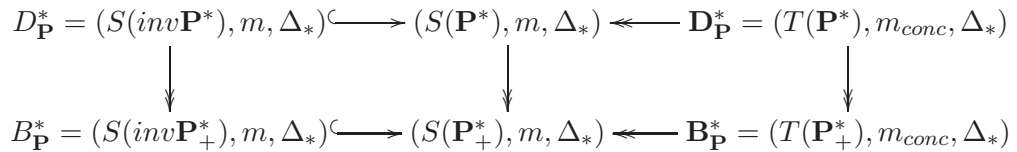
\xymatrix{D_\bfP^*=(S(inv\bfP^*),m,\Delta_*)\ar@{->>}[d]\ar@{^(->}[r]&(S(\bfP^*),m,\Delta_*)\ar@{->>}[d]
&\bfD_\bfP^*=(T(\bfP^*),m_{conc},\Delta_*)\ar@{->>}[d]\ar@{->>}[l]\\
B^*_\bfP=(S(inv\bfP^*_+),m,\Delta_*)\ar@{^(->}[r]&(S(\bfP^*_+),m,\Delta_*)&\bfB^*_\bfP=(T(\bfP^*_+),m_{conc},\Delta_*)\ar@{->>}[l]}$$
Moreover, $\bfB^*_\bfP$, $(S(\bfP^*_+),m,\Delta_*)$ and $B^*_\bfP$ are graded, connected Hopf algebras, dual to
$\bfB_\bfP$, $(S(\bfP_+),*,\Delta)$ and $B_\bfP$  respectively. 
\item Considering the monoids of characters of these bialgebras, we obtain a commutative diagram of monoids:
$$\xymatrix{M^D_\bfP&(\overline{\bfP},\lozenge)\ar@{=}[r] \ar@{->>}[l]&M^\bfD_\bfP\\
G^B_\bfP\ar@{^(->}[u] &(\overline{\bfP_+},\lozenge)\ar@{^(->}[u] \ar@{=}[r] \ar@{->>}[l] 
&G^\bfB_\bfP\ar@{^(->}[u]}$$
\end{enumerate}
\item Let us consider an operad morphism $\theta_\bfP:\petitbinfini\longrightarrow\bfP$. Let $V$ be a finite-dimensional vector space.
We denote by $\bfC_V$ the operad of morphisms from $V$ to $V^{\otimes n}$.
\begin{enumerate}
\item We put:
\begin{align*}
\bfB_\bfP(V)&=\bfB_{\bfP\otimes \bfC_V},&\bfB_\bfP^*(V)&=\bfB^*_{\bfP\otimes \bfC_V},&
\bfD_\bfP(V)&=\bfD_{\bfP\otimes \bfC_V},&\bfD_\bfP^*(V)&=\bfD^*_{\bfP\otimes \bfC_V},\\
B_\bfP(V)&=B_{\bfP\otimes \bfC_V},&B_\bfP^*(V)&=B^*_{\bfP\otimes \bfC_V},&
D_\bfP(V)&=D_{\bfP\otimes \bfC_V},&D_\bfP^*(V)&=D^*_{\bfP\otimes \bfC_V}.
\end{align*}
\item The morphism $\theta_\bfP$ induces a product $\star$ on $S(F_\bfP(V))$, making 
$(S(F_\bfP(V)),\star,\Delta)$ a graded, connected Hopf algebra, denoted by $A_\bfP(V)$. Its graded dual is denoted by $A^*_\bfP(V)$.
\item $A_\bfP(V)$ is a Hopf algebra in the category of $D_\bfP(V)$-modules; dually, $A_\bfP^*(V)$ is a Hopf algebra in the category of  
$D_\bfP^*(V)$-comodules.
\item The monoid $(M^D_\bfP(V),\lozenge)$ of characters of $D^*_\bfP(V)$ acts by endomorphisms on the group $G^A_\bfP(V)$ of characters
of $A_\bfP^*(V)$, and is isomorphic to the monoid of continuous endomorphisms of the $\bfP$-algebra $\overline{F_\bfP(V)}$.
The group of characters $(G^B_\bfP(V),\lozenge)$ of $B^*_\bfP(V)$ acts by group automorphisms on $G^A_\bfP(V)$,
and is isomorphic to the group of formal diffeomorphisms of $\overline{F_\bfP}(V)$ which are tangent to the identity. 
\end{enumerate}\end{enumerate}

Here are diagrams of the different functors which appear in this text (contravariant functors are represented by dashed arrows).
$$
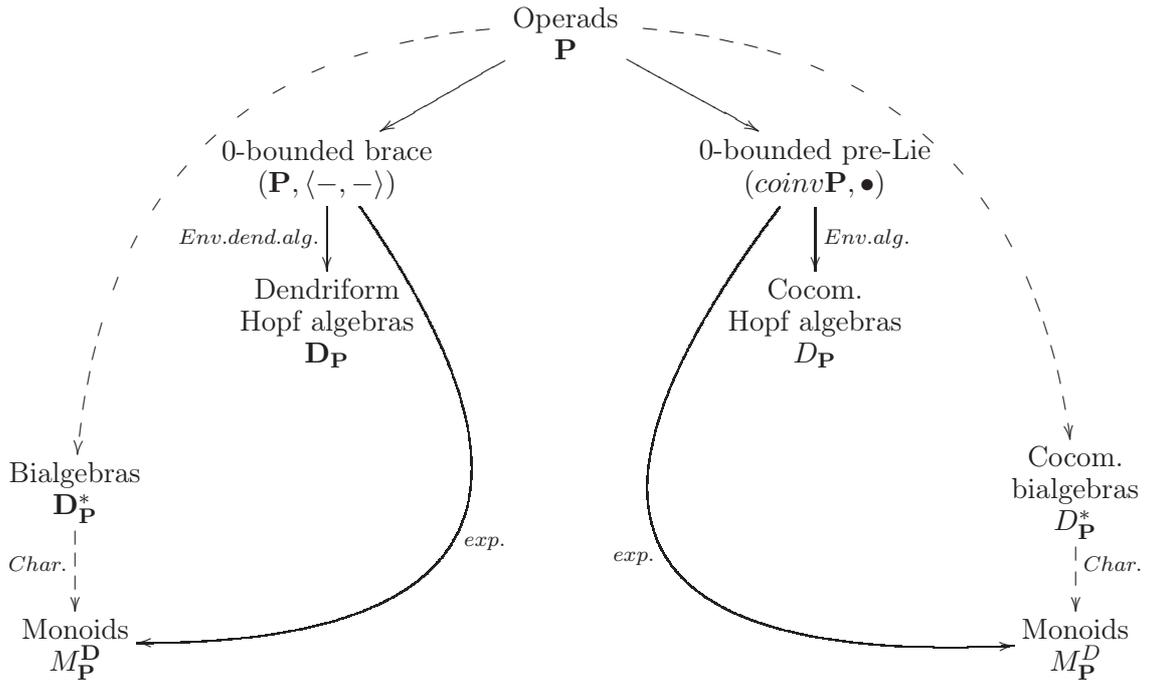
\xymatrix{&&\txt{Operads\\ $\bfP$}\ar[rd] \ar[ld] \ar@{-->}@/^5pc/[dddrr]\ar@{-->}@/_5pc/[dddll]&&\\
&\txt{$0$-bounded brace\\ $(\bfP,\langle-,-\rangle)$}\ar[d]_{Env. dend. alg.}\ar@(dr,dr)@/^9pc/[lddd]^{exp.}&&
\txt{$0$-bounded pre-Lie\\ $(coinv\bfP,\bullet)$}\ar[d]^{Env. alg.}\ar@(dl,dl)@/_10pc/[rddd]_{exp.}&\\
&\txt{Dendriform\\Hopf algebras\\ $\bfD_\bfP$}&&\txt{Cocom.\\ Hopf algebras\\ $D_\bfP$}&\\
\txt{Bialgebras\\ $\bfD_\bfP^*$}\ar@{-->}[d]_{Char.}&&&&\txt{Cocom.\\bialgebras\\ $D_\bfP^*$}\ar@{-->}[d]^{Char.}\\
\txt{Monoids\\ $M_\bfP^\bfD$}&&&&\txt{Monoids\\ $M_\bfP^D$}
}$$
Replacing $\bfP$ by its augmentation ideal $\bfP_+$:
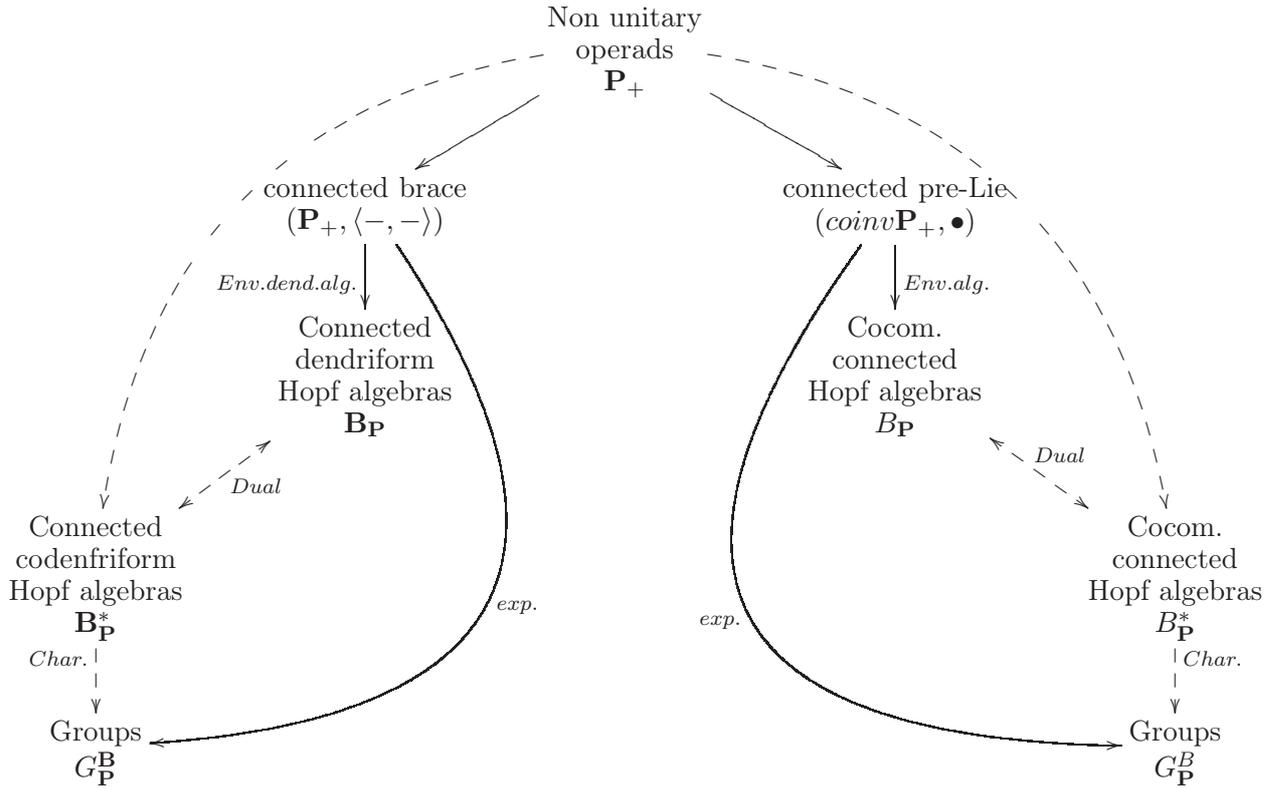
$$\xymatrix{&&\txt{Non unitary\\operads\\ $\bfP_+$}\ar[rd] \ar[ld] \ar@{-->}@/^5pc/[dddrr]\ar@{-->}@/_5pc/[dddll]&&\\
&\txt{connected brace\\ $(\bfP_+,\langle-,-\rangle)$}\ar[d]_{Env. dend. alg.}\ar@(dr,dr)@/^9pc/[lddd]^{exp.}&&
\txt{connected pre-Lie\\ $(coinv\bfP_+,\bullet)$}\ar[d]^{Env. alg.}\ar@(dl,dl)@/_10pc/[rddd]_{exp.}&\\
&\txt{Connected\\dendriform\\Hopf algebras\\ $\bfB_\bfP$}\ar@{<-->}[ld]^{Dual}&&\txt{Cocom.\\connected\\ Hopf algebras\\ $B_\bfP$}\ar@{<-->}[rd]^{Dual}&\\
\txt{Connected\\codenfriform\\Hopf algebras\\ $\bfB_\bfP^*$}\ar@{-->}[d]_{Char.}&&&&
\txt{Cocom.\\connected\\Hopf algebras\\ $B_\bfP^*$}\ar@{-->}[d]^{Char.}\\
\txt{Groups\\ $G_\bfP^\bfB$}&&&&\txt{Groups\\ $G_\bfP^B$}
}$$
For operads with morphisms form $\petitbinfini$, which we here call $\petitbinfini$-operads:
$$\xymatrix{&\txt{$\petitbinfini$-operads\\ $(\bfP,\theta_\bfP)$}\ar@{-->}[rd]\ar[ld]&\\
\txt{Interacting\\bialgebras\\ $(A_\bfP,D_\bfP)$}&&\txt{Cointeracting\\bialgebras\\ $(A_\bfP^*,D_\bfP^*)$}
}$$
Replacing $\bfP$ by $\bfP_+$:
$$\xymatrix{&\txt{Non unitary\\ $\petitbinfini$-operads\\ $(\bfP_+,\theta_\bfP)$}\ar@{-->}[rd]\ar[ld]&\\
\txt{Interacting\\connected\\Hopf algebras\\ $(A_\bfP,B_\bfP)$}\ar@{<-->}[rr]_{Dual}&&\txt{Cointeracting\\connected\\Hopf algebras\\ $(A_\bfP^*,B_\bfP^*)$}
}$$

\bibliographystyle{amsplain}
\bibliography{biblio}
\end{document}